\documentclass[a4paper]{amsart}
\usepackage{amssymb, enumerate, faktor}
\usepackage{bbm}
\usepackage{enumitem}
\usepackage{stmaryrd}
\usepackage{tikz}
\usepackage[all]{xy}

\usepackage{hyperref, aliascnt}
\usepackage{mathtools}
\usepackage{comment}
\def\today{\number\day\space\ifcase\month\or   January\or February\or
   March\or April\or May\or June\or   July\or August\or September\or
   October\or November\or December\fi\   \number\year}

\theoremstyle{definition}
\newtheorem{lma}{Lemma}[section]

\newaliascnt{thmCt}{lma}
\newtheorem{thm}[thmCt]{Theorem}
\aliascntresetthe{thmCt}

\newtheorem*{thm*}{Theorem}

\newaliascnt{corCt}{lma}
\newtheorem{cor}[corCt]{Corollary}
\aliascntresetthe{corCt}

\newaliascnt{propCt}{lma}
\newtheorem{prop}[propCt]{Proposition}
\aliascntresetthe{propCt}

\newaliascnt{pgrCt}{lma}

\aliascntresetthe{pgrCt}

\newaliascnt{dfCt}{lma}
\newtheorem{df}[dfCt]{Definition}
\aliascntresetthe{dfCt}

\newaliascnt{remCt}{lma}
\newtheorem{rem}[remCt]{Remark}
\aliascntresetthe{remCt}

\newaliascnt{remsCt}{lma}

\aliascntresetthe{remsCt}

\newaliascnt{egCt}{lma}
\newtheorem{eg}[egCt]{Example}
\aliascntresetthe{egCt}

\newaliascnt{egsCt}{lma}

\aliascntresetthe{egsCt}

\newaliascnt{qstCt}{lma}

\aliascntresetthe{qstCt}

\newaliascnt{pbmCt}{lma}

\aliascntresetthe{pbmCt}

\newaliascnt{notaCt}{lma}
\newtheorem{nota}[notaCt]{Notation}
\aliascntresetthe{notaCt}

\newaliascnt{convCt}{lma}

\aliascntresetthe{convCt}

\newaliascnt{cnjCt}{lma}

\aliascntresetthe{cnjCt}

\newcounter{theoremintro}
\newtheorem{thmintro}[theoremintro]{Theorem}

\newtheorem{corintro}[theoremintro]{Corollary}

\newcommand{\ph}{\varphi}

\newcommand{\Ph}{\Phi}

\newcommand{\Q}{{\mathbb{Q}}}

\newcommand{\Z}{{\mathbb{Z}}}
\newcommand{\R}{{\mathbb{R}}}
\newcommand{\C}{{\mathbb{C}}}
\newcommand{\N}{{\mathbb{N}}}
\newcommand{\T}{{\mathbb{T}}}

\newcommand{\ind}{{\mathbbm{1}}}
\newcommand{\B}{{\mathcal{B}}}

\newcommand{\Ot}{{\mathcal{O}_2}}

\newcommand{\im}{{\mathrm{im}}}

\newcommand{\supp}{{\mathrm{supp}}}

\newcommand{\andeqn}{\,\,\,\,\,\, {\mbox{and}} \,\,\,\,\,\,}

\newcommand{\ca}{C*-algebra}
\newcommand{\cas}{C*-algebras}

\title[]{Embeddings of $L^p$-operator algebras}

\date{\today}
\thanks{The first named author was supported by a starting grant of 
the Swedish Research Council. \indent The second named author was supported by Kungl. Vetenskapsakademiens stiftelser.
}

\author{Eusebio Gardella}
\author{Jan Gundelach}
\address[Eusebio Gardella and Jan Gundelach]
{Department of Mathematical Sciences, Chalmers University of
Technology and University of Gothenburg, Gothenburg 412 96, Sweden.}
\email{gardella@chalmers.se; jangund@chalmers.se}
\urladdr{www.math.chalmers.se/~gardella}
\urladdr{www.chalmers.se/en/persons/jangund/}

\begin{document}

\begin{abstract}
We study embeddings of $L^p$-operator algebras arising from (twis\-ted) \'etale groupoids, with particular emphasis on rigidity phenomena for $p\neq 2$. Our methods rely on a detailed analysis of core normalizers and their functorial behavior under algebra homomorphisms. Using the notion of actors between groupoids, we show that under natural hypotheses, embeddings between reduced $L^p$-groupoid algebras can be described entirely in terms of morphisms of the underlying groupoids.
We further show that embeddings of $L^p$-groupoid algebras induce embeddings of the associated topological full groups. 
Our results provide new tools for studying embeddability questions in the $L^p$-setting, and are particularly helpful when ruling out the existence of embeddings.\par

As applications, we obtain strong embeddability results both for spatial AF $L^p$-operator algebras and for tensor products of $L^p$-Cuntz algebras. For $p\not \in \{1,2\}$, a reduced $L^p$-groupoid algebra associated with a principal \'etale groupoid embeds into a spatial AF $L^p$-operator algebra if and only if the underlying groupoid is AF. In particular, and in contrast with classical results of Pimsner-Voiculescu, irrational $L^p$-noncommutative tori do not embed into spatial AF $L^p$-operator algebras for $p\neq 2$. Furthermore, if $p\neq 2$, there is no unital contractive homomorphism from $\mathcal{O}_2^p \otimes_p \mathcal{O}_2^p$ into $\mathcal{O}_2^p$, showing that there is no $L^p$-analog of Kirchberg's $\mathcal{O}_2$-embedding theorem.\end{abstract}
\maketitle
\tableofcontents

\section{Introduction}

The study of Banach algebras arising from analytic, geometric, and dynamical constructions has a long and fruitful history. 
Early on in the development of the theory, questions concerning the structure and embeddability of subalgebras of operator algebras have played a central role in the subject. For example, Arveson showed in \cite{Arv_subalgebras_1969} that the study of (not necessarily self-adjoint) subalgebras of C*-algebras already leads to a rich and robust theory.
Within this broad landscape, operator algebras on Hilbert spaces have played a central role, leading to deep classification results and a highly developed structural theory. At the same time, many naturally occurring Banach algebras are not C*-algebras and do not even admit contractive representations on Hilbert spaces, so that their analysis requires techniques that go beyond this setting.

In recent years, there has been growing interest in the systematic study of \emph{$L^p$-operator algebras}, that is, Banach algebras admitting isometric representations on $L^p$-spaces for $1 \leq p < \infty$. These algebras are somewhere in between classical Banach algebras and C*-algebras: They retain enough analytic rigidity to support a meaningful structure theory, while at the same time exhibiting genuinely new phenomena that cannot be expected in the case of operator algebras on Hilbert spaces. Early work in this direction includes $L^p$-analogs of group algebras, crossed products, as well as $L^p$-versions of some very well-studied C*-algebras such as the Cuntz algebras, AF-algebras, and irrational rotation algebras. From a broader perspective, the study of $L^p$-operator algebras fits into a long-standing program concerned with embeddings and representations of non-selfadjoint operator algebras, going back at least to Arveson’s work. In this sense, $L^p$-operator algebras provide a natural testing ground for understanding how far techniques from C*-algebra theory extend beyond Hilbert spaces.

A particularly important class of examples is provided by $L^p$-operator algebras associated with \'etale groupoids. Groupoids offer a flexible framework that simultaneously encodes topology, dynamics, and geometry, and they have proved to be a powerful tool in the theory of C*-algebras. In the $L^p$-context, groupoid algebras provide a large and tractable supply of examples that generalize both group $L^p$-operator algebras and algebras arising from dynamical systems. Moreover, many constructions that are familiar from the C*-theory admit meaningful analogs in this setting.
In particular, the presence of a canonical diagonal subalgebra and its normalizers allows one to associate a Weyl groupoid to any $L^p$-operator algebra (see \cite{ChoGarThi_rigidity_2021}), extending Renault's reconstruction theory in the setting of Cartan subalgebras. These extensions provide a bridge between the associated $L^p$-operator algebras and the underlying dynamics of the groupoid, which in the case $p\neq 2$ tend to be more rigid than what can be expected for C*-algebras.

In order to better understand a given class of Banach algebras, it is natural to study what the possible subalgebras of a given algebra are. This leads to the problem of \emph{embeddability}: Given a Banach algebra $A$, one seeks to understand whether $A$ embeds into a better understood or more rigid algebra, and what such an embedding reveals about the structure of $A$. In the setting of C*-algebras, embedding theorems have had a profound impact, most notably Kirchberg's theorem asserting that every separable, exact, unital C*-algebra embeds unitally into the Cuntz algebra $\mathcal{O}_2$, and the more recent result, due to Schafhauser, characterizing those UCT C*-algebras that embed into an AF-algebra \cite{Sch_AFemb_2020}. These results, together with some predecessors, have influenced large parts of the modern classification program and have inspired analogous questions for other classes of algebras. For example, the failure of a counterpart to Kirchberg's $\Ot$-embedding theorem for $\mathbb{Z}$-algebras has been explored in \cite{BroSor_L2_2016}, where it is shown that the Leavitt algebra $L_{2,\mathbb{Z}}$ does not contain its tensor square $L_{2,\mathbb{Z}}\otimes L_{2,\mathbb{Z}}$ unitally.

In contrast, embedding problems for $L^p$-operator algebras remain far less understood. The absence of adjoints and the absence of techniques available for Hilbert spaces produce significant technical obstacles, and many arguments from the C*-setting simply do not generalize. Nevertheless, recent progress has shown that groupoid methods can be adapted effectively to the $L^p$-framework, allowing one to prove striking rigidity and reconstruction results under suitable hypotheses. This suggests that embeddability questions for $L^p$-operator algebras should be approachable through a careful analysis of the canonical diagonal subalgebra, normalizers, or the topological full group, among others.

This paper is devoted to the study of embeddings of $L^p$-operator algebras associated with (twisted) \'etale groupoids, with a particular emphasis on applications to embeddability questions for some well-studied algebras, such as irrational rotation $L^p$-operator algebras, spatial AF $L^p$-operator algebras, and $L^p$-Cuntz algebras. Our approach is based on a careful study of Weyl twists, core normalizers, and their behavior under algebra homomorphisms. By exploiting the notion of actors between groupoids developed by Meyer and Zhu \cite{MeyZhu_actors_2015}, we are able to analyze embeddings of algebras associated to groupoids directly at the level of the underlying groupoid data. This is surprising at first sight, since in the C*-algebraic setting it is well-known that maps between groupoid C*-algebras do not have to be induced in any way at the groupoid level. We show that, under natural hypotheses, embeddings between certain $L^p$-operator algebras, for $p\neq 2$, are in one-to-one correspondence with free actor morphisms between the associated groupoids. More explicitly, we have the following. For an \'etale, Hausdorff groupoid $\mathcal{G}$, the canonical conditional expectation is denoted by $E_\mathcal{G}\colon F^p_\lambda(\mathcal{G})\to C_0(\mathcal{G}^{(0)})$.

\begin{thmintro}\label{isometric equivalences intro} (See \autoref{isometric equivalences}.)
Let $\mathcal{G}$ and $\mathcal{H}$ be \'etale, effective, Hausdorff groupoids with compact unit spaces, let $p\in (1,\infty)\setminus \{2\}$ and let $\ph\colon F_\lambda^p(\mathcal{G}) \rightarrow F_\lambda^p(\mathcal{H})$ be a unital contractive homomorphism. Then the following are equivalent:
    \begin{enumerate}
        \item $\ph\vert_{C(\mathcal{G}^{(0)})}$ is injective and satisfies $\ph\circ E_{\mathcal{G}} = E_{\mathcal{H}} \circ \ph$.
        \item $\ph$ is isometric and satisfies $\ph\circ E_{\mathcal{G}} = E_{\mathcal{H}}\circ \ph$.
        \item $\varphi$ is induced at the level of the groupoids, in the following sense: There are an \'etale, effective, Hausdorff groupoid $\mathcal{K}$ and groupoid homomorphisms 
        \[ \xymatrix{\mathcal{G} & \mathcal{K} \ar[l]_-{\ \pi} \ar[r]^{\iota \ }& \mathcal{H}}
        \]
        such that $\pi$ is surjective and fiberwise bijective and $\iota$ has open image, is injective and bijective on units, and $\ph(f)= \ind_{\iota(\mathcal{K})}\cdot ( f\circ \pi \circ \iota^{-1})$ for $f\in C_c(\mathcal{G})$.
\end{enumerate}
\end{thmintro}

In (1) and (2), the condition that $\varphi$ intertwines the conditional 
expectations is often automatic, and this allows us to apply the above theorem in many cases of interest. For example, we use this result to obtain different Cartan-type characterizations of embeddings, providing new tools for ruling out the existence of maps between $L^p$-groupoid algebras. These results are $L^p$-analogs of well-known reconstruction theorems in the C*-setting, but their proofs require different techniques.

One of the main applications of our methods is AF-embeddability. In the setting of C*-algebras, AF-embeddability has long been recognized as a strong regularity property, and its relationship with exactness, quasidiagonality, and the UCT has been extensively studied. We investigate AF-embeddability for reduced groupoid $L^p$-operator algebras, and establish a strong rigidity result. Under mild assumptions on $\mathcal{G}$ and for $p\neq 2$, the algebra $F^p_\lambda(\mathcal{G})$ embeds into a spatial AF $L^p$-operator algebra if and only if $\mathcal{G}$ is itself an AF-groupoid. The following consequence of this stands in contrast to the celebrated construction by Pimsner-Voiculescu of an AF-embedding of the irrational rotation algebra:

\begin{corintro}\label{AF rigidity} (See \autoref{AF rigidity text}.)
Let $p\in (1,\infty)\setminus \{2\}$, let $\theta\in \R \setminus \Q$, and let $A_\theta^p := C(\T)\rtimes_{r_\theta}^p \Z$ be the associated $L^p$-irrational rotation algebra, where $r_\theta$ is the rotation homeomorphism by angle $2\pi\theta$. Then there is no unital contractive homomorphism from $A_\theta^p$ into a spatial AF $L^p$-operator algebra. \end{corintro}

Another direction explored in this paper is that of induced embeddings of topological full groups, at least in the setting of ample groupoids. Topological full groups play an important role in the study of dynamical systems and groupoids, and their interaction with operator algebras has been a source of fruitful connections, systematically studied after the work of Matui \cite{Mat_homology_2012}. 
Since topological full groups capture subtle orbit and isotropy information of the groupoids in question, embeddings between them often reflect strong rigidity phenomena. Understanding how isometric embeddings reflect on the level of full groups thus provides a mechanism for transferring analytic information into a purely group-theoretic setting.
In this direction, we show that embeddings of $L^p$-operator algebras give rise to embeddings of the associated full groups, in the following way. 

\begin{thmintro}\label{intro tfg actor and inj} (See \autoref{tfg actor and inj}.)
    Let $p\in [1,\infty)\setminus \{2\}$ and let $\mathcal{G}$ and $\mathcal{H}$ be ample Weyl groupoids satisfying condition (W). If there exists a unital contractive homomorphism $F_\lambda^p(\mathcal{G})\rightarrow F_\lambda^p(\mathcal{H})$ that is injective on $C(\mathcal{G}^{(0)})$, then $[\![\mathcal{G}]\!]\hookrightarrow [\![\mathcal{H}]\!]$. 
\end{thmintro}

The above result allows us to deduce group-theoretical properties from Banach-algebraic data, and we apply this to the study of tensor products of $L^p$-Cuntz algebras. In this direction, we obtain further rigidity results which once again show a striking contrast with the setting of C*-algebras. Among other things, \autoref{cuntz rigidity text} implies the following.

\begin{corintro}\label{O2 rigidity} (See \autoref{non emb text}.)
Let $p\in [1,\infty)\setminus \{2\}$. Then there does not exist a unital contractive homomorphism $\mathcal{O}_2^p\otimes_p \mathcal{O}_2^p\to \mathcal{O}_2^p$.
\end{corintro}

Besides being a significant strengthening of the result, obtained in \cite{ChoGarThi_rigidity_2021}, that there is no isometric isomorphism
$\mathcal{O}_2^p\otimes_p \mathcal{O}_2^p\cong \mathcal{O}_2^p$ if $p\neq 2$, our result also shows that there is no reasonable analog of Kirchberg's $\Ot$-embedding theorem in the $L^p$-setting. 

\vspace{.2cm} 
\textbf{Acknowledgments:} We would like to thank Bartosz Kwa\'sniewski and Hannes Thiel for their feedback on a preliminary version of this work. 

\section{Preliminaries}\label{prelim chapter}

In this section, we define $L^p$-operator algebras as introduced by Phillips in \cite{Phi_Cuntz_2012}, and we also define reduced twisted groupoid algebras as in \cite[Section 6.2]{HetOrt_twistlp_2022}. Our notation for groupoids follows the one used in \cite[Sections 3 and 4]{Ren_Cartan_2008}, \cite[Section 4]{Tay_Fellactor_2023}, \cite[Sections 3--5]{Exe_invsg_2008}, and \cite[Section 2]{BarKwaMcK_Banach_2023}.

\begin{df}\label{lpdef}
    Let $p\in [1,\infty)$. A Banach algebra $A$ is called an \textit{$L^p$-operator algebra} if there are a measure space $(\Omega,\mathcal{A},\mu)$ and an isometric homomorphism from $A$ into the bounded linear operators $\B(L^p(\mu))$.
\end{df}
A fundamental class of examples of $L^p$-operator algebras is given by reduced (twisted) groupoid algebras. For a proper introduction to groupoids, we refer the reader to \cite[Section 2]{Sims_groupoids_2011}. At this point, we merely recall our standard notation and definitions for \'etale groupoids.

\begin{nota}\label{fibre prod}
    Let $X,\,Y$ and $Z$ be topological spaces and let $f\colon X\rightarrow Z$ and $g\colon Y\rightarrow Z$ be continuous maps. We denote the topological pullback of $f$ and $g$ by $X\times_{f,\,g} Y$, that is, 
\[X\times_{f,\,g} Y= \{(x,y)\in X\times Y\colon f(x)=g(y)\}.\] We endow $X\times_{f,\,g} Y$ with the topology given by the subspace topology in $X\times Y$.
\end{nota}

\begin{df}\label{et grpd def}
    A \textit{topological groupoid} $\mathcal{G}$ is a groupoid, that is, a small category where all morphisms are invertible, equipped with a locally compact topology such that multiplication and inversion are continuous, and such that the \textit{unit space} $\mathcal{G}^{(0)} := \{\gamma\in \mathcal{G}\colon \gamma^2=\gamma\}$ is Hausdorff. We define the \textit{range} and \textit{domain} maps by $\mathsf{r}(\gamma) := \gamma\gamma^{-1}$ and $\mathsf{d}(\gamma) := \gamma^{-1}\gamma$ for all $\gamma\in \mathcal{G}$. 
    The set of \textit{composable pairs} is given by $\mathcal{G}^{(2)}:= \mathcal{G}\times_{\mathsf{d},\,\mathsf{r}} \mathcal{G}$.
    For $x\in \mathcal{G}^{(0)}$, the range and domain \textit{fibers} are $x\mathcal{G}:= \mathsf{r}^{-1}(\{x\})$ and $\mathcal{G}x := \mathsf{d}^{-1}(\{x\})$. We define an \textit{open bisection} as an open subset $S\subseteq \mathcal{G}$ such that $\mathsf{r}\vert_S$ and $\mathsf{d}\vert_S$ are injective.
    
    We say that $\mathcal{G}$ is:
\begin{itemize}
\item \textit{\'etale}, if $\mathsf{r}$ and $\mathsf{d}$ are local homeomorphisms. 
\item \textit{effective}, if the interior of the isotropy bundle $\textup{Iso}(\mathcal{G}):= \{\gamma\in \mathcal{G}\colon \mathsf{r}(\gamma)=\mathsf{d}(\gamma)\}$ is $\mathcal{G}^{(0)}$.
\item \emph{principal}, if the isotropy bundle agrees with the unit space.\end{itemize}
\end{df}

\begin{rem}
We will often use the fact that if $\mathcal{G}$ is an \'etale groupoid, then the unit space $\mathcal{G}^{(0)}$ is open, the fibers are discrete, and the topology of $\mathcal{G}$ is generated by open bisections. Moreover, $\mathcal{G}$ is Hausdorff if and only if $\mathcal{G}^{(0)}$ is closed in $\mathcal{G}$.
\end{rem}


Concrete examples of \'etale groupoids are given by transformation groupoids. In their most general form, which is also the form in which we will need them, they are built out of inverse semigroup actions, as we explain next.

\begin{df}\label{germ grpd}
    Let $X$ be a locally compact Hausdorff space and let $\mathcal{S}$ be an inverse semigroup. Let $\beta\colon \mathcal{S}\rightarrow \textup{Homeo}_{\textup{par}}(X)$ be a semigroup homomorphism such that $\bigcup_{s\in \mathcal{S}}\textup{dom}(\beta_s)=X$, which we abbreviate to $\beta\colon \mathcal{S}\curvearrowright X$. On the set $\{(s,x)\colon s\in \mathcal{S},\,x\in \textup{dom}(\beta_s)\}$ we define an equivalence relation by
    \[(s,x)\sim (s',y) \textup{\quad if and only if\quad} \begin{cases} x=y \,\,\, \textup{and there is an idempotent}\,\,e\in \mathcal{S}\\ \textup{such that}\,\,\, se=s'e \,\, \textup{and}\,\,x\in \textup{dom}(\beta_e).\end{cases}\]     
    The equivalence class of $(s,x)$ is called the \textit{germ} of $(s,x)$, and is denoted by $[s,x]$. The set $\mathcal{S} \ltimes_\beta X := \{[s,x]\colon s\in \mathcal{S},\,x\in \textup{dom}(\beta_s)\}$ admits a groupoid structure via
    \[ [s,\beta_t(y)]\cdot [t,y] := [st,y]\andeqn [s,x]^{-1} := [s^{*},\beta_s(x)]. 
    \]
    We call $\mathcal{S} \ltimes_\beta X$ the \textit{transformation groupoid} of $\beta\colon \mathcal{S}\curvearrowright X$. Unless stated otherwise, we equip the transformation groupoid with the \textit{\'etale topology} generated by basic open subsets of the form $\{[s,x]\colon x\in U\}$ for $s\in \mathcal{S}$ and $U\subseteq\textup{dom}(\beta_s)$ open. In this topology, $\mathcal{S} \ltimes_\beta X$ is an \'etale groupoid with unit space homeomorphic to $\bigcup_{s\in \mathcal{S}}\textup{dom}(\beta_s)=X$. If $\beta$ is injective, then $\mathcal{S}$ is isomorphic to its image under $\beta$ and we refer to $\mathcal{S} \ltimes_\beta X \cong \beta(\mathcal{S}) \ltimes X$ as the \textit{groupoid of germs}.
\end{df}

\begin{df}\label{BG alpha and tfg nota}
    Let $\mathcal{G}$ be an \'etale groupoid. We write $\mathcal{B}(\mathcal{G})$ for the set of open bisections. Given $S,\,T\in \mathcal{B}(\mathcal{G})$, the sets $S^{-1}:= \{\gamma^{-1}\colon \gamma\in S\}$ and $ST:= \{\gamma\tau\colon (\gamma,\tau)\in S\times_{\mathsf{d},\,\mathsf{r}} T\}$ are also open bisections, and thus $\mathcal{B}(\mathcal{G})$ is a unital inverse semigroup with unit $\mathcal{G}^{(0)}$. We further write $\alpha\colon \mathcal{B}(\mathcal{G})\rightarrow \textup{Homeo}_{\textup{par}}(\mathcal{G}^{(0)})$ for the \textit{bisection action} given by $\alpha_S := \mathsf{r}\vert_S\circ \mathsf{d}\vert_S^{-1}$ for all $S\in \mathcal{B}(\mathcal{G})$. By restricting $\alpha$ to bisections with full domain and range, in turn, we obtain the \textit{topological full group} 
    \[[\![\mathcal{G}]\!] := \{\alpha_S\colon S\in \mathcal{B}(\mathcal{G}),\, 
    \mathsf{d}(S)=\mathsf{r}(S)=\mathcal{G}^{(0)}\}.\] 
\end{df}

The following is \cite[Proposition 3.3]{Ren_Cartan_2008}, and we record it here for convenience.

\begin{prop}\label{alpha and eff}
    Let $\mathcal{G}$ be a Hausdorff \'etale groupoid. Then the unital semigroup homomorphism $\alpha$ in \autoref{BG alpha and tfg nota} is injective if and only if $\mathcal{G}$ is effective.
\end{prop}

\begin{df}\label{wide}
    Let $\mathcal{G}$ be an \'etale groupoid and let $\mathcal{S}\subseteq \mathcal{B}(\mathcal{G})$ be an inverse subsemigroup. We say that $\mathcal{S}$ is \textit{wide} if $\mathcal{S}$ covers $\mathcal{G}$ and if for every $S,T\in \mathcal{S}$, the intersection $S\,\cap\, T$ is a union of bisections in $\mathcal{S}$. We abbreviate the corresponding restricted bisection action to $\alpha\colon \mathcal{S}\curvearrowright \mathcal{G}^{(0)}$.
\end{df}
The following is \cite[Proposition 5.4]{Exe_invsg_2008} and asserts that all \'etale groupoids are transformation groupoids of their respective bisection actions.

\begin{prop}\label{etale as trafo grpd}
    Let $\mathcal{G}$ be an \'etale groupoid. Let $\mathcal{S}\subseteq \mathcal{B}(\mathcal{G})$ be a wide inverse subsemigroup and let $\alpha\colon \mathcal{S}\curvearrowright \mathcal{G}^{(0)}$ be the associated bisection action. Then the map $\Psi\colon \mathcal{S} \ltimes_\alpha \mathcal{G}^{(0)} \rightarrow \mathcal{G}$ given by $\Psi([S,x]):= Sx$ for all $[S,x]\in \mathcal{S} \ltimes_\alpha \mathcal{G}^{(0)}$ is an isomorphism of topological groupoids.
\end{prop}

\begin{df}\label{twist def}
    Let $\mathcal{G}$ be a topological groupoid. A surjective continuous grou\-poid homomorphism $\pi\colon \Sigma \rightarrow \mathcal{G}$ is called a \textit{groupoid extension} by the groupoid $\ker(\pi):= \{\sigma\in \Sigma\colon \pi(\sigma)\in \mathcal{G}^{(0)}\}$. A \textit{twist} over $\mathcal{G}$ is a groupoid extension $\pi\colon \Sigma \rightarrow \mathcal{G}$ by the group bundle $\T\times \mathcal{G}^{(0)}$ such that 
    \begin{enumerate}[label=(\roman*)]
        \item the homomorphism $\pi$ is locally trivializable, and
        \item for all $\sigma \in \Sigma$ and $\zeta\in \T$, we have $\big(\zeta,\mathsf{r}(\pi(\sigma))\big)\sigma = \sigma \big(\zeta, \mathsf{d}(\pi(\sigma))\big)$, that is, the group bundle $\T\times \mathcal{G}^{(0)}$ is \textit{central}.
    \end{enumerate}
    In this case, $\mathcal{G}$ is called the \textit{base groupoid} and $\Sigma$ is called the \textit{twisted groupoid}. Throughout, we regard the torus as a group under multiplication and use centrality in (ii) to write products in $\Sigma$ by elements in $\T\times \mathcal{G}^{(0)}$ as scalar multiplication. Using this notation, a map $f\colon \Sigma \rightarrow \mathbb{C}$ is called \textit{$\T$-equivariant} if $f(\zeta \sigma) = \overline{\zeta}f(\sigma)$ for all $\zeta\in \T$ and $\sigma\in \Sigma$. 
    \end{df}

\begin{nota}\label{twist notation}
     Let $\pi\colon \Sigma \rightarrow \mathcal{G}$ be a twist. We write $(\mathcal{G},\Sigma)$ for $\pi\colon \Sigma \rightarrow \mathcal{G}$ and identify the unit space $\Sigma^{(0)} = \{(1,x)\colon x\in \mathcal{G}^{(0)}\}$ with $\mathcal{G}^{(0)}$. We use the abbreviation $\sigma_{\bullet}:= \pi(\sigma)\in \mathcal{G}$ for $\sigma \in \Sigma$ and the notation $\widehat{\gamma}$ for an arbitrary lift of $\gamma\in \mathcal{G}$ along $\pi$.
\end{nota}    

\begin{df}\label{twist hom}
    Let $\pi_\mathcal{G}\colon \Sigma \rightarrow \mathcal{G}$ and $\pi_\mathcal{H}\colon \Omega \rightarrow \mathcal{H}$ be twists. A \textit{twist homomorphism} between them is a continuous groupoid homomorphism $\kappa\colon \Sigma\rightarrow \Omega$ that satisfies $\kappa(\zeta \sigma) = \zeta\kappa(\sigma)$ for all $\sigma\in \Sigma$ and $\zeta \in \T$. For a twist homomorphism $\kappa$, the assignment $\kappa_{\textup{b}}(\gamma):= \pi_\mathcal{H}(\kappa(\widehat{\gamma}))$ for $\gamma \in \mathcal{G}$ is a well-defined homomorphism between the base groupoids satisfying $\pi_\mathcal{H}\circ\kappa = \kappa_{\textup{b}}\circ\pi_\mathcal{G}$.
    \end{df}
See \cite{Sims_groupoids_2011} for foundations on the theory of twisted groupoids. We exclusively work with base groupoids $\mathcal{G}$ which are \'etale, and will be mainly concerned with Weyl twists, which we define next.

\begin{df}\label{df weyl twist abbr}
    We define a \textit{Weyl groupoid} to be an effective Hausdorff \'etale groupoid, and a \textit{Weyl twist} to be a twist $(\mathcal{G},\Sigma)$ over a Weyl groupoid $\mathcal{G}$.
\end{df}

\begin{df}\label{fell line bundle df}
    Let $\pi\colon \Sigma \rightarrow \mathcal{G}$ be a twist over an \'etale groupoid. We equip the set $\C\times \Sigma$ with the equivalence relation $(z,\sigma)\sim (\overline{\zeta} z,\zeta \sigma)$ for all $\zeta \in \T$ and denote the equivalence class of $(z,\sigma)\in \C\times \Sigma$ by $[z,\sigma]$. On the set $L$ of equivalence classes we define a bundle map $p_L\colon L \rightarrow \mathcal{G}$ and an involution $L\to L$ given by $p_L([z,\sigma]) := \pi(\sigma) = \sigma_\bullet\in \mathcal{G}$ and $[z,\sigma]^* := [\overline{z},\sigma^{-1}]$ for all $[z,\sigma]\in L$. Given $\gamma\in\mathcal{G}$, we set $L_\gamma=p_L^{-1}(\gamma)$, and define a multiplication 
\[L_\gamma \times_{\mathsf{d},\,\mathsf{r}} L_\eta \rightarrow L_{\gamma\eta}\] 
   by $[z,\sigma]\cdot [z',\sigma'] := [zz',\sigma \sigma']$ for all $([z,\sigma],[z',\sigma']) \in L_\gamma \times_{\mathsf{d},\,\mathsf{r}} L_\eta$.
   Finally, we define range and domain maps on $L$ by 
\[\mathsf{d}([z,\sigma])=\mathsf{d}(\sigma_\bullet) \ \ \mbox{ and } \ \
\mathsf{r}([z,\sigma])=\mathsf{r}(\sigma_\bullet)
\]
for all $[z,\sigma]\in L$.   
   We call $L$ the \textit{Fell line bundle associated to $\Sigma$}.\par Given a Fell line bundle $p_L\colon L\rightarrow \mathcal{G}$, an element $u\in L$ such that $u^*u = 1_{\mathsf{d}(u)}$ and $uu^*= 1_{\mathsf{r}(u)}$ is called a \textit{unitary} and we write $\mathcal{U}(L)$ for the set of unitaries.
\end{df}

By \cite[Example 5.5]{DeaKumBir_Fellmorphism_2008}, for an \'etale groupoid $\mathcal{G}$, there is a ca\-no\-ni\-cal one-to-one correspondence between twists over $\mathcal{G}$ and Fell line bundles over it. We will therefore freely identify a twist $\Sigma$ with the groupoid $\mathcal{U}(L)$ of unitaries of the associated Fell line bundle over $\mathcal{G}$, according to \autoref{fell line bundle df}.



\begin{lma}\label{sections corr T equi}
    Let $(\mathcal{G},\Sigma)$ be a twist over an \'etale groupoid. Then there is a ca\-no\-ni\-cal one-to-one correspondence between sections $\mathcal{G} \rightarrow L$ of the bundle associated to $\Sigma$ and $\T$-equivariant maps $\Sigma \rightarrow \mathbb{C}$. This correspondence identifies a $\T$-equivariant map $f\colon \Sigma \rightarrow \mathbb{C}$ with the section $\widehat{f}\colon \mathcal{G}\to L$ given by $\widehat{f}(\gamma):= [f(\widehat{\gamma}),\widehat{\gamma}]$ for all $\gamma\in\mathcal{G}$.
    \end{lma}

\begin{proof}
Since $\pi^{-1}(\gamma) = \T\widehat{\gamma}$, the map $\widehat{f}$ in the statement is independent of the lift, and it is easy to see that it is a section. Conversely, given a section $\widehat{f}\colon \mathcal{G} \rightarrow L$ and $\sigma \in \Sigma$, we observe that $\textup{pr}_\Sigma(\widehat{f}(\sigma_{\bullet}))\in \T\sigma$, and we define $\zeta(\sigma)\in \T$ to satisfy the identity $\zeta(\sigma)\sigma = \textup{pr}_\Sigma(\widehat{f}(\sigma_{\bullet}))$. We recover the $\T$-equivariant map by setting $f(\sigma):= \textup{pr}_\C(\widehat{f}(\sigma_{\bullet}))\zeta(\sigma)$ for all $\sigma\in \Sigma$.
\end{proof}

From now on, we treat $\T$-equivariant maps $\Sigma \rightarrow \mathbb{C}$ as sections $\mathcal{G}\to L$ via the correspondence in \autoref{sections corr T equi}.

\begin{df}\label{trivializable}
    Let $(\mathcal{G},\Sigma)$ be a twist over an \'etale groupoid. We set 
    \[\mathcal{B}(\mathcal{G},\Sigma):= \{S\in \mathcal{B}(\mathcal{G})\colon L\vert_S \cong \C\times S\}.\] For each $S\in \mathcal{B}(\mathcal{G},\Sigma)$, we choose a homeomorphism $\iota\colon \T\times S\rightarrow \Sigma\vert_S$ that is the identity on $\{1\}\times S\cap \mathcal{G}^{(0)}$. For $\gamma\in S$, we set $1_\gamma := \iota(1,\gamma)$. Then the assignment $1_{(\cdot)}\colon S\rightarrow \Sigma$ is a continuous unitary section, which is the identity on units. Let $f\colon \mathcal{G}\rightarrow \C$ be a function and set $S:=\{\gamma\in \mathcal{G}\colon f(\gamma)\neq 0\}$. Assume that $S$ belongs to $\mathcal{B}(\mathcal{G},\Sigma)$. We define the \textit{$1_S$-lift} of $f$ as the section $f_{1_S}\colon \Sigma \rightarrow \C$ given by \[f_{1_S}(\sigma):= \begin{cases}\overline{\zeta}f(\sigma_\bullet), \,\,\,\textup{if}\,\,\, \sigma = \zeta 1_{\sigma_\bullet}\in \T1_S,\\ 0,\,\,\,\textup{otherwise}.\end{cases}\]
\end{df}
We proceed to introduce the strict support of $\mathbb{T}$-equivariant maps $\Sigma\to \mathbb{C}$, as well as their convolution. Since the base groupoid might be non-Hausdorff, we use a more general definition of the $C_c$-section algebra in terms of bisection supports. For a more detailed treatment, see also \cite[Definition 3.9]{Exe_invsg_2008}.

\begin{df}\label{section operations}
    Let $(\mathcal{G},\Sigma)$ be a twist over an \'etale groupoid. Let $f\colon \Sigma \rightarrow \C$ be a $\T$-equivariant map. Its \textit{involution} $f^*\colon \Sigma \rightarrow \C$ is given by $f^*(\sigma) := \overline{f(\sigma^{-1})}$ for $\sigma\in \Sigma$. We further define both the \textit{strict support} $\supp(f) := \{\gamma\in \mathcal{G}\colon f(\widehat{\gamma})\neq 0\}$ and the \textit{support} $\overline{\supp(f)}$ as subsets of $\mathcal{G}$. \par A $\T$-equivariant map $f\colon \Sigma \rightarrow \C$ is a \textit{$C_c$-section} if there are finitely many $\T$-equivariant maps $f_1, \ldots, f_n\colon \Sigma \rightarrow \C$ such that $S_k:=\supp(f_k)$ belongs to $\mathcal{B}(\mathcal{G},\Sigma)$, the restriction $f_k\vert_{S_k}\in C_c(S_k,L)$ is continuous and compactly supported, and $f=\sum_{k=1}^n f_k$. We denote the set of $C_c$-sections by $\mathcal{C}_c(\mathcal{G},\Sigma)$ and regard it as a $*$-algebra under the convolution product given by
    \begin{align*}\label{twisted conv}
        \left(f_1*f_2 \right)(\sigma) := \sum_{\tau \in \mathsf{r}(\sigma)\mathcal{G}} f_1(\widehat{\tau})f_2(\widehat{\tau}^{-1}\sigma) \,\,\end{align*}
for all $f_1,f_2\in \mathcal{C}_c(\mathcal{G},\Sigma)$, and all $\sigma \in \Sigma$.\end{df}

Since $\pi$ is locally trivializable, we have that if the base groupoid $\mathcal{G}$ is Hausdorff, then $\mathcal{C}_c(\mathcal{G},\Sigma) = C_c(\mathcal{G},\Sigma)$.

\begin{rem}\label{cocycle twist}
If there exists a global continuous section $1_{(\cdot)}\colon \mathcal{G}\rightarrow \Sigma$, then the twisted groupoid has the form $\Sigma = \bigsqcup_{\gamma\in \mathcal{G}} \T1_\gamma$ and can equivalently be described by a continuous and normalized $2$-cocycle $c \colon \mathcal{G}^{(2)} \rightarrow \T$ via the relation that the multiplication in $\Sigma$ is dictated by $1_\gamma \cdot 1_\eta = c(\gamma,\eta)1_{\gamma \eta}$ for all $(\gamma,\eta)\in \mathcal{G}^{(2)}$. A section $f\colon \Sigma \rightarrow \C$ is then fully determined by its values on $\{1_\gamma \colon \gamma \in \mathcal{G}\}$, and convolution for $f_1,\,f_2\in \mathcal{C}_c(\mathcal{G},\Sigma)$ agrees with twisted convolution on $\mathcal{C}_c(\mathcal{G})$, namely:
\[ (f_1*f_2)(1_\gamma) = \sum_{\tau \in \mathsf{r}(\gamma)\mathcal{G}} f_1(1_\tau)c(\tau,\tau^{-1}\gamma)f_2(1_{\tau^{-1}\gamma}) \,\,\textup{for all}\,\, \gamma\in \mathcal{G}.\] In particular, for the trivial twist $\Sigma = \T \times \mathcal{G}$ over a Hausdorff base groupoid, the $*$-algebras $\mathcal{C}_c(\mathcal{G},\Sigma)$ and $C_c(\mathcal{G})$ are naturally isomorphic.
\end{rem}

\begin{nota}\label{set convention}
Let $(\mathcal{G},\Sigma)$ be a twist over an \'etale groupoid, let $S\in \mathcal{B}(\mathcal{G},\Sigma)$, and let $\sigma\in \Sigma$. Then the set $1_{S}\sigma\subseteq\Sigma$ is either a singleton or the empty set. To lighten the notation, we interpret function evaluations as evaluation at this singleton in the first case and as zero otherwise. That is, for any section $f\colon \Sigma \rightarrow \C$, we write
    \[f(1_{S}\sigma) := \begin{cases}
        f(1_{\tau}\sigma),\,\, \textup{if} \,\, \mathsf{r}(\sigma)\in \mathsf{d}(S) \,\, \textup{and}\,\, S\mathsf{r}(\sigma) = \{\tau\}, \\ 0,\,\, \textup{otherwise}
    \end{cases}\]
for all $\sigma\in\Sigma$. The expression $f(\sigma 1_{S})$ is defined similarly. Note that $1_{S^{-1}}\sigma$ and $1_{S}^{-1}\sigma$, as choices for a lift of $S^{-1}\sigma_{\bullet}$, do not necessarily agree in $\Sigma$. In the untwisted case, this subtlety disappears by \autoref{cocycle twist}. 
\end{nota}

\begin{prop}\label{S convo}
    Let $(\mathcal{G},\Sigma)$ be a twist over an \'etale groupoid. Let $S\in \mathcal{B}(\mathcal{G},\Sigma)$ and interpret $\ind_S$ as the section that is the $1_S$-lift of the indicator function on $1_S\subseteq \Sigma$ as in \autoref{trivializable}. For $f\in \mathcal{C}_c(\mathcal{G},\Sigma)$ and $\sigma \in \Sigma$, we have
    \[(\ind_S * f)(\sigma)= f(1_{S}^{-1}\sigma) \andeqn (f*\ind_S)(\sigma)= f(\sigma 1_{S}^{-1}).
    \]
\end{prop}

\begin{proof}
    We show the first identity since the second one follows analogously. Note that the sums below collapse and directly lead to
    \[ (\ind_S*f)(\sigma) = \sum_{\tau\in \mathsf{r}(\gamma)\mathcal{G}} \ind_S(\widehat{\tau})f(\widehat{\tau}^{-1}\sigma) = \sum_{\tau\in \mathsf{r}(\gamma)S} \ind_S(1_{\tau})f(1_\tau^{-1}\sigma) = f(1_S^{-1}\sigma).\qedhere\]
\end{proof}

\begin{df}\label{Fp def}
    Let $(\mathcal{G},\Sigma)$ be a twist over an \'etale groupoid. For every $x\in \mathcal{G}^{(0)}$, we choose a fiberwise lift $\widehat{\mathcal{G}x}$, so that $\pi$ induces a bijection $\widehat{\mathcal{G}x}\to\mathcal{G}x$. For $p\in [1,\infty)$, there is thus an isometric isomorphism $\ell^p(\mathcal{G}x)\cong \ell^p(\widehat{\mathcal{G}x})$ induced by $\pi$, and we identify an element $\xi\in \ell^p(\mathcal{G}x)$ with its $\T$-equivariant extension to $\T \widehat{\mathcal{G}x} = \Sigma x$ whenever convolution identities apply. The \textit{left regular representation} on the domain fiber of $x\in \mathcal{G}^{(0)}$ is the homomorphism $\lambda_x\colon \mathcal{C}_c(\mathcal{G},\Sigma) \rightarrow \B(\ell^p(\mathcal{G}x))$ given by \[\lambda_x(f)(\xi):= (f*\xi)\vert_{\widehat{\mathcal{G}x}}\] for all $f\in \mathcal{C}_c(\mathcal{G},\Sigma)$ and $\xi \in \ell^p(\mathcal{G}x)$. We define the \emph{reduced norm} $\|\cdot\|_\lambda$ on $\mathcal{C}_c(\mathcal{G},\Sigma)$ as $\|f\|_\lambda := \sup\{\|\lambda_x(f)\|\colon x\in \mathcal{G}^{(0)}\}$ for $f\in \mathcal{C}_c(\mathcal{G},\Sigma)$. The completion of $\mathcal{C}_c(\mathcal{G},\Sigma)$ in this norm is called the \textit{reduced twisted groupoid $L^p$-operator algebra} and is denoted by $F_\lambda^p(\mathcal{G},\Sigma)$. 
\end{df}

\begin{rem}\label{Fp is Lp op}
    The reduced twisted groupoid $L^p$-operator algebra is indeed an $L^p$-operator algebra since the extension of $\bigoplus_{x\in \mathcal{G}^{(0)}}\lambda_x$ by continuity provides an isometric embedding into $\bigoplus_{x\in \mathcal{G}^{(0)}} \B(\ell^p(\mathcal{G}x))$, which is isometrically represented on $\bigoplus_{x\in \mathcal{G}^{(0)}} \ell^p(\mathcal{G}x)$. See also \cite[Proposition 5.1]{BarKwaMcK_Banach_2023} for a representation on $\ell^p(\mathcal{G})$. Note that the construction of $F_\lambda^p(\mathcal{G},\Sigma)$ includes reduced groupoid $L^p$-operator algebras for the trivial twist $\Sigma = \T \times \mathcal{G}$ and reduced twisted groupoid C*-algebras $F_\lambda^2(\mathcal{G},\Sigma)= C_\lambda^*(\mathcal{G},\Sigma)$ as special cases.
\end{rem}

\begin{nota}\label{j supp' nota}
     Let $(\mathcal{G},\Sigma)$ be a twist over an \'{e}tale groupoid. For $p\in [1,\infty)$, let $q\in (1,\infty]$ satisfy $p^{-1}+q^{-1}=1$ with the convention that $\infty^{-1}:=0$. For $x\in \mathcal{G}^{(0)}$, define a pairing between $\ell^p(\mathcal{G}x)$ and $\ell^q(\mathcal{G}x)$ by $\langle \xi,\omega\rangle_x := \sum_{\gamma\in \mathcal{G}x} \xi(\gamma)\omega(\gamma)$ for all $\xi \in \ell^p(\mathcal{G}x)$ and $\omega \in \ell^q(\mathcal{G}x)$. As in \cite[Section 4]{ChoGarThi_rigidity_2021} and \cite[Proposition 6.21]{HetOrt_twistlp_2022}, we denote by $j\colon F_\lambda^p(\mathcal{G},\Sigma) \rightarrow C_0(\Sigma)$ the linear injection induced by the identity map $(\mathcal{C}_c(\mathcal{G},\Sigma),\|\cdot\|_\lambda)\rightarrow (\mathcal{C}_c(\mathcal{G},\Sigma),\|\cdot\|_\infty)$. Concretely, by using the $\T$-equivariant functions $\delta_\sigma\colon \Sigma x \rightarrow \T$ defined by $\delta_\sigma(\tau) := \textup{pr}_{\T}(\tau^{-1}\sigma)\delta_{\sigma_{\bullet},\tau_{\bullet}}$ for $\sigma,\tau \in \Sigma x$, the $j$-map is given by $j_a(\sigma) = \langle \lambda_{x}(a)(\delta_{x}),\delta_{\sigma}\rangle_x$ for all $a\in F_\lambda^p(\mathcal{G},\Sigma), \sigma \in \Sigma x$ and $x\in \mathcal{G}^{(0)}$. This map $j$ allows us to consistently extend the notion of \textit{strict supports} from $\mathcal{C}_c(\mathcal{G},\Sigma)$ to $F_\lambda^p(\mathcal{G},\Sigma)$ and for $a\in F_\lambda^p(\mathcal{G},\Sigma)$, we write $\supp(a):= \{\gamma\in \mathcal{G}\colon j_a(\widehat{\gamma})\neq 0\}$.
\end{nota}
Note that the convolution of functions supported on the unit space is simply their pointwise product. Moreover, for such functions, the reduced norm agrees with the supremum norm, as it does on all functions supported on an open bisection. As a result, for all $p\in [1,\infty)$ and $S\in \mathcal{B}(\mathcal{G},\Sigma)$, we obtain an isometric inclusion $C_0(S)\subseteq F_\lambda^p(\mathcal{G},\Sigma)$ given by the $1_S$-lift in \autoref{trivializable}. In particular, $C_0(\mathcal{G}^{(0)})$ is a commutative subalgebra of $F_\lambda^p(\mathcal{G},\Sigma)$.

\begin{df}\label{E expectation}
    Let $(\mathcal{G},\Sigma)$ be a twist over a Hausdorff \'{e}tale groupoid and let $p\in [1,\infty)$. We define the \textit{conditional expectation} $E_{\mathcal{G}}\colon F_\lambda^p(\mathcal{G},\Sigma) \rightarrow C_0(\mathcal{G}^{(0)})$ as $E_{\mathcal{G}}(a) := j_a\vert_{\mathcal{G}^{(0)}}$ for all $a\in F_\lambda^p(\mathcal{G},\Sigma)$.
\end{df}
Note that Hausdorffness is required for the restriction to the unit space to be continuous. In this case, the conditional expectation $E_\mathcal{G}$ is faithful since $F_\lambda^p(\mathcal{G},\Sigma)$ is a reduced Banach algebra in the sense of \cite[Definition 3.1]{BarKwaMcK_BanachII_2024}.

\section{Weyl twists and core normalizers}\label{core pres chapter}

In this section, we study Weyl twists $(\mathcal{G},\Sigma)$ in the sense of Renault \cite{Ren_Cartan_2008} and \autoref{df weyl twist abbr}. The goal of this section is to adapt the C*-algebraic technique to model Weyl twists in terms of a C*-normalizer action to the $L^p$-operator algebraic context. To this end, we introduce spatial normalizers. In contrast to the case of \cas, for $p,q\not \in \{1,2\}$ and for twisted \'etale groupoids $(\mathcal{G},\Sigma)$ and $(\mathcal{H},\Omega)$, we show in \autoref{sn preserved} that unital contractive homomorphisms $F_\lambda^p(\mathcal{G},\Sigma)\rightarrow F_\lambda^q(\mathcal{H},\Omega)$ preserve both the algebra of continuous functions on the unit space and the set of spatial normalizers. For Weyl twists, this results in an induced spatial normalizer action.\par We begin by discussing hermitian elements.

\begin{df}\label{core df}
    Let $A$ be a unital Banach algebra. An element $a\in A$ is said to be \textit{hermitian} if $\|e^{ita}\|\leq 1$ for all $t\in \R$. We denote the set of hermitian elements in $A$ by $A_{\mathrm{h}}$ and define the \textit{core} of $A$ as $\textup{core}(A):= A_{\mathrm{h}} + iA_{\mathrm{h}}$. We further define an involution on $\textup{core}(A)$ by $(a+ib)^*:= a - ib$ for $a,b\in A_{\mathrm{h}}$.
\end{df}
In contrast to general Banach algebras, it was shown in \cite[Section 2]{ChoGarThi_rigidity_2021} that, for unital $L^p$-operator algebras, the core is closed under multiplication and defines a canonical C*-subalgebra in the following sense.

\begin{df}\label{c* subalg}
    Let $A$ be a Banach algebra. We say that a closed subalgebra $B\subseteq A$ is a \textit{C*-subalgebra} of $A$ if, as a Banach algebra, $B$ is isometrically isomorphic to a \ca.
\end{df}
We summarize the properties of the core in the following theorem.

\begin{thm}\label{core}
    Let $p\in [1,\infty)$ and let $A$ be a unital $L^p$-operator algebra. Then $\textup{core}(A)$ is the largest C*-subalgebra of $A$. In particular, if $p=2$ and if $A$ is a \ca, then $\textup{core}(A)=A$. Moreover, if $p\neq 2$, then $\textup{core}(A)$ is automatically a commutative \ca.
\end{thm}

\begin{lma}\label{hermitian preservation}
    Let $\ph\colon A\rightarrow B$ be a unital contractive Banach algebra homomorphism. Then $\ph(A_{\mathrm{h}}) \subseteq B_{\mathrm{h}}$. If $\ph$ is isometric, then $a\in A_{\mathrm{h}}$ if and only if $\ph(a)\in B_{\mathrm{h}}$.
\end{lma}

\begin{proof}
    Let $t\in \R$ and $a\in A_{\mathrm{h}}$. Contractivity of $\ph$ yields $\|\ph(e^{ita})\| \leq \|e^{ita}\|\leq 1$ and the element on the left agrees with $e^{it\ph(a)}$, since functional calculus commutes with (unital) homomorphisms. This shows the first claim. The norm condition in the definition of hermitian elements is, in fact, equivalent to $\|e^{ita}\|= 1$ for all $t\in \R$. Thus, if $\ph$ is isometric, then both norm inequalities above are equalities. This shows the second claim.
\end{proof}

We only define hermitian elements in unital Banach algebras because its definition in terms of exponentials requires the existence of a unit. Beyond this setting, one can consider norms on the minimal unitization to define hermitian elements in greater generality, as in \cite{BlePhi_appid_2019}, but we do not pursue this direction here.

The following is probably known to experts, but it has not appeared in the literature. Accordingly, we include the short argument.

\begin{lma}\label{unital cpt}
    Let $(\mathcal{G},\Sigma)$ be a twist over an \'etale groupoid and let $p\in [1,\infty)$. Then $F_\lambda^p(\mathcal{G},\Sigma)$ is unital if and only if the unit space $\mathcal{G}^{(0)}$ is compact, in which case the unit is given by the $\T$-equivariant extension of $\ind_{\mathcal{G}^{(0)}}$.
\end{lma}

\begin{proof}
    The if implication is straightforward. Conversely, if $u\in F_\lambda^p(\mathcal{G},\Sigma)$ is a multiplicative unit, then \[j_u(\sigma) = \langle u*\delta_x,\delta_\sigma\rangle_x = \langle \delta_x,\delta_\sigma \rangle_x = \delta_{\sigma,x}\] for $\sigma \in \Sigma x$. Hence $\ind_{\mathcal{G}^{(0)}}=E_{\mathcal{G}}(u)$ is a $C_0$-function, and thus $\mathcal{G}^{(0)}$ is compact.
\end{proof}

Next, we record the computation of the core of the reduced $L^p$-operator algebra of a twisted groupoid. Its proof is almost identical to the untwisted case; see \cite[Proposition 5.1]{ChoGarThi_rigidity_2021}.

\begin{prop}\label{core = CX for p not 2}
    Let $(\mathcal{G},\Sigma)$ be a twist over an \'etale groupoid with compact unit space and let $p\in [1,\infty)\setminus \{2\}$. Then $\textup{core}(F_\lambda^p(\mathcal{G},\Sigma))= C(\mathcal{G}^{(0)})$.
\end{prop}
If $p=2$ and if $\mathcal{G}$ is additionally assumed to be Hausdorff and effective, then $C_0(\mathcal{G}^{(0)})\subseteq C^*_\lambda(\mathcal{G},\Sigma)$ is a Cartan subalgebra in the sense of \cite{Ren_Cartan_2008}. In the same paper, Renault showed that for any Cartan pair $(A, B)$, there is a canonical twist over an effective Hausdorff \'etale groupoid $(\mathcal{G},\Sigma)$, called the Weyl twist, such that $(A, B) \cong (C^*_\lambda(\mathcal{G},\Sigma), C_0(\mathcal{G}^{(0)}))$. For $p\neq 2$ and a unital $L^p$-operator algebra $A$, \autoref{core = CX for p not 2} can be taken as a motivation to investigate the rigidity of the inclusion $\textup{core}(A)\subseteq A$ with the methods developed by Renault. Carrying this out requires an analog of the normalizer action on the core spectrum as in \cite[Section 5]{ChoGarThi_rigidity_2021}, which is what we will do next.

\begin{df}\label{realizable homeos}
    Let $A$ be a Banach algebra and let $C_0(X)\subseteq A$ be an abelian C*-subalgebra. Let $\beta\colon U\rightarrow V$ be a partial homeomorphism between open subsets of $X$. We say that $\beta$ is \textit{realizable} if there are $a,b\in A$ satisfying the \textit{normalization conditions}
    \begin{itemize}
        \item[(N1)] $bC_0(X)_+a \subseteq C_0(X)_+ \,\,\textup{and}\,\,aC_0(X)_+b \subseteq C_0(X)_+$,
        \item[(N2)] $\supp(ba) = U \,\,\textup{and}\,\, \supp(ab) = V$,
    \end{itemize}
    as well as the \textit{realization identities}
    \begin{itemize}
        \item[(R1)] $bfa= (f\circ\beta)ba \,\,\, \textup{for all}\,\,\, f\in C_0(X)$, and
        \item[(R2)] $afb = (f\circ\beta^{-1})ab\,\,\, \textup{for all}\,\,\,f\in C_0(X)$.
    \end{itemize} In this case, the pair $(a,b)$ is called an \textit{admissible pair} that realizes $\beta$, and we write $\underline{\alpha}_{(a,b)} = \beta$. We denote the set of realizable partial homeomorphisms on $X$ by $\mathcal{R}_X(A)$ and the set of admissible pairs by $\textup{Adm}_{X}(A)$.
\end{df}

\begin{rem}\label{adm prod flip}
By \cite[Proposition 3.2]{ChoGarThi_rigidity_2021}, the set $\mathcal{R}_X(A)$ is a unital inverse subsemigroup of $\textup{Homeo}_{\textup{par}}(X)$ with $\underline{\alpha}_{(a,b)}\circ \underline{\alpha}_{(c,d)}$ realized by $(ac,db)$ and $\underline{\alpha}_{(a,b)}^{-1}$ realized by $(b,a)$. 
\end{rem}

For a twist $(\mathcal{G},\Sigma)$ over an \'etale groupoid with unit space $X$, there are two canonical inverse semigroup actions on $X$, namely the bisection action $\alpha\colon \mathcal{B}(\mathcal{G},\Sigma) \curvearrowright X$, and the action by realizable partial homeomorphisms $\mathcal{R}_X(F_\lambda^p(\mathcal{G},\Sigma)) \curvearrowright X$. In the following, we compare the corresponding groupoids of germs. Recall the definition of the involution of a section $\mathcal{G}\to \Sigma$ from \autoref{section operations}.

\begin{prop}\label{bis realizable}
    Let $(\mathcal{G},\Sigma)$ be a twist over an \'etale groupoid with unit space $X$. Let $S\in \mathcal{B}(\mathcal{G},\Sigma)$, let $\xi\in C_0(\mathsf{r}(S))$, and let $\eta\in C_0(\mathsf{d}(S))$. Then 
    \begin{equation}\label{ind normal eq}
        \ind_S^* *\xi*\ind_S = 
\xi\circ \alpha_S \in C_0(\mathsf{d}(S))\andeqn \ind_S *\eta*\ind_S^* = \eta\circ \alpha_S^{-1}\in C_0(\mathsf{r}(S)).\tag{3.1}
    \end{equation}
    If the open set $\mathsf{d}(S)$ is the strict support $\supp(h)$ of a positive function $h\in C_0(X)_+$, then, for any $p\in [1,\infty)$, the pair $(\ind_S *h,\,h*\ind_{S}^*)\in F_\lambda^p(\mathcal{G},\Sigma)\times F_\lambda^p(\mathcal{G},\Sigma)$ realizes the partial homeomorphism $\alpha_S\colon \mathsf{d}(S) \rightarrow \mathsf{r}(S)$. In particular, all partial homeomorphisms induced by bisections from $\mathcal{B}(\mathcal{G},\Sigma)$ are locally realizable.
\end{prop}

\begin{proof}
It suffices to prove the first identity.
For $x\in \mathsf{d}(S)$, it is easy to check that $Sx = \alpha_S(x)S$. For $\xi\in C_0(\mathsf{r}(S))$, using \autoref{S convo} repeatedly, we get 
    \begin{align*}
        (\ind_{S}^* *\xi*\ind_S)(x) = \xi(1_Sx1_S^{-1}) = \xi(\alpha_S(x)1_S1_S^{-1}) = \xi(\alpha_S(x)),
    \end{align*}
proving the identity in \eqref{ind normal eq}. For the remainder of the proof, choose $h\in C_0(X)_+$ with $\supp(h)= \mathsf{d}(S)$ and let $k:= h\circ \alpha_S^{-1}$ denote the resulting positive function with $\supp(k)= \mathsf{r}(S)$. Let $p\in [1,\infty)$. Since multiplication and inversion in $\Sigma$ are continuous, \autoref{S convo} implies that the sections $\ind_S*h$ and $h*\ind_S^*$ are continuous and supported on the open bisections $S$ and $S^{-1}$, respectively. Moreover, $\|\ind_S*h\|_\infty = \|\ind_S*h\|_\lambda = \|h\|_\infty$, and similarly for $h*\ind_{S}^*$. Thus, the pair belongs to $F_\lambda^p(\mathcal{G},\Sigma)\times F_\lambda^p(\mathcal{G},\Sigma)$. Applying the identity in \eqref{ind normal eq} to approximate identities, and taking limits, we deduce that \[\ind_{S}^* *\ind_{S}= \ind_{\mathsf{d}(S)} \andeqn \ind_S*\ind_{S}^*= \ind_{\mathsf{r}(S)}.\] Therefore, $h*\ind_{S}^**\ind_S*h = h\ind_{\mathsf{d}(S)}h= h^2$ is positive with strict support $\mathsf{d}(S)$. Likewise, since $h^2= h^2\ind_{\mathsf{d}(S)}$ and $\ind_{\mathsf{r}(S)}k^2 = k^2$, applying equation \eqref{ind normal eq} at the second step yields \[\ind_S*h*h*\ind_{S}^* = \ind_S*h^2*(\ind_{S}^**\ind_S)*\ind_{S}^*= (h^2\circ \alpha_S^{-1})*\ind_S*\ind_{S}^* = k^2.\] This shows the normalization condition (N2) in \autoref{realizable homeos}. We proceed to establish the realization identities and thereby also (N1). For $f\in C_0(X)$, the identity in \eqref{ind normal eq} still applies, although the function $(f\circ\alpha_S^{-1})\ind_{\mathsf{d}(S)}$ is not necessarily continuous at the boundary of $\mathsf{d}(S)$. Convoluting with $h$ from both sides yields (R1). For the second realization identity (R2), note that we have $hfh\in C_0(\mathsf{d}(S))$ and $(hfh)\circ \alpha_S^{-1} = k(f\circ \alpha_S^{-1})k\in C_0(\mathsf{r}(S))$. Thus, applying equation \eqref{ind normal eq} yields
    \begin{align*}
        (\ind_{S}*h)*f*(h*\ind_{S}^*)
        &= (hfh)\circ \alpha_S^{-1}\\
        &= (f\circ\alpha_S^{-1})k^2\\
        &= (f\circ \alpha_S^{-1})* \ind_{S}*h^2*\ind_{S}^*.
    \end{align*}
    This establishes (R2) and finishes the proof.
\end{proof}

For the next proposition, recall the notion of a wide inverse subsemigroup from \autoref{wide}.

\begin{prop}\label{supp inverse sg}
    Let $(\mathcal{G},\Sigma)$ be a twist over an \'etale groupoid. Set $\mathcal{S}_\Sigma := \{S\in \mathcal{B}(\mathcal{G},\Sigma)\colon S\,\,\textup{is}\,\, F_\sigma\}$. Then $\mathcal{S}_\Sigma$ is a wide unital inverse subsemigroup of $\mathcal{B}(\mathcal{G})$.
\end{prop}

\begin{proof}
Let $S,T\in \mathcal{S}_\Sigma$. For both $S^{-1}$ and $ST$, the domain map defines a homeomorphism onto the $F_\sigma$ sets $\textsf{r}(S)$ and $\alpha_{T^{-1}}(\textsf{r}(T)\cap \textsf{d}(S))$, respectively. Combined with the observation that continuous unitary sections are closed under involution and convolution, this shows that $\mathcal{S}_\Sigma$ is an inverse subsemigroup of $\mathcal{B}(\mathcal{G})$. Since open $F_\sigma$ sets and $\mathcal{B}(\mathcal{G},\Sigma)$ are closed under taking intersections, so is $\mathcal{S}_\Sigma$. Furthermore, $\mathcal{G}^{(0)}\in \mathcal{S}_\Sigma$ serves as a unit. It remains to show that $\mathcal{S}_\Sigma$ covers $\mathcal{G}$. Given $\gamma\in \mathcal{G}$, since $\Sigma$ is locally trivializable, there is an $S\in \mathcal{B}(\mathcal{G},\Sigma)$ such that $\gamma\in S$. Then $S$ is homeomorphic to $\mathsf{d}(S)\subseteq \mathcal{G}^{(0)}$. Since $\mathcal{G}^{(0)}$ is assumed to be locally compact Hausdorff, $\mathsf{d}(S)$ contains an open $F_\sigma$ neighborhood of $\mathsf{d}(\gamma)$. Thus there is a $F_\sigma$ open subset of $S$ containing $\gamma$ that belongs to $\mathcal{S}_\Sigma$ since $\mathcal{B}(\mathcal{G},\Sigma)$ is closed under open subsets. Hence $\mathcal{S}_\Sigma$ covers $\mathcal{G}$.
\end{proof}
\autoref{supp inverse sg} allows us to restrict our attention to bisections as in \autoref{bis realizable} from now on. If we further assume the groupoid $\mathcal{G}$ in \autoref{bis realizable} to be Hausdorff and effective, that is, for Weyl twists as in \autoref{df weyl twist abbr}, all realizable partial homeomorphisms come from the bisection action, as we show next.

\begin{prop}\label{eff realizable have bis support}
    Let $(\mathcal{G},\Sigma)$ be a Weyl twist. Let $p\in [1,\infty)$ and let $(a,b)$ be an admissible pair for $C_0(\mathcal{G}^{(0)})\subseteq F_\lambda^p(\mathcal{G},\Sigma)$. Set $S:= \supp(a)\cap \supp(b)^{-1}$. Then $S$ is an open bisection and we have $\underline{\alpha}_{(a,b)} = \alpha_S$.
\end{prop}

\begin{proof}
    The proof of \cite[Proposition 5.4]{ChoGarThi_rigidity_2021} does not require $p\neq 2$ nor compact unit space, and generalizes readily to twists.
\end{proof}


\begin{cor}\label{base reconstruction}
    Let $(\mathcal{G},\Sigma)$ be a Weyl twist with unit space $X$ and let $p\in [1,\infty)$. There is a canonical groupoid isomorphism $\kappa_{\textup{b}}\colon\mathcal{R}_{X}(F_\lambda^p(\mathcal{G},\Sigma)) \ltimes X\rightarrow \mathcal{G}$ given by \[\kappa_{\textup{b}}([\underline{\alpha}_{(a,b)},x]):= (\supp(a)\cap \supp(b)^{-1})x\] 
    for all $[\underline{\alpha}_{(a,b)},x]\in \mathcal{R}_{X}(F_\lambda^p(\mathcal{G},\Sigma)) \ltimes X$.
\end{cor}

\begin{proof}
Let $\mathcal{S}$ be as in the statement of \autoref{supp inverse sg}.
Combine \autoref{bis realizable} and \autoref{eff realizable have bis support} to deduce that $\mathcal{R}_{X}(F_\lambda^p(\mathcal{G},\Sigma)) \ltimes X$ equals $\alpha(\mathcal{S})\ltimes X$. Since the bisection action $\alpha$ is faithful by \autoref{alpha and eff}, and since $\mathcal{S}$ is wide by \autoref{supp inverse sg}, the groupoid homomorphism $\kappa_{\textup{b}}$ is well-defined. That it is an isomorphism follows from the canonical isomorphism $\mathcal{S}\ltimes_\alpha X \cong \mathcal{G}$ in \autoref{etale as trafo grpd}.
\end{proof}

It follows that to reconstruct the base groupoid of a Weyl twist $(\mathcal{G},\Sigma)$ from the inclusion $C_0(\mathcal{G}^{(0)})\subseteq F_\lambda^p(\mathcal{G},\Sigma)$, it suffices to consider admissible pairs parametrized by functions of the form $\ind_Sh$, as in \autoref{bis realizable}. We now aim to express such elements algebraically in terms of certain partial isometries.

\begin{df}\label{mp partial isom}
    Let $A$ be a unital Banach algebra. An element $u\in A$ is called an \textit{MP-partial isometry} if there is an element $v\in A$ such that
    \begin{enumerate}[label=(\roman*)]
        \item $\|u\|,\,\|v\| \leq 1,$
        \item $uvu = u \,\, \textup{and}\,\, vuv = v$, and
        \item $uv,\,vu\in A_\mathrm{h}.$
    \end{enumerate} We denote the set of MP-partial isometries in $A$ by $\textup{PI}(A)$. If it exists, $v$ is necessarily unique, and we refer to it as the \textit{MP-inverse} of $u$, and write $u^{\dagger}:= v$. 
\end{df}

For operators on Hilbert spaces, an MP-partial isometry is a partial isometry in the classical sense with the adjoint as its MP-inverse. See also \cite[Section 2.4]{BarKwaMcK_Banach_2023}.

\begin{rem}\label{mp preserved}
MP-partial isometries are preserved under unital contractive homomorphisms, because hermitian elements are by \autoref{hermitian preservation}, and because all other conditions directly follow from contractivity or multiplicativity.
\end{rem}
For further preservation results, we work with dual Banach algebras, as introduced by Runde; see \cite[Section 1]{Run_dualBanach_2001}.

\begin{df}\label{weak* defs}
    Let $A$ be a Banach algebra. If there is a predual Banach space $A_*$, that is, an isometric Banach space isomorphism $(A_*)^*\cong A$, we equip $A$ with the \textit{weak-* topology} of pointwise convergence on $A_*$ and denote weak-* limits by $\lim^{A_*}$ or ${A_*}\mbox{-}\lim$. If multiplication in $A$ is separately weak-* continuous, then the pair $(A,A_*)$ is called a \textit{dual Banach algebra}.
\end{df}

The following are the main examples of dual Banach algebras that we need.

\begin{eg}\label{refl E dual Banach}
    Let $E$ be a reflexive Banach space and let $\widehat{\otimes}$ denote the projective tensor product. Then $(\mathcal{B}(E),E\widehat{\otimes} E^*)$ is a dual Banach algebra by \cite[Section 1]{Run_dualBanach_2001}. \end{eg}

The double dual of a Banach algebra can be turned into a Banach algebra by equipping it with either the first or second Arens product.

\begin{df}\label{double dual df}
    Let $A$ be a Banach algebra. For all $a,\,b\in A^{**}$, there exist nets $(a_i)_i$ and $(b_j)_j$ in $A$ such that $\lim^{A^*}_i a_i = a$ and $\lim^{A^*}_j b_j = b$. Both the first Arens product $a\,\Box\,b:= \lim^{A^*}_i\lim^{A^*}_j(a_ib_j)$ and the second Arens product $a\cdot b:= \lim^{A^*}_j\lim^{A^*}_i(a_ib_j)$ turn $A^{**}$ into a unital Banach algebra that isometrically contains $A$ as a subalgebra. By default, we will use the second Arens product. We say that $A$ is \textit{Arens regular} if both Arens products agree.
\end{df}

\begin{eg} \label{refl E dual Banach 2}
    Let $p\in (1,\infty)$. Any $L^p$-space is super-reflexive in the sense of \cite[Definition 1]{Daw_arensBanach_2004}, and thus any $L^p$-operator algebra is Arens regular by \cite[Theorem 1]{Daw_arensBanach_2004} and \cite[Corollary 6.3]{CivYoo_Arensreg_1961}. For an Arens regular Banach algebra $A$, we have that $(A^{**},A^*)$ is a unital dual Banach algebra by \cite[Examples, Section 1]{Run_dualBanach_2001}.
\end{eg}
Given a twist $(\mathcal{G},\Sigma)$, recall the notation from \autoref{trivializable}, \autoref{S convo}, and the definition of $\mathcal{S}_\Sigma$ in \autoref{supp inverse sg}.
\begin{prop}\label{indicator as MP}
    Let $(\mathcal{G},\Sigma)$ be a twist over an \'etale groupoid. For $p\in [1,\infty)$, we have that $C_0(\mathcal{G}^{(0)})^{**}\subseteq F_\lambda^p(\mathcal{G},\Sigma)^{**}$ is a unital C*-subalgebra with unit $\ind_{\mathcal{G}^{(0)}}$. Moreover, for any $S\in \mathcal{S}_\Sigma$, the $1_S$-lift $\ind_S\in C_0(S)^{**}\subseteq F_\lambda^p(\mathcal{G},\Sigma)^{**}$ is an MP-partial isometry in $F_\lambda^p(\mathcal{G},\Sigma)^{**}$ with MP-inverse given by $\ind_S^*$.
\end{prop}

\begin{proof}
    Fix $p\in [1,\infty)$ and set $A:= F_\lambda^p(\mathcal{G},\Sigma)$. The indicator function $\ind_{\mathcal{G}^{(0)}}$ is a unit for the C*-algebra $C_0(\mathcal{G}^{(0)})^{**}$ and we may identify it with the $\mathcal{G}^{(0)}$-lifted section $\ind_{\mathcal{G}^{(0)}}\in A^{**}$. Similarly, let $S\in \mathcal{S}_\Sigma$ and set $U:=\mathsf{d}(S)$ and $V:=\mathsf{r}(S)$. Since $S$ is $F_\sigma$, choose $h\in C_0(\mathcal{G}^{(0)})_+$ such that $\supp(h)=U$. Then the sequence $(h^{\frac{1}{n}})_{n\in \N}$ is an approximate identity for $C_0(U)$ and thus we have $\ind_U = \lim^{A^*}_n h^{\frac{1}{n}}\in A^{**}$. By \autoref{bis realizable}, we have $\ind_Sh^{\frac{1}{n}}\in A$ for all $n\in \N$ and since multiplication is continuous in $A^{**}$, we obtain in the weak-* limit that $\ind_S\in A^{**}$. By \autoref{S convo}, we compute that $\ind_S^{*}\ind_S = \ind_{U}$ is an idempotent in $C_0(\mathcal{G}^{(0)})^{**}\subseteq A^{**}$ and thus, for all $t\in \R$, we have \[\|e^{it\ind_{U}}\| = \|\ind_{\mathcal{G}^{(0)}\setminus U} + e^{it}\ind_U\|_\infty = 1.\] Analogously, by \autoref{S convo}, we have $\ind_S\ind_S^* = \ind_{V}\in (A^{**})_{\mathrm{h}}$, as well as $\ind_S \ind_U= \ind_V\ind_S= \ind_S$ and $\ind_S^* \ind_V= \ind_U\ind_S^*= \ind_S^*$. Thus, all conditions in \autoref{mp partial isom} are met and this shows that $\ind_S\in \textup{PI}(A^{**})$ with $\ind_S^\dagger = \ind_S^*$.
\end{proof}
\autoref{bis realizable} and \autoref{indicator as MP} motivate the following definitions.

\begin{df}\label{nor part iso}
    Let $A$ be a Banach algebra and let $C_0(X)\subseteq A$ be an abelian C*-subalgebra that contains a bounded approximate identity of $A$. A \textit{$C_0(X)$-normalizing partial isometry} is an MP-partial isometry $u\in \textup{PI}(A^{**})$ for which there exists a partial homeomorphism $\beta\colon U\rightarrow V$ in $\mathcal{R}_X(A)$ such that 
    \begin{enumerate}[label=(\roman*)]
        \item $u^\dagger u = \ind_{U}$ and $uu^\dagger = \ind_{V}$,
        \item $u^\dagger fu = f\circ \beta \,\,\,\textup{for all}\,\,\, f\in C_0(V)$, and
        \item $ugu^\dagger = g\circ \beta^{-1} \,\,\,\textup{for all}\,\,\, g\in C_0(U)$.
    \end{enumerate}
    We denote the set of $C_0(X)$-normalizing partial isometries by $\textup{PI}_{X}(A^{**})$. 
\end{df}
For $u\in \textup{PI}_{X}(A^{**})$, note that, by (i), the sets $U$ and $V$ are uniquely determined by $u$ and since $C_0(V)$ separates points in $V$, it follows from (ii) that $\beta$ is uniquely determined by $u$. We therefore write $\beta_u$ for the partial homeomorphism associated to $u\in \textup{PI}_{X}(A^{**})$.

\begin{df}\label{normalizer def}
    Let $A$ be a Banach algebra and let $C_0(X)\subseteq A$ be an abelian C*-subalgebra that contains a bounded approximate identity of $A$. We define the set of \textit{spatial normalizers} with respect to $C_0(X)$ as 
    \[ \mathcal{SN}_X(A) :=\left\{n\in A \colon\,\,\,
    \begin{aligned}
    &\textup{there are}\,\, u\in \textup{PI}_{X}(A^{**})\,\, \textup{and}\,\, h\in C_0(X)_+\\ &\textup{such that}\,\,\,\supp(h) = \textup{dom}(\beta_u)\,\,\textup{and}\,\, n= uh
    \end{aligned}
    \right\}.\]
    Given $n\in \mathcal{SN}_X(A)$, we refer to a concrete choice of $u$ and $h$ with $n=uh$ as a \textit{product representation} of $n$. When $(\mathcal{G},\Sigma)$ is a twist over an \'etale groupoid and $p\in [1,\infty)$, then we abbreviate $\mathcal{SN}_{\mathcal{G}^{(0)}}(F_\lambda^p(\mathcal{G},\Sigma))$ to $\mathcal{SN}_\lambda^p(\mathcal{G},\Sigma)$.
\end{df}
For the canonical inclusion $C_0(\mathcal{G}^{(0)})\subseteq F_\lambda^p(\mathcal{G},\Sigma)$, \autoref{indicator as MP} yields that any $\ind_Sh$ as in \autoref{bis realizable} is a spatial normalizer by design. For general inclusions $C_0(X)\subseteq A$, the advantage of spatial normalizers compared to admissible pairs is that they form a more tractable subset of representatives for realizable partial homeomorphisms in $A$. 

\begin{lma}\label{normalizing inv sg}
    Let $A$ be a Banach algebra and let $C_0(X)\subseteq A$ be an abelian C*-subalgebra that contains a bounded approximate identity of $A$. Then $\textup{PI}_{X}(A^{**})$ is a unital inverse semigroup under multiplication and with involution given by taking MP-inverses.
\end{lma}

\begin{proof}
The definition of a $C_0(X)$-normalizing partial isometry is symmetric under taking MP-inverses, that is, if $u\in\textup{PI}_{X}(A^{**})$, then $u^\dagger\in \textup{PI}_{X}(A^{**})$ with $\beta_{u^\dagger} = \beta_u^{-1}$. It is left to show that $\textup{PI}_{X}(A^{**})$ is closed under multiplication. For $u,\,v\in\textup{PI}_{X}(A^{**})$, note that $u^\dagger u= \ind_{\textup{dom}(\beta_{u})}$ and $vv^\dagger = \ind_{\textup{ran}(\beta_{v})}$ commute. As a result, we compute \[uv(v^\dagger u^\dagger)uv = u(vv^\dagger)( u^\dagger u)v = u(u^\dagger u)(vv^\dagger)v = uv\] and likewise $v^\dagger u^\dagger(uv)v^\dagger u^\dagger = v^\dagger u^\dagger$ so that $uv$ is an MP-partial isometry with MP-inverse $v^\dagger u^\dagger$. Iterating the realization identities for $u$ and $v$ shows that the partial homeomorphism $\beta_{uv}= \beta_u\circ \beta_v$ meets the additional normalizing conditions, and thus $uv\in\textup{PI}_{X}(A^{**})$.
\end{proof}

\begin{lma}\label{well defined sn homeo}
    Let $A$ be a Banach algebra and let $C_0(X)\subseteq A$ be an abelian C*-subalgebra that contains a bounded approximate identity of $A$. For every spatial normalizer $n\in\mathcal{SN}_X(A)$ with product representations $n=u_1h_1 = u_2h_2$, we have $\beta_{u_1}=\beta_{u_2}$. Moreover, the set of spatial normalizers $\mathcal{SN}_X(A)$ is closed under multiplication.
\end{lma}

\begin{proof}
    Let $\beta_{u_1}\colon U_1\rightarrow V_1$ and $\beta_{u_2}\colon U_2\rightarrow V_2$ be as in \autoref{normalizer def}. Then $h_2 = (u_2^\dagger u_1)h_1 = (u_2^\dagger u_1)(u_1^\dagger u_2)h_2$, so that $\ind_{V_1}\circ \beta_{u_2}$ is supported on $U_2$ and thus $V_2\subseteq V_1$. Analogous computations establish $U_1=U_2$ and $V_1=V_2$. Furthermore, for all $f\in C_0(V_1)$, we have 
    \[ n(f\circ \beta_{u_2}) = u_2h_2(f\circ \beta_{u_2}) = (h_2\circ \beta_{u_2}^{-1})fu_2= f(h_2\circ \beta_{u_2}^{-1})u_2 = fn = n(f\circ \beta_{u_1}).\] Multiplying both sides from the left by $u_1^\dagger$ yields $h_1(f\circ \beta_{u_2})= h_1(f\circ \beta_{u_1})$. Since $\supp(h_1)=U_1$, we get $\beta_{u_1}=\beta_{u_2}$, as claimed.
    
    Finally, we show that $\mathcal{SN}_X(A)$ is closed under products. For product representations $uh,\,vk \in \mathcal{SN}_X(A)$, we claim that the product $(uh)(vk) = uv(h\circ \beta_{v})k$ is a product representation. Indeed, we have $uv\in\textup{PI}_{X}(A^{**})$ by \autoref{normalizing inv sg} and
    \[\supp((h\circ \beta_v)k) = \beta_{v}^{-1}(\supp(h))\cap \supp(k) = \textup{dom}(\beta_u\circ \beta_v).\qedhere\] 
\end{proof}
We now introduce the concept of a saturated inclusion, which will allow us to define an action by the spatial normalizers.

\begin{df}\label{approx def}
    Let $A$ be a Banach algebra and let $C_0(X)\subseteq A$ be an abelian C*-subalgebra that contains a bounded approximate identity of $A$. For $n \in \mathcal{SN}_X(A)$, we set $\underline{\alpha}_{n}=\beta_u$ for any product representation $n=uh$. This is unambiguous by \autoref{well defined sn homeo}. 
    We say that the inclusion $C_0(X)\subseteq A$ is \textit{saturated} if any spatial normalizer $n\in \mathcal{SN}_X(A)$ realizing the identity map $\underline{\alpha}_{n} = \textup{id}_{U}$ satisfies $n\in C_0(X)$. For two spatial normalizers $m,n\in \mathcal{SN}_X(A)$, we write $m\approx n$ if $\textup{dom}(\underline{\alpha}_{m})=\textup{dom}(\underline{\alpha}_{n})$ and if there are $f_1,f_2\in C_0(X)_+$ with strict support $\textup{dom}(\underline{\alpha}_{m})$ and $mf_1 = nf_2$. We write $[n]_\approx$ for the set of all $m\in \mathcal{SN}_X(A)$ with $n\approx m$.
\end{df}

\begin{prop}\label{sn inv sem prop}
    Let $A$ be a Banach algebra and let $C_0(X)\subseteq A$ be an abelian C*-subalgebra that contains a bounded approximate identity of $A$. The relation $\approx$ in \autoref{approx def} is an equivalence relation on $\mathcal{SN}_X(A)$. Moreover, for $m,n\in \mathcal{SN}_X(A)$, the operation $[m]_\approx\cdot [n]_\approx := [mn]_\approx$ defines a product on the set of $\approx$-equivalence classes. If $C_0(X)\subseteq A$ is, in addition, saturated, then the assignment $[uh]_\approx^* := [hu^\dagger]_\approx$ for $uh\in \mathcal{SN}_X(A)$ also defines an involution and turns the set of $\approx$-equivalence classes into a unital inverse se\-mi\-group that acts on $X$ by realizable partial homeomorphisms.
\end{prop}

\begin{proof}
    The relation $\approx$ is clearly symmetric and reflexive. It is easy to see that it is transitive. Let $m_1,\,m_2,\,n_1,\,n_2\in\mathcal{SN}_X(A)$ satisfy $m_1 \approx m_2$ and $n_1\approx n_2$. 
    Let $\underline{\alpha}_{m}\colon U_m\to V_m$ and $\underline{\alpha}_{n}\colon U_n\to V_n$ be the associated partial homeomorphisms as in \autoref{realizable homeos}. Note that $U_m=U_n$ since $m\approx n$, and we denote this open set simply by $U$.
    Choose functions $f_{m_1},f_{m_2}$ and $f_{n_1},f_{n_2}\in C_0(X)_+$ satisfying $\supp(f_{m_1})=\supp(f_{m_2})=U$ and $m_1f_{m_1}=m_2f_{m_2}$, and $\supp(f_{n_1})=\supp(f_{n_2})=U$ and $n_1f_{n_1}=n_2f_{n_2}$, respectively. These are witnesses of $m_1\approx m_2$ and $n_1\approx n_2$, respectively. By \autoref{adm prod flip}, the products $m_1n_1$ and $m_2n_2$ realize the same partial homeomorphism $\underline{\alpha}_{m}\circ\underline{\alpha}_{n}$ with domain $\underline{\alpha}_{n}^{-1}(U_m)\cap U_n$. Moreover, the functions $f_{m_i}\circ \underline{\alpha}_{n}$, for $i=1,2$, have this set as their strict supports and satisfy 
    \[m_1n_1(f_{m_1}\circ \underline{\alpha}_{n})f_{n_1} = m_1f_{m_1}n_1f_{n_1} = m_2f_{m_2}n_2f_{n_2} = m_2n_2(f_{m_2}\circ \underline{\alpha}_{n})f_{n_2}.\]
    This shows that $m_1n_1\approx m_2n_2$ and hence the product of classes is well-defined. It follows from \autoref{normalizing inv sg} that the class $[\ind_Xh]_\approx$ represented by any $h\in C_0(X)_+$ with $\supp(h)=X$ serves as a multiplicative unit.\par Now we assume that $C_0(X)\subseteq A$ is a saturated inclusion. We proceed to show that the claimed involution is well-defined. Let $n_1\approx n_2$ be as before and choose product representations $u_1h_1$ and $u_2h_2$ with corresponding joint domain $U$, joint partial homeomorphism $\underline{\alpha}_{n}$, and functions $f_{1},f_{2}\in C_0(X)_+$ such that $\supp(f_1)=\supp(f_2)=U$ and $n_1f_1=n_2f_2$. We aim to show that $h_1u_1^\dagger\approx h_2u_2^\dagger$. Note that $h_iu_i^\dagger = u_i^\dagger (h_i\circ \beta_{u_i}^{-1})$ are product representations of spatial normalizers for $i=1,2$, and as a result the pairs $(h_iu_i^\dagger,n_i)$, for $i=1,2$, are admissible, and both realize $\underline{\alpha}_n^{-1}$ by \autoref{adm prod flip}. Again by \autoref{adm prod flip}, the pair $(u_1h_1h_2u_2^\dagger,u_2h_2h_1u_1^\dagger)$ is admissible and realizes $\textup{id}_{\underline{\alpha}_n(U)}$. Since the inclusion $C_0(X)\subseteq A$ is saturated, we obtain that both elements of the latter pair are in the subalgebra $C_0(X)$. We define \[g_1:= (f_1\circ \underline{\alpha}_{n}^{-1})u_1h_1h_2u_2^\dagger \andeqn g_2:= (f_2\circ \underline{\alpha}_{n}^{-1})u_2h_2h_1u_1^\dagger.\] Note that both functions have $\underline{\alpha}_{n}(U)$ as strict support. Using the realization identities for the pair $(u_1h_1,h_1u_1^\dagger)$ and the equality $u_1h_1f_1=u_2h_2f_2$, we get \[h_1u_1^\dagger g_1 = h_1u_1^\dagger u_1h_1f_1h_2u_2^\dagger = h_1u_1^\dagger u_2h_2f_2h_2u_2^\dagger = f_2h_1u_1^\dagger (u_2h_2h_2u_2^\dagger).\] Analogously, the realization identities for the admissible pair $(u_2h_2,h_2u_2^\dagger)$ yield \[h_2u_2^\dagger g_2 = h_2u_2^\dagger u_2h_2f_2h_1u_1^\dagger = f_2((u_2h_2h_2u_2^\dagger)\circ \underline{\alpha}_n)h_1u_1^\dagger = f_2h_1u_1^\dagger (u_2h_2h_2u_2^\dagger).\] This shows that $h_1u_1^\dagger\approx h_2u_2^\dagger$ and hence $[n_1]_\approx^*$ is well-defined.\par Finally, for $uh\in \mathcal{SN}_X(A)$, we have $hu^\dagger uh = h^2$ and $uhhu^\dagger = h^2\circ \underline{\alpha}_{n}^{-1}$. As a result, we deduce that $(uh)(hu^\dagger)(uh)\approx uh$ and $(hu^\dagger)(uh)(hu^\dagger) \approx hu^\dagger$. This concludes the proof that the set of $\approx$-equivalence classes is a unital inverse se\-mi\-group.
\end{proof}

\begin{df}\label{sn action}
    Let $A$ be a Banach algebra and let $C_0(X)\subseteq A$ be a saturated inclusion of an abelian C*-subalgebra that contains a bounded approximate identity of $A$. We denote the inverse semigroup of $\approx$-equivalence classes in \autoref{sn inv sem prop} by $\mathcal{SN}_{X,\approx}(A)$ and the induced unital semigroup homomorphism by $\underline{\alpha}_\approx \colon \mathcal{SN}_{X,\approx}(A)\rightarrow \mathcal{R}_X(A)$. We refer to $\underline{\alpha}_\approx$ as the \textit{spatial normalizer action} and abbreviate it to $\underline{\alpha}\colon \mathcal{SN}_{X}(A)\curvearrowright X$. We also write
    \[[n,x]\in \mathcal{SN}_X(A)\ltimes_{\underline{\alpha}} X \textup{\quad instead of\quad} [[n]_\approx,x]\in \mathcal{SN}_{X,\approx}(A)\ltimes_{\underline{\alpha}_\approx} X\] 
    to represent equivalence classes in the associated transformation groupoid as in \autoref{germ grpd}. We define the \textit{$\T$-sensitive topology} on $\mathcal{SN}_X(A)\ltimes_{\underline{\alpha}} X$ as the topology generated by basic open subsets of the form $\{[\zeta n,x]\colon \zeta\in Z,\,x\in V\}$ for $n\in \mathcal{SN}_X(A)$, $Z\subseteq \T$ open, and $V\subseteq \textup{dom}(\underline{\alpha}_{n})$ open. In contrast to the \'etale topology in \autoref{germ grpd}, we equip the transformation groupoid $\mathcal{SN}_X(A)\ltimes_{\underline{\alpha}} X$ with the $\T$-sensitive topology instead.
\end{df}

\begin{prop}\label{saturated}
    Let $(\mathcal{G},\Sigma)$ be a twist over a Hausdorff \'etale groupoid with unit space $X$ and let $p\in [1,\infty)$. Then the inclusion $C_0(X)\subseteq F_\lambda^p(\mathcal{G},\Sigma)$ is saturated if and only if $\mathcal{G}$ is effective. If $\mathcal{G}$ is effective, then the strict support of a spatial normalizer is an open bisection.
\end{prop}

\begin{proof}
    Suppose that $\mathcal{G}$ is not effective. Then the bisection action is not faithful by \autoref{alpha and eff}, and by \autoref{supp inverse sg} we can choose $S\in \mathcal{S}_\Sigma$ with $S\not \subseteq X$ and $h\in C_0(X)_+$ with $\supp(h)= \mathsf{d}(S)$ such that $\alpha_S = \textup{id}_{\mathsf{d}(S)}$. Then the corresponding spatial normalizer $\ind_Sh \in \mathcal{SN}_\lambda^p(\mathcal{G},\Sigma)$ in \autoref{bis realizable} satisfies $\underline{\alpha}_{\ind_Sh} = \textup{id}_{\mathsf{d}(S)}$. In this case, the identity $\supp(\ind_Sh) = S\not \subseteq X$ shows that the inclusion $C_0(X)\subseteq F_\lambda^p(\mathcal{G},\Sigma)$ is not saturated. Conversely, let $\mathcal{G}$ be effective and let $uh \in \mathcal{SN}_\lambda^p(\mathcal{G},\Sigma)$ be a spatial normalizer. Then the identity $h^2 = (hu^\dagger)*(uh)\in C_0(X)$ shows that $\supp(uh)= \supp(hu^\dagger)^{-1}$. Since $\mathcal{G}$ is effective, we deduce from \autoref{eff realizable have bis support} that $\supp(uh)$ is an open bisection with $\underline{\alpha}_{uh} = \alpha_{\supp(uh)}$, showing the last assertion in the statement. Now assume $\underline{\alpha}_{uh} = \textup{id}_{\supp(h)}$. Since $\alpha$ is faithful by \autoref{alpha and eff}, we have $\supp(uh)\subseteq X$, that is, $uh\in C_0(X)$. This shows that the inclusion $C_0(X)\subseteq F_\lambda^p(\mathcal{G},\Sigma)$ is saturated.\qedhere
\end{proof}
Moreover, the equivalence of $(1)$ and $(2)$ in \cite[Theorem A]{BarKwaMcK_BanachII_2024} shows that the inclusion in \autoref{saturated} is saturated if and only if it is maximal abelian.

\begin{prop}\label{Weyl twist lp prop}
    Let $A$ be a Banach algebra and let $C_0(X)\subseteq A$ be a saturated inclusion of an abelian C*-subalgebra that contains a bounded approximate identity of $A$. Then the canonical map $\underline{\alpha}\colon \mathcal{SN}_X(A)\ltimes_{\underline{\alpha}} X \to \im(\underline{\alpha}_\approx)\ltimes X$ given by $\underline{\alpha}([n,x]) := [\underline{\alpha}_{n},x]$
    for $[n,x]\in \mathcal{SN}_X(A)\ltimes_{\underline{\alpha}} X$, is a Weyl twist in the $\T$-sensitive topology described in \autoref{sn action}. 
\end{prop}

\begin{proof}
    The homomorphism $\underline{\alpha}$ is surjective by design. Its image is an effective \'etale groupoid as in \autoref{germ grpd} with unit space homeomorphic to $X$.
    
    We claim that $\ker(\underline{\alpha}) = \{[n,x]\in \mathcal{SN}_X(A)\ltimes_{\underline{\alpha}} X\colon [\underline{\alpha}_n,x]\in X\}$ is canonically isomorphic to the group bundle $\T\times X$ via the map 
    $[f,x] \mapsto  \left(\frac{f(x)}{|f(x)|},x\right)$. Note that an element $[n,x]$ is sent to a unit if and only if there is an open neighborhood $U_x\subseteq \textup{dom}(\underline{\alpha}_n)$ around $x$ such that $\underline{\alpha}_n\vert_{U_x} = \textup{id}_{U_x}$. Without loss of generality, we can assume that there is $h\in C_0(X)_+$ with $\supp(h)= U_x$ representing an idempotent in the inverse semigroup $\mathcal{SN}_{X,\approx}(A)$. Thus the germs $[n,x]$ and $[nh,x]$ agree and $\underline{\alpha}_{nh} = \textup{id}_{U_x}$. Since $C_0(X)\subseteq A$ is saturated, we have $nh= f\in C_0(X)$ with $x\in \supp(|f|^2)$. Thus 
    \[\ker(\underline{\alpha}) = \{[f,x]\colon f\in C_0(X),\, x\in \supp(f)\}.\] Moreover, for all $[n,x]\in \mathcal{SN}_X(A)\ltimes_{\underline{\alpha}} X$ and $f\in C_0(X)$, we get \[[f,\underline{\alpha}_n(x)][n,x]= [nf,x] = [n,x][f,x].\] Furthermore, since $f\approx \lambda f$ for all $f\in C_0(X)$ and $\lambda >0$, while $f\not \approx \zeta f$ for $\zeta\in \T\setminus\{1\}$, we deduce that the isotropy group at $x\in X$ is isomorphic to $\T$. This proves the claim and shows that $\underline{\alpha}$ defines the desired algebraic extension. To see that the homomorphism $\underline{\alpha}$ is locally trivializable, note that for every $n\in \mathcal{SN}_X(A)$, the $\approx$-classes $\{[\zeta n]_\approx\colon \zeta \in \T\}$ represent the same partial homeomorphism. Finally, let $\{[\beta,x]\colon x\in U\}$ be a basic open subset for a partial homeomorphism $\beta\in \im(\underline{\alpha}_\approx)$ and an open subset $U\subseteq \textup{dom}(\beta)$. Choose a spatial normalizer representative $n$ such that $\underline{\alpha}_n= \beta$ and note that \[\underline{\alpha}^{-1}(\{[\beta,x]\colon x\in U\}) = \{[\zeta n,x]\colon \zeta\in \T,\,x\in U\}\] is open in the $\T$-sensitive topology. This shows continuity of the twist homomorphism $\underline{\alpha}$ and completes the proof that $\underline{\alpha}$ defines a Weyl twist.
\end{proof}

\begin{nota}\label{Weyl twist lp}
    In the context of \autoref{Weyl twist lp prop}, we denote the associated Weyl twist by $\textup{Weyl}_X(A) := (\im(\underline{\alpha}_\approx)\ltimes X, \mathcal{SN}_X(A)\ltimes_{\underline{\alpha}} X).$
\end{nota}
We continue to show that for a Weyl twist $(\mathcal{G},\Sigma)$ and $p\in [1,\infty)$, the twist $\textup{Weyl}_X(F_\lambda^p(\mathcal{G},\Sigma))$ agrees with the original Weyl twist $(\mathcal{G},\Sigma)$. 

\begin{thm}\label{Normalizer twist rec}
    Let $(\mathcal{G},\Sigma)$ be a Weyl twist with unit space $X$ and let $p\in [1,\infty)$. Let $[n,x]\in \mathcal{SN}_\lambda^p(\mathcal{G},\Sigma)\ltimes_{\underline{\alpha}} X$ and let $S_x\subseteq \supp(n)$ be an open subset containing $\supp(n)x$ with $S_x\in \mathcal{S}_\Sigma$. Then the assignment 
    \begin{equation}\label{weyl rec eq}
        \kappa\colon \textup{Weyl}_X(F_\lambda^p(\mathcal{G},\Sigma))\rightarrow (\mathcal{G},\Sigma) \textup{\quad given by \quad} \kappa([n,x]):=\frac{j_{n}(S_xx)}{\sqrt{(j_n^**j_n)(x)}}1_{S_xx}\tag{3.2}
    \end{equation} for all $[n,x]\in \mathcal{SN}_X(A)\ltimes_{\underline{\alpha}} X$, is a twist isomorphism in the sense of \autoref{twist hom}. The induced isomorphism between the base groupoids $\kappa_{\textup{b}}\colon \im(\underline{\alpha}_\approx)\ltimes X \rightarrow \mathcal{G}$ satisfying $\pi_\mathcal{G}\circ \kappa = \kappa_{\textup{b}} \circ \underline{\alpha}$ agrees with the one in \autoref{base reconstruction}. 
\end{thm}

\begin{proof}
    To see that the assignment in \eqref{weyl rec eq} takes values in $\Sigma$, let $n\in \mathcal{SN}_\lambda^p(\mathcal{G},\Sigma)$. By \autoref{saturated}, we have that $\supp(n)$ is an open bisection. One readily checks that $|j_n(S_xx)|^2 = (j_n^**j_n)(x)$, and thus the fraction in front of $1_{S_xx}$ in \eqref{weyl rec eq} is in $\T$. We claim that $\kappa$ is a well-defined map, that is, that the expression defining $\kappa([n,x])$ only depends on $[n,x]$. By definition, if $[[n_1]_\approx,x] = [[n_2]_\approx,x]\in \mathcal{SN}_\lambda^p(\mathcal{G},\Sigma)\ltimes_{\underline{\alpha}} X$ then there are $f_i\in C_0(X)_+$ with $f_i(x)>0$ for $i=1,2$ and $n_1f_1=n_2f_2$. In particular, the open bisection supports agree on a neighborhood of the element $\supp(n_i)x$, and for any choice of subsets $S_{x,i}\in \mathcal{S}_\Sigma$, that exist since $\Sigma$ is locally trivializable, we have that $S_{x,1}x= S_{x,2}x$ agree as elements in $\mathcal{G}$. Then the computation 
    \begin{align*}
        \frac{j_{n_1}(S_{x,1}x)}{\sqrt{(j_{n_1}^**j_{n_1})(x)}} &= \frac{j_{n_1}(S_{x,1}x)f_1(\mathsf{d}(S_{x,1}x))}{\sqrt{f_1(x)}\sqrt{(j_{n_1}^**j_{n_1})(x)}\sqrt{f_1(x)}} = \frac{j_{n_1f_1}(S_{x,1}x)}{\sqrt{(j_{n_1f_1}^**j_{n_1f_1})(x)}}  \\& = \frac{j_{n_2f_2}(S_{x,2}x)}{\sqrt{(j_{n_2f_2}^**j_{n_2f_2})(x)}} = \frac{j_{n_2}(S_{x,2}x)}{\sqrt{(j_{n_2}^**j_{n_2})(x)}}
    \end{align*}
    shows that $\kappa$ is well-defined. Restricted to $\ker(\underline{\alpha})$, the map $\kappa$ simplifies to \[\kappa\vert_{\ker(\underline{\alpha})}([f,x]) = \frac{f(x)}{|f(x)|}x.\] This restriction agrees with the isomorphism onto $\ker(\pi_\mathcal{G})= \T\times X$ in \autoref{Weyl twist lp prop}. Further, given $[n,x]\in \mathcal{SN}_\lambda^p(\mathcal{G},\Sigma)\ltimes_{\underline{\alpha}} X$, choose a subset $S_x\subseteq \supp(n)$ with $S_x\in\mathcal{S}$. For $\zeta\in \T$, using $\supp(\zeta n) = \supp(n)$ at the second step, we get \[\kappa([\zeta n,x])= \frac{j_{\zeta n}(S_xx)}{\sqrt{(j_{\zeta n}^**j_{\zeta n})(x)}}1_{S_xx} = \frac{\zeta\cdot j_{n}(S_xx)}{\sqrt{(j_n^**j_n)(x)}}1_{S_xx} = \zeta \kappa([n,x]).\] That is, the map $\kappa$ in \eqref{weyl rec eq} induces a well-defined map $\kappa_{\textup{b}}\colon \im(\underline{\alpha}_\approx)\ltimes X\rightarrow \mathcal{G}$ satisfying $\kappa_{\textup{b}} \circ \underline{\alpha} = \pi_\mathcal{G}\circ \kappa$. We obtain $\kappa_{\textup{b}}([\underline{\alpha}_n,x]) = S_xx$ for all $[\underline{\alpha}_n,x]\in \im(\underline{\alpha}_\approx)\ltimes X$. Recall that by \autoref{bis realizable} and \autoref{eff realizable have bis support}, we have $\mathcal{R}_{X}(F_\lambda^p(\mathcal{G},\Sigma)) \ltimes X=\alpha(\mathcal{S})\ltimes X = \im(\underline{\alpha}_\approx)\ltimes X$. Hence, without loss of generality, all germs of realizable partial homeomorphisms can be realized by spatial normalizers and thus $\kappa_{\textup{b}}$ agrees with the groupoid isomorphism in \autoref{base reconstruction}, as claimed.
    
    Finally, for $\kappa$ to be a twist isomorphism, it suffices to show that it is a homomorphism. Let $[m,y],[n,x]\in\mathcal{SN}_X(A)\ltimes_{\underline{\alpha}} X$ satisfy $y=\underline{\alpha}_n(x)$, and let $T_{y},\,S_x\in \mathcal{S}_\Sigma$ be corresponding bisection choices. We have $j_{mn}(T_yS_xx) = j_m(T_yy)j_n(S_xx)$, while for the denominator we compute \[(j_{mn}^**j_{mn})(x) = (j_n^**j_m^**j_m*j_n)(x) = (j_m^**j_m)(\underline{\alpha}_n(x))(j_n^**j_n)(x).\] Therefore, we have $\kappa([mn,x]) = \kappa([m,\underline{\alpha}_n(x)])\kappa([n,x])$ and this shows that $\kappa$ is a twist isomorphism.\qedhere
\end{proof}
Our goal is to describe, for $p\neq 2$, all unital contractive homomorphisms between reduced $L^p$-operator algebras associated to Weyl twists, generalizing the case of isometric isomorphisms from \cite{HetOrt_twistlp_2022}. For this, we turn to spatial normalizers.\par Recall that a measure space $(\Omega,\mathcal{A},\mu)$ is \textit{localizable} if the Boolean algebra $\mathcal{A}/\mathcal{N}$ modulo $\mu$-nullsets is a complete lattice and any class of infinite measure admits a subclass of finite measure. See further \cite[Section 3]{GarThi_convolution_2022}. Localizable measure spaces are the canonical setup to introduce Radon-Nikodym derivatives.

\begin{df}\label{ps aut}
    Let $(\Omega,\mathcal{A},\mu)$ be a localizable measure space. For $D\in \mathcal{A}$, set $\mathcal{A}_D := \{C\cap D \colon C\in \mathcal{A}\}$ and $\mu\vert_{D}(C):= \mu(C\cap D)$. A \textit{partial set automorphism} on $\Omega$ consists of a choice of subsets $D_\Theta,D_{\Theta^*}\in \mathcal{A}$ and a map $\Theta\colon \mathcal{A}_{D_{\Theta^*}}\rightarrow \mathcal{A}_{D_{\Theta}}$ with $\Theta(\mathcal{N}_{D_{\Theta^*}})\subseteq \mathcal{N}_{D_{\Theta}}$ that induces an automorphism of the associated Boolean algebras up to nullsets. Given a partial set automorphism $\Theta$, we denote the inverse map by $\Theta^*$ and for a measurable set $C\in \mathcal{A}$, the assignment $\ind_C \mapsto \ind_{\Theta(C\cap D_{\Theta^*})}$ defines an operator $T_\Theta\colon L^0(\mu)\rightarrow L^0(\mu\vert_{D_{\Theta}})$. The measures $\mu\vert_{D_{\Theta^*}}\circ \Theta^*$ and $\mu\vert_{D_{\Theta}}$ on $\mathcal{A}_{D_\Theta}$ are equivalent, since they have the same null-sets by design, and we denote the associated Radon-Nikodym derivative $\frac{\textup{d}\mu\vert_{D_{\Theta^*}}\circ \Theta^*}{\textup{d}\mu\vert_{D_{\Theta}}}$ by $\mathrm{d}_{\mu,\Theta}$; see \cite[Theorem 2.8]{GarThi_convolution_2022}.
\end{df}

\begin{df}\label{spi}
    Let $(\Omega,\mathcal{A},\mu)$ be a localizable measure space and let $p\in [1,\infty)$. An operator $s\in \mathcal{B}(L^p(\mu))$ is said to be a \textit{spatial partial isometry} if there exist a partial set automorphism $\Theta$ on $\Omega$ and a measurable function $w\colon D_\Theta \rightarrow \mathbb{T}$ such that $s= w\cdot\mathrm{d}_{\mu,\Theta}^{\frac{1}{p}}\cdot T_\Theta$.
\end{df}

Spatial partial isometries were introduced by Phillips in \cite[Definition 6.4]{Phi_Cuntz_2012} and are relevant in the context of the Banach-Lamperti theorem. For a localizable measure space $(\Omega,\mathcal{A},\mu)$ and $p\in [1,\infty)\setminus \{2\}$, the MP-version of said theorem in \cite[Theorem 2.28]{BarKwaMcK_Banach_2023} states that an operator on $L^p(\mu)$ is an MP-partial isometry if and only if it is a spatial partial isometry.

The following is essentially a restatement, in a slightly different
language, of \cite[Remark 5.2]{BarKwaMcK_Banach_2023}, and we include a proof 
for the convenience of the reader. 

\begin{prop}\label{lambda indicator as SPI}
    Let $(\mathcal{G},\Sigma)$ be a twist over an \'etale groupoid with compact unit space $X$ and let $p\in [1,\infty)$. By \autoref{Fp def}, depending on a choice of lifts, we identify $\ell^p(\mathcal{G}x)$ with $\ell^p(\widehat{\mathcal{G}x})$ and obtain a unital isometric homomorphism $\lambda\colon F_\lambda^p(\mathcal{G},\Sigma)\rightarrow\mathcal{B}(\oplus_{x\in X} \ell^p(\mathcal{G}x))$. Let $S\in \mathcal{B}(\mathcal{G},\Sigma)$. Then the left convolution operator $\Lambda(\ind_S)\in \mathcal{B}(\oplus_{x\in X} \ell^p(\mathcal{G}x))$ given by $\Lambda(\ind_S)(\xi):= (\ind_S*\xi)\vert_{\widehat{\mathcal{G}x}}$ for all $\xi\in \ell^p(\mathcal{G}x)$ is a spatial partial isometry.
\end{prop}

\begin{proof}
    Since Boolean algebra automorphisms preserve atoms, partial set automorphisms on $\mathcal{G}x$ map singletons to singletons. The associated Radon-Nikodym derivative is thus equal to 1. Left multiplication by the set $S\subseteq \mathcal{G}$ induces a partial set automorphism $\widetilde{\alpha}_S\colon \mathsf{d}(S)\mathcal{G} x \rightarrow \mathsf{r}(S)\mathcal{G}x$. Let $\gamma\in \mathsf{r}(S)\mathcal{G}x$. Using the convention in \autoref{set convention}, we observe that $\widehat{\gamma}^{-1}(\mathsf{r}(\gamma)1_S)\widehat{S^{-1}\gamma}\in \T x$. Set \[w_{S,x}(\gamma):= \begin{cases} \widehat{\gamma}^{-1}(\mathsf{r}(\gamma)1_S)\widehat{S^{-1}\gamma},\,\,\,\textup{if}\,\,\,\gamma\in \mathsf{r}(S)\mathcal{G} x,\\ 1,\,\,\,\textup{if}\,\,\,\gamma \in (X\setminus \mathsf{r}(S))\mathcal{G}x.
    \end{cases}\] Then $w_{S,x}\colon \mathcal{G}x\rightarrow \T$ is measurable. The associated spatial partial isometry $w_ST_{\widetilde{\alpha}_S}\in \mathcal{B}(\oplus_{x\in X} \ell^p(\mathcal{G}x))$ acts on $\xi \in \ell^p(\mathcal{G}x)$ by
    \[ w_ST_{\widetilde{\alpha}_S}(\xi)(\gamma) = \begin{cases} w_{S,x}(\gamma)\cdot\xi(\widetilde{\alpha}_S^{-1}(\gamma)),\,\,\,\textup{if}\,\,\,\gamma\in \mathsf{r}(S)\mathcal{G} x,\\ 0,\,\,\,\textup{otherwise}.\end{cases}
    \]
    Likewise, with the notational convention for $\xi \in \ell^p(\mathcal{G}x)$ as in \autoref{Fp def}, we compute for $\gamma\in \mathcal{G}x$ that \[(\ind_S*\xi)(\widehat{\gamma}) = \xi(1_S^{-1}\widehat{\gamma}) = w_{S,x}(\gamma)\xi \left((1_S^{-1}\widehat{\gamma})w_{S,x}(\gamma)\right) = w_{S,x}(\gamma)\xi \big(\widehat{S^{-1}\gamma}\big).\] This shows that $\Lambda(\ind_S)= w_sT_{\widetilde{\alpha}_S}$, as desired. \qedhere
\end{proof}


\begin{prop}\label{Ph extension}
    Let $A$ be a unital Banach algebra and let $C(X)\subseteq A$ be a unital abelian C*-subalgebra. Let $(B,B_*)$ be a unital dual Banach algebra and let $\ph\colon A\rightarrow B$ be a unital contractive homomorphism. For $n\in \mathcal{SN}_X(A)$, let $n=uh$ be a product representation as in \autoref{normalizer def}, let $U=\textup{dom}(\underline{\alpha}_n)$, and let $(e_i)_{i\in I}$ be an approximate identity of $C_0(U)$. Then $\Ph(u):= \lim^{B_\ast}_i \ph(ue_i)$ defines an MP-partial isometry in $B$ such that, for all $f\in C_0(U)$, we have $\ph(uf) = \Ph(u)\ph(f)$.
\end{prop}

\begin{proof}
Since $(\ph(e_i))_i$ is uniformly bounded by 1, we can choose a weak-$*$ accumulation point $e\in \overline{\ph(C_0(U))}^{B_*}\subseteq B$ and a subnet $(e_{i(j)})_j$ such that $e = \lim^{B_\ast}_j \ph(e_{i(j)})$.

\textbf{Claim:} \emph{For all $b\in\overline{\ph(C_0(U))}^{B_*}$, we have $eb=be = b$.}
By separate weak-$*$ continuity of multiplication in $B$ according to \autoref{weak* defs}, we obtain 
\[\ph(f)= B_\ast\mbox{-}\lim_j \ph(fe_{i(j)}) = \ph(f) B_\ast\mbox{-}\lim_j \ph(e_{i(j)})=\varphi(f)e\] 
for all $f\in C_0(U)$. Thus, for all weak-$*$ limit points $b\in\overline{\ph(C_0(U))}^{B_*}$ and for any approximating net such that $b = \lim^{B_\ast}_\lambda \ph(f_\lambda)$, we get \[be = B_\ast\mbox{-}\lim_\lambda \ph(f_\lambda)e = B_\ast\mbox{-}\lim_\lambda \ph(f_\lambda).\] 
An analogous argument shows that $eb=be = b$, proving the claim. 
\vspace{.1cm}

In particular, it follows from the claim that if $e'$ is any other weak-$*$ accumulation point of $(\ph(e_i))_i$, then $e=ee'=e'$. Since every $\ph(e_i)$ is hermitian by \autoref{hermitian preservation}, also $e$ has real numerical range and we deduce that $(\varphi(e_i))_i$ converges weak-$\ast$ to the hermitian idempotent $e$.

\textbf{Claim:} \emph{If $(a_kf_k)_k$ is a norm-bounded net in $AC_0(U)$ such that $\|a_kf_kf-uf\|\to 0$ for all $f\in C_0(U)$, then $\lim^{B_\ast}_k \ph(a_kf_k)$ exists and agrees with the limit $\Ph(u)$ in the statement.} Note that $(\ph(a_kf_k))_k$ is norm-bounded in $B$. Choose a weak-$*$ convergent subnet $(\ph(a_{k(j)}f_{k(j)}))_j$ and set the limit as $\widetilde{\Ph}(u)$. Note that, for all $f\in C_0(U)$, we have $uhf\in A$ and since $h$ is strictly positive on $U$, we deduce that $uf\in A$. We compute \[\ph(uf) = B_\ast\mbox{-}\lim_j \ph(a_{k(j)}f_{k(j)}f) = \widetilde{\Ph}(u)\ph(f).\] Applying this identity to $f=e_i$ at the first step, the first claim at the second step, and choosing $b= \ph(f_{k(j)})$ at the final step, we obtain 
\begin{align*}\Ph(u)&=B_\ast\mbox{-}\lim_i \ph(ue_i) = B_\ast\mbox{-}\lim_i \Ph(u)\ph(e_i) \\&= \Ph(u)e = B_\ast\mbox{-}\lim_j \ph(a_{k(j)})\ph(f_{k(j)})e = \widetilde{\Ph}(u).\end{align*}

\vspace{.1cm}

It is left to show that $\Ph(u)$ is an MP-partial isometry in $B$. Indeed, since $\ph$ is contractive we have $\|\Ph(u)\|, \|\Ph(u^\dagger)\|\leq 1$. Furthermore, $\Ph(u^\dagger)\Ph(u)=e$ is hermitian and, by symmetry, so is $\Ph(u)\Ph(u^\dagger)$. We compute $\Ph(u)e= \Ph(u)$ and $e\Ph(u^\dagger)=\Ph(u^\dagger)$. Thus $\Ph(u)$ is an MP-partial isometry with $\Ph(u)^\dagger = \Ph(u^\dagger)$.\qedhere
\end{proof}
For a given measure space $(\Omega,\mathcal{A},\mu)$ and $q\in [1,\infty)$, \autoref{Ph extension} applies to $B=\mathcal{B}(L^q(\mu))$ as long as $L^q(\mu)$ is reflexive by \autoref{refl E dual Banach}, that is, $q\neq 1$. In conjunction with the Banach-Lamperti theorem, which applies in $B$ as long as $q\neq 2$, \autoref{Ph extension} enables us to show the automatic preservation of spatial normalizers with respect to the core under a wide range of unital contractive homomorphisms. 

\begin{thm}\label{sn preserved}
    Let $p,q\in [1,\infty)$, let $A$ be a unital $L^p$-operator algebra, let $B$ be a unital $L^{q}$-operator algebra, and let $\ph\colon A \rightarrow B$ be a unital contractive homomorphism. Then $\ph(\textup{core}(A))\subseteq \textup{core}(B)$ and $\ph\vert_{\textup{core}(A)}\colon \textup{core}(A)\rightarrow \textup{core}(B)$ is a $*$-homomorphism. Moreover, if $C(X)\subseteq \textup{core}(A)$ is a unital abelian C*-subalgebra, if $q\in (1,\infty)\setminus\{2\}$, and if $C(Y)=\textup{core}(B)$, then $\ph(\mathcal{SN}_{X}(A))\subseteq \mathcal{SN}_{Y}(B)$.
\end{thm}

\begin{proof}
The first claim follows from \autoref{hermitian preservation} since $\ph(A_{\mathrm{h}}) \subseteq B_{\mathrm{h}}$. In the following, we will discuss the inclusion $\ph(\mathcal{SN}_{X}(A))\subseteq \mathcal{SN}_{Y}(B)$ under the additional assumption that $q\not \in \{1,2\}$.\par Let $n\in \mathcal{SN}_{X}(A)$ be a spatial normalizer with product representation $n=uh$ as in \autoref{normalizer def}, and set $U:=\supp(h),\,k:=h\circ \underline{\alpha}_n^{-1}$, and $V:=\supp(k)$. By the first paragraph, the restriction $\ph\vert_{C(X)}\colon C(X)\rightarrow C(Y)$ is a $*$-ho\-mo\-mor\-phism and by Gelfand duality there is a continuous map $\rho\colon Y\rightarrow X$ such that, for all $f\in C(X)$, we have $\ph(f)= f\circ \rho$. Since $B$ is a unital $L^q$-operator algebra, by \cite[Proposition~3.7]{Gar_modern_2021} there are a measure space $(\Omega,\mathcal{A},\mu)$ and a unital isometric homomorphism $\iota \colon B\rightarrow \mathcal{B}(L^q(\mu))$. Without loss of generality, we can assume $\mu$ to be localizable by \cite[Proposition 2.6]{ChoGarThi_rigidity_2021}. By \autoref{hermitian preservation} we have $\iota(B)\cap \mathcal{B}(L^q(\mu))_{\mathrm{h}} = \iota(B_{\mathrm{h}})$. Since $B_{\mathrm{h}} = C(Y,\mathbb{R})$ by assumption, this allows us to treat $\iota$ as an inclusion $B\subseteq\mathcal{B}(L^q(\mu))$ whose restriction to the core gives the inclusion $C(Y)\subseteq L^\infty(\mu)$. Since we assumed $q\neq 1$, the pair $(\mathcal{B}(L^q(\mu)),L^q(\mu)\widehat{\otimes} L^q(\mu)^*)$ is a dual Banach algebra by \autoref{refl E dual Banach}. Hence \autoref{Ph extension} applies to $\ph$, and we obtain an MP-partial isometry $\Ph(u)\in \textup{PI}(\mathcal{B}(L^q(\mu)))$ as in its statement. Since we also assumed that $q\neq 2$, the Banach-Lamperti theorem in \cite[Theorem 2.28]{BarKwaMcK_Banach_2023} applies and $\Ph(u)$ is a spatial partial isometry. Thus, there are a partial set automorphism $\Theta\colon \mathcal{A}_{D_{\Theta^*}}\rightarrow \mathcal{A}_{D_{\Theta}}$ and a measurable function $w\colon D_\Theta \rightarrow \T$ such that 
\[ \Ph(u) = w\left(\mathrm{d}_{\mu,\Theta}\right)^{\frac{1}{q}}T_\Theta.\]
For $g\in L^\infty(\mu)$, we claim that $\Ph(u)g\Ph(u)^\dagger = T_\Theta(g)\in L^\infty(\mu)$. To show this, given $\xi\in L^q(\mu)$, we compute
\begin{align*}
\Ph(u)g\Ph(u)^\dagger(\xi) &= w\left(\mathrm{d}_{\mu,\Theta}\right)^{\frac{1}{q}}T_\Theta \left[ gT_{\Theta^*}(\overline{w}) \left(\mathrm{d}_{\mu,\Theta}\right)^{-\frac{1}{q}}T_{\Theta^*}(\xi)\right] 
\\&= w\left(\mathrm{d}_{\mu,\Theta}\right)^{\frac{1}{q}}T_\Theta \left[ gT_{\Theta^*}\left(\overline{w} \left(\mathrm{d}_{\mu,\Theta}\right)^{-\frac{1}{q}}\xi\right)\right] 
\\&= w\left(\mathrm{d}_{\mu,\Theta}\right)^{\frac{1}{q}}T_\Theta(g)\overline{w} \left(\mathrm{d}_{\mu,\Theta}\right)^{-\frac{1}{q}}T_\Theta(T_{\Theta^*}(\xi))
=T_\Theta(g)\xi.
\end{align*}
The composition operator is multiplicative and defines a $*$-isomorphism \[T_\Theta\colon L^\infty(\mu\vert_{D_{\Theta^*}})\to L^\infty(\mu\vert_{D_{\Theta}})\] with inverse $T_{\Theta^*}$. Note that, by \autoref{Ph extension}, we have $\ph (uh) = \Ph(u)\cdot(h\circ \rho)$ and $\ph(hu^\dagger) = (h\circ \rho)\cdot \Ph(u)^\dagger$. In particular, for all $f\in C(Y)$, we have \[B\ni \ph(uh)f\ph(hu^\dagger) = T_\Theta(f)\cdot T_\Theta(h^2\circ \rho) = T_\Theta(f)\ph(uh)\ph(hu^\dagger)\in L^\infty(\mu\vert_{D_{\Theta}}).\] Since the core $C(Y)\subseteq B$ is the largest C*-subalgebra by \autoref{core}, we have $\ph(uh)f\ph(hu^\dagger)\in B\cap L^\infty(\mu\vert_{D_{\Theta}}) \subseteq C(Y)$. In fact, since \[\supp(T_\Theta(h^2\circ \rho)) = \supp(k^2\circ \rho) = \rho^{-1}(V),\] the composition operator $T_\Theta$ restricts to a $*$-homomorphism $R_\Theta\colon C_0(\rho^{-1}(U))\to C_0(\rho^{-1}(V))$. Reversing the roles of $u$ and $u^\dagger$, and of $T_\Theta$ and $T_{\Theta^*}$, shows that we get an analogous restricted $*$-homomorphism $R_{\Theta^*}\colon C_0(\rho^{-1}(V))\to C_0(\rho^{-1}(U))$ that is inverse to $R_\Theta$. By Gelfand duality, both $*$-isomorphisms are spatially induced and hence $(\ph(uh),\ph(hu^\dagger))$ is an admissible pair for $C(Y)\subseteq B$. Let $\iota_B\colon B\rightarrow B^{**}$ be the canonical inclusion and let $(e_i)_{i\in I}$ be an approximate identity for $C_0(U)$. Since $(B^{**},B^*)$ is a dual Banach algebra by \autoref{refl E dual Banach 2}, we can apply \autoref{Ph extension} again to $\iota_B\circ\ph$ and deduce that $\Ph_B(u):= \lim_i^{B^*}\iota_B(\ph(ue_i))$ is a well-defined MP-partial isometry in $B^{**}$ satisfying $\iota_B(\ph(uf)) = \Ph_B(u)\iota_B(\ph(f))$ for all $f\in C_0(U)$. Now, the realization identities for the pair $(\Ph_B(u)\ph(h)),\Ph_B(u^\dagger)\ph(k))$ show that $\Ph_B(u)$ is $C(Y)$-normalizing as well. Finally, since $\supp(\ph(h)) = \rho^{-1}(U)$, we deduce that $\ph(n) = \Ph_B(u)\ph(h)\in \mathcal{SN}_Y(B)$.
\end{proof}

The following is the main result of this section.
 
\begin{thm}\label{ind psi rho}
    Let $p\in [1,\infty)\setminus \{2\}$ and let $q\in (1,\infty)\setminus \{2\}$. Let $A$ be a unital $L^p$-operator algebra, and let $B$ be a unital $L^q$-operator algebra with saturated core inclusions $C(X)\subseteq A$ and $C(Y)\subseteq B$. Let $\ph\colon A \rightarrow B$ be a unital contractive homomorphism, and let 
$\rho\colon Y\rightarrow X$ be the continuous map with $\ph\vert_{C(X)}=\rho^*$. Then there is a unital semigroup homomorphism $\underline{\ph}\colon \mathcal{SN}_{X,\approx}(A) \rightarrow \mathcal{SN}_{Y,\approx}(B)$ 
with
$\rho \circ \underline{\alpha}_{\ph(n)} = \underline{\alpha}_{n} \circ \rho$
for all $n\in \mathcal{SN}_{X}(A)$.
\end{thm}

\begin{proof}
    By \autoref{sn preserved}, $\ph$ preserves the core and spatial normalizers. Since $\ph(C_0(U))\subseteq C_0(\rho^{-1}(U))$ for every open set $U\subseteq X$, the assignment $\underline{\ph}\colon \mathcal{SN}_{X,\approx}(A) \rightarrow \mathcal{SN}_{Y,\approx}(B)$ given by $\underline{\ph}((n)_\approx):= (\ph(n))_\approx$, for $n\in \mathcal{SN}_X(A)$, is well-defined. One readily checks, using multiplicativity from $\ph$, that $\underline{\ph}$ is 
    a unital semigroup homomorphism. 
    Moreover, for all $f\in C(X)$ and $y\in \rho^{-1}(U)$, we use the realization identities for $(\ph(uh),\ph(hu^\dagger))\in \textup{Adm}_Y(B)$ in the second line and for $(uh,hu^\dagger)\in \textup{Adm}_X(A)$ in the fourth line to show
    \begin{align*}
        f(\rho(\underline{\alpha}_{\ph(uh)}(y)))\ph(h^2)(y) &= \ph(f)(\underline{\alpha}_{\ph(uh)}(y))\ph(hu^\dagger uh)(y)\\
        &= [\ph(hu^\dagger )\ph(f)\ph(uh)](y)\\
        &= \ph(hu^\dagger fuh)(y)\\
        &= \ph(f\circ \underline{\alpha}_{uh}\cdot h^2)(y)\\
        &= f(\underline{\alpha}_{uh}(\rho(y)))\ph(h^2)(y).
    \end{align*}
    Thus $f(\rho(\underline{\alpha}_{\ph(uh)}(y))) = f(\underline{\alpha}_{uh}(\rho(y)))$ and hence $\rho \circ \underline{\alpha}_{\ph(uh)} = \underline{\alpha}_{uh} \circ \rho$, as desired. 
\end{proof}


\section{Actors and Cartan map characterizations}\label{actor chapter}

We review the concept of actors and their induced homomorphisms, and the characterization of Cartan embeddings following Taylor \cite{Tay_Fellactor_2023} and Li \cite{Li_classifiable_2020}. Using \autoref{ind psi rho}, we arrive at concrete characterizations of unital isometric homomorphisms between reduced Weyl twist algebras; see \autoref{isometric equivalences}.

\begin{df}\label{left action} (See \cite[Definition 4.15]{MeyZhu_actors_2015}.)
    Let $\mathcal{G}$ be a topological groupoid, let $Y$ be a topological space, and recall the fiber product notation from \autoref{fibre prod}. A \textit{groupoid left action} $h\colon \mathcal{G} \curvearrowright Y$ consists of a continuous map $\rho_h\colon Y\rightarrow \mathcal{G}^{(0)}$, called the \textit{anchor map}, and a continuous multiplication $\cdot_h\colon \mathcal{G}\times_{\mathsf{d},\,\rho_h} Y \rightarrow Y$ such that
    \begin{enumerate}[label=(\roman*)]
        \item $\rho_h(\gamma\cdot_h y) = \mathsf{r}(\gamma)$ for all $(\gamma,y)\in \mathcal{G}\times_{\mathsf{d},\,\rho_h} Y$,
        \item $\gamma_1 \cdot_h (\gamma_2 \cdot_h y) = (\gamma_1\gamma_2)\cdot_h y$ for all $(\gamma_2,y)\in \mathcal{G}\times_{\mathsf{d},\,\rho_h} Y$ and $\gamma_1\in \mathcal{G}\mathsf{r}(\gamma_2)$, and
        \item $\rho_h(y)\cdot_h y = y$ for all $y\in Y$.
    \end{enumerate}
\end{df}

\begin{df}\label{actor def}
    Let $\mathcal{G}$ and $\mathcal{H}$ be topological groupoids. An \textit{actor} $h\colon \mathcal{G} \curvearrowright \mathcal{H}$ is a groupoid left action commuting with the canonical right multiplication on $\mathcal{H}$. That is, for all $(\gamma,\eta_1) \in \mathcal{G}\times_{\mathsf{d},\,\rho_h} \mathcal{H}$ and $\eta_2\in \mathsf{d}(\eta_1)\mathcal{H}$, we have
    \begin{enumerate}[label=(\roman*)]
        \item $\rho_h(\eta_1\eta_2) = \rho_h(\eta_1)$,
        \item $\mathsf{d}(\gamma\cdot_h \eta_1) = \mathsf{d}(\eta_1)$, and
        \item $\gamma\cdot_h(\eta_1\eta_2) = (\gamma\cdot_h \eta_1)\eta_2$.
    \end{enumerate}
    Since the action on units determines an actor, we refer to the restriction $\rho:= \rho_h\vert_{\mathcal{H}^{(0)}}\colon \mathcal{H}^{(0)}\rightarrow \mathcal{G}^{(0)}$ as the \textit{anchor map} of $h$. We say that an actor $h\colon \mathcal{G} \curvearrowright \mathcal{H}$ is \textit{free} at a unit $y\in \mathcal{H}^{(0)}$ if $\gamma \cdot_h y = y$ implies $\gamma = \rho(y)$. We say that $h$ is \textit{free} if it is free at all units.
\end{df}

The following basic example of an actor will be used
repeatedly.

\begin{eg}\label{trafo actors}
    Let $\beta^{X}\colon \mathcal{S}\curvearrowright X$ and $\beta^{Y}\colon \mathcal{T}\curvearrowright Y$ be actions of unital inverse semigroups on locally compact Hausdorff spaces. Let $\psi\colon \mathcal{S} \rightarrow \mathcal{T}$ be a unital semigroup homomorphism and let $\rho\colon Y \rightarrow X$ be a continuous map satisfying 
    \[\textup{dom}(\beta_{\psi(S)}^Y) = \rho^{-1}(\textup{dom}(\beta_S^X))\andeqn \rho\circ \beta^Y_{\psi(S)} = \beta^{X}_S \circ \rho\]
for all $S\in \mathcal{S}$. 
By \cite[Example 3.6]{Tay_Fellactor_2023}, any such pair $(\psi,\rho)$ defines an actor \[h_{(\psi,\rho)}\colon \mathcal{S}\ltimes_{\beta^X} X \curvearrowright \mathcal{T}\ltimes_{\beta^Y} Y\] between the transformation groupoids in the following way: Using the homeomorphisms $X\cong (\mathcal{S}\ltimes_{\beta^X} X)^{(0)}$ and $Y\cong (\mathcal{T}\ltimes_{\beta^Y} Y)^{(0)}$, we take $\rho$ as the anchor map. Hence $[S,x]\in \mathcal{S}\ltimes_{\beta^X} X$ and $[T,y]\in \mathcal{T}\ltimes_{\beta^Y} Y$ are composable if and only if $x =\rho(\beta_T^Y(y))$, and in this case we set \[[S,x]\cdot_{h_{(\psi,\rho)}} [T,y] := [\psi(S)T,y].\]
\end{eg}

\begin{df}\label{intermediate diag def}
    Let $h_{(\psi,\rho)}\colon \mathcal{S}\ltimes_{\beta^X} X \curvearrowright \mathcal{T}\ltimes_{\beta^Y} Y$ be an actor as in \autoref{trafo actors}. For $s\in \mathcal{S}$ and $y\in \rho^{-1}(\textup{dom}(\beta^X_s))$, we define $\kappa_\rho([s,y]):=[s,\rho(y)]$  and $\kappa_\psi([s,y]):= [\psi(s),y]$. We refer to the resulting diagram \[\xymatrix{\mathcal{S}\ltimes_{\beta^X} X& \mathcal{S}\ltimes_{\beta^Y \circ \psi} Y \ar[l]_-{\ \ \kappa_\rho} \ar[r]^-{\kappa_\psi \ } &\mathcal{T}\ltimes_{\beta^Y} Y}\] as the \textit{actor diagram} associated to $h_{(\psi,\rho)}$.
\end{df}

\begin{rem}\label{kappa props}
    Note that $\kappa_\rho$ is a well-defined groupoid homomorphism because $\rho$ is continuous and equivariant, while $\kappa_\psi$ is well-defined because $\psi$ is a unital semigroup homomorphism. The homomorphism $\kappa_\rho$ is fiberwise injective, while $\kappa_\psi$ has open image and is the identity on the unit space.
\end{rem}
In the following, we reproduce the proof that all actors between \'etale groupoids are particular cases of the class considered in \autoref{trafo actors}.

\begin{prop}\label{image bisection}
    Let $\mathcal{G}$ and $\mathcal{H}$ be \'etale groupoids and let $h\colon \mathcal{G} \curvearrowright \mathcal{H}$ be an actor. Then $\cdot_h$ is open. Moreover, for every open bisection $S\in \mathcal{B}(\mathcal{G})$, the product $S\cdot_h \mathcal{H}^{(0)}$ is an open bisection satisfying $\rho\circ \alpha_{S\cdot_h \mathcal{H}^{(0)}} = \alpha_S \circ \rho$.
\end{prop}

\begin{proof}
    Combine \cite[Lemma 2.5]{Tay_Fellactor_2023} and \cite[Proposition 3.5]{Tay_Fellactor_2023}.
\end{proof}

\begin{cor}\label{psi rho}
    Let $\mathcal{G}$ and $\mathcal{H}$ be \'etale groupoids. Let $\mathcal{S}\subseteq \mathcal{B}(\mathcal{G})$ and $\mathcal{T}\subseteq \mathcal{B}(\mathcal{H})$ be unital wide inverse subsemigroups and use \autoref{etale as trafo grpd} to model both groupoids as transformation groupoids $\mathcal{G} = \mathcal{S}\ltimes_\alpha \mathcal{G}^{(0)}$ and $\mathcal{H} = \mathcal{T}\ltimes_\alpha \mathcal{H}^{(0)}$. Then there is a canonical one-to-one correspondence between actors $h\colon \mathcal{G}\curvearrowright \mathcal{H}$ and pairs $(\psi,\rho)$ consisting of a unital semigroup homomorphism $\psi\colon \mathcal{S} \rightarrow \mathcal{T}$ and a continuous map $\rho\colon \mathcal{H}^{(0)} \rightarrow \mathcal{G}^{(0)}$ satisfying 
    \[\mathsf{d}(\psi(S)) = \rho^{-1}(\mathsf{d}(S))\andeqn \rho\circ \alpha_{\psi(S)} = \alpha_S \circ \rho \,\,\,\textup{for all}\,\,\,S\in \mathcal{S}.\]
    Given $h\colon \mathcal{G}\curvearrowright \mathcal{H}$, the associated pair $(\psi_h,\rho_h)$ consists of $\psi_h(S):= S\cdot_h \mathcal{H}^{(0)}$ for $S\in \mathcal{S}$ and the anchor map. The assignment $h\mapsto (\psi_h,\rho_h)$ and the construction $(\psi,\rho)\mapsto h_{(\psi,\rho)}$ in \autoref{trafo actors} are mutual inverses.
\end{cor}

\begin{proof}
    By \autoref{image bisection}, for any actor $h\colon \mathcal{G}\curvearrowright \mathcal{H}$, the pair $(\psi_h,\rho_h)$ satisfies the desired properties. Conversely, given a pair $(\psi,\rho)$ as in the statement, we obtain an actor $h_{(\psi,\rho)}\colon \mathcal{S}\ltimes_\alpha \mathcal{G}^{(0)}\curvearrowright \mathcal{T}\ltimes_\alpha \mathcal{H}^{(0)}$ by \autoref{trafo actors}. For $y\in \mathcal{H}^{(0)}$ and $\gamma \in \mathcal{G}\rho(y)$, let $S\in \mathcal{S}$ contain $\gamma$. Under the isomorphism in \autoref{etale as trafo grpd}, the multiplication is concretely given by 
    \[\gamma \cdot_{h_{(\psi,\rho)}} y = S\rho(y)\cdot_{h_{(\psi,\rho)}} y = \psi(S)y.\] Hence $h_{(\psi_h,\rho_h)} = h$ and $\psi_{h_{(\psi,\rho)}} = \psi$. This shows the claimed correspondence.\qedhere
\end{proof}

\begin{prop}\label{automatic psi inj}
    Let $h_{(\psi,\rho)}\colon \mathcal{G}\curvearrowright \mathcal{H}$ be an actor between \'etale groupoids as in \autoref{psi rho}. If the anchor map $\rho$ is surjective and if $\mathcal{G}$ is effective, then the unital semigroup homomorphism $\psi\colon \mathcal{S} \rightarrow \mathcal{T}$ is injective.
\end{prop}

\begin{proof}
    Let $S_1,S_2\in \mathcal{S}$ such that $\psi(S_1)=\psi(S_2)$. Then \[ \psi(S_1S_1^*) = \psi(S_1S_2^*) = \psi(S_2S_2^*) \andeqn \psi(S_1^*S_1) = \psi(S_2^*S_1) = \psi(S_2^*S_2)\] are in $\textup{Op}(\mathcal{H}^{(0)})= \{V\colon V\subseteq \mathcal{H}^{(0)}\,\,\textup{open}\}$. Let $S\in \psi^{-1}(\textup{Op}(\mathcal{H}^{(0)}))$ and let $\gamma\in S$. By equivariance of $\rho$, we have \[\mathsf{r}(\gamma) = \alpha_S(\mathsf{d}(\gamma)) = \alpha_S(\rho(y))= \rho(\alpha_{\psi(S)}(y)) = \rho(y) = \mathsf{d}(\gamma)\]
    for all $y\in \rho^{-1}(\{\mathsf{d}(\gamma)\})$. Thus $S\subseteq \textup{Iso}(\mathcal{G})$ and since $\mathcal{G}$ is effective, this shows that $S\subseteq\mathcal{G}^{(0)}$. Hence, both $S_1S_2^*$ and $S_2^*S_1$ are in $\textup{Op}(\mathcal{G}^{(0)})$. Since $\rho$ is surjective, the restriction $\psi\vert_{\textup{Op}(\mathcal{G}^{(0)})} = \rho^{-1}\colon \textup{Op}(\mathcal{G}^{(0)}) \rightarrow \textup{Op}(\mathcal{H}^{(0)})$ is injective, and we get 
    \[ S_1S_1^* = S_1S_2^* = S_2S_2^* \andeqn S_1^*S_1 = S_2^*S_1 = S_2^*S_2.\] Therefore, \[ S_1 = (S_1S_1^*)S_1 = S_2S_2^*S_1 = S_2(S_2^*S_2) = S_2,\] which proves that $\psi$ is injective.
\end{proof}

Intuitively speaking, given $y\in \mathcal{H}^{(0)}$ and $\gamma \in \mathcal{G}\rho(y)$, an actor $h\colon \mathcal{G} \curvearrowright \mathcal{H}$ between \'etale groupoids lifts $\gamma$ consistently with the groupoid multiplication along $\rho$ to $\gamma\cdot_h y \in \mathcal{H}y$ or, equivalently, it lifts an open bisection consistently along $\rho$. In this context, the composition with another actor $k\colon \mathcal{H} \curvearrowright \mathcal{K}$ is nothing but the iteration of these lifting processes upwards from $\mathcal{G}$ over $\mathcal{H}$ into $\mathcal{K}$.

\begin{rem}\label{actor intuition}
Freeness at a unit $y\in \mathcal{H}^{(0)}$ means that the representation of $\eta = \gamma \cdot_h y\in \mathcal{G}\cdot_h y$ as a product is unique. Equivalently, that $y\not \in (S\setminus \mathcal{G}^{(0)})\cdot_h \mathcal{H}^{(0)}$ for all $S\in \mathcal{S}$. The actor $h$ is thereby free if and only if, for any $S\in \mathcal{S}$ with $S\cap \mathcal{G}^{(0)} = \emptyset$, we have $\psi_h(S)\cap \mathcal{H}^{(0)} = \emptyset$.
\end{rem}
We turn to actors between twists over \'etale groupoids and, more specifically, between Weyl twists.

\begin{df}\label{base actor}
    Let $(\mathcal{G},\Sigma)$ and $(\mathcal{H},\Omega)$ be twists over \'etale groupoids. Let $\beta \colon \Sigma \curvearrowright \Omega$ be an actor. By identifying unit spaces, we view the anchor map of $\beta$ as a map $\rho\colon \mathcal{H}^{(0)} \rightarrow \mathcal{G}^{(0)}$. We define the induced \textit{base actor} $h\colon \mathcal{G} \curvearrowright \mathcal{H}$ by using the same anchor map and setting $\gamma \cdot_h y := \pi_\mathcal{H}(\widehat{\gamma} \cdot_\beta y)$ for $(\gamma,y)\in \mathcal{G}\times_{\mathsf{d},\,\rho} \mathcal{H}^{(0)}$.
\end{df}

Naturally, we need to show that $\gamma\cdot_h y$ is independent of the choice of the lift $\widehat{\gamma}$.

\begin{lma}
    Let $(\mathcal{G},\Sigma)$ and $(\mathcal{H},\Omega)$ be twists over \'etale groupoids. Let $\beta \colon \Sigma \curvearrowright \Omega$ be an actor. Then the base actor in \autoref{base actor} is well-defined. Moreover, $\beta$ is free at $y\in \mathcal{H}^{(0)}$ if and only if $h$ is free at $y$.
\end{lma}

\begin{proof}
    Since $\cdot_\beta$ is associative, we have $(\zeta\cdot \sigma)\cdot_\beta y = \zeta \cdot (\sigma\cdot_\beta y)$ for all $\zeta\in \T$, all $\sigma\in \Sigma$ and all $y\in  \mathcal{H}^{(0)}$. Hence $\cdot_h$ is well-defined, and $h$ inherits all the actor conditions in \autoref{actor def} from $\beta$.
    
    Suppose that $\beta$ is free at $y\in \mathcal{H}^{(0)}$. If $\gamma \cdot_h y =y$, then for any lift $\widehat{\gamma}$ of $\gamma$ there exists $\zeta\in \T$ such that $\widehat{\gamma} \cdot_\beta y = \zeta y$. By freeness, we obtain $\widehat{\gamma} = \zeta \rho(y)$ and thus $\gamma = \rho(y)$. Conversely, suppose that $h$ is free at $y\in \mathcal{H}^{(0)}$. If $\sigma \cdot_\beta y =y$, then $\sigma_\bullet \cdot_h y =y$. By freeness, we first obtain $\zeta\in \T$ such that $\sigma= \zeta\rho(y)$ and then $\zeta = 1$ by $\T$-equivariance. This shows the desired equivalence.
\end{proof}


Weyl twists, in turn, are transformation groupoids of spatial normalizer actions with the $\T$-sensitive topology by \autoref{Normalizer twist rec}. With this actor theory in place, we can reinterpret \autoref{ind psi rho} for reduced Weyl twist algebras, and thereby the content of \autoref{core pres chapter} in the following way.

\begin{thm}\label{ph induces actor}
    Let $(\mathcal{G},\Sigma)$ and $(\mathcal{H},\Omega)$ be Weyl twists with compact unit spaces $X$ and $Y$. Let $p\in [1,\infty)\setminus \{2\}$, let $q\in (1,\infty)\setminus \{2\}$, and let $\ph\colon F_\lambda^p(\mathcal{G},\Sigma) \rightarrow F_\lambda^q(\mathcal{H},\Omega)$ be a unital contractive homomorphism. Then $\ph$ induces an actor $\beta_{(\underline{\ph},\rho)}\colon \Sigma \curvearrowright \Omega$ with base actor given by $h_{(\psi,\rho)}$, and we obtain a Weyl twist \[\supp\colon \mathcal{SN}_\lambda^p(\mathcal{G},\Sigma)\ltimes_{\underline{\alpha} \circ \underline{\ph}} Y\rightarrow \supp(\mathcal{SN}_\lambda^p(\mathcal{G},\Sigma))\ltimes_{\alpha \circ \psi} Y.\] Moreover, the actor diagram \[ \xymatrix{\Sigma& \mathcal{SN}_\lambda^p(\mathcal{G},\Sigma)\ltimes_{\underline{\alpha} \circ \underline{\ph}} Y \ar[l]_-{\kappa_\rho} \ar[r]^-{\kappa_{\underline{\ph}}} &\Omega}
    \]
    for $\beta_{(\underline{\ph},\rho)}$ as in \autoref{intermediate diag def} consists of twist homomorphisms in the sense of \autoref{twist hom}.
\end{thm}

\begin{proof}
    By \autoref{saturated}, the inclusions $C(X)\subseteq F_\lambda^p(\mathcal{G},\Sigma)$ and $C(Y)\subseteq F_\lambda^q(\mathcal{H},\Omega)$ are saturated. Thus, by \autoref{ind psi rho}, $\ph$ induces a unital semigroup homomorphism $\underline{\varphi}\colon \mathcal{SN}_\lambda^p(\mathcal{G},\Sigma)_\approx \to \mathcal{SN}_\lambda^p(\mathcal{H},\Omega)_\approx$ and a continuous map $\rho\colon Y\to X$ such that $\rho \circ \underline{\alpha}_{\ph(n)} = \underline{\alpha}_{n} \circ \rho$ for all $n\in \mathcal{SN}_\lambda^p(\mathcal{G},\Sigma)$. By \autoref{trafo actors}, for $n\in \mathcal{SN}_{\lambda}^p(\mathcal{G},\Sigma),\, m \in \mathcal{SN}_{\lambda}^q(\mathcal{H},\Omega)$, and $y\in Y$, the multiplication
    \[[n,\rho(\underline{\alpha}_{m}(y))]\cdot_{\beta_{(\underline{\ph},\,\rho)}} [m,y] := [\underline{\ph}((n)_\approx),\underline{\alpha}_{m}(y)]\cdot [m,y] = [\ph(n)m,y]\] is well-defined and 
    \[\beta_{(\underline{\ph},\,\rho)}\colon \mathcal{SN}_\lambda^p(\mathcal{G},\Sigma)\ltimes_{\underline{\alpha}} X \curvearrowright \mathcal{SN}_\lambda^q(\mathcal{H},\Omega)\ltimes_{\underline{\alpha}} Y\] is an actor with respect to the \'etale topologies on the transformation groupoids. We continue to show that the multiplication \[\cdot_{\beta_{(\underline{\ph},\,\rho)}}\colon (\mathcal{SN}_\lambda^p(\mathcal{G},\Sigma)\ltimes_{\underline{\alpha}} X)\times_{\mathsf{d},\,\rho} Y\rightarrow \mathcal{SN}_\lambda^q(\mathcal{H},\Omega)\ltimes_{\underline{\alpha}} Y\] is continuous with respect to the $\T$-sensitive topologies on the transformation groupoids as well. By linearity of $\ph$, the preimage of $\{([\zeta m,y]\colon \zeta\in Z,\,y\in U\}$ for given $m\in \mathcal{SN}_\lambda^q(\mathcal{H},\Omega)$, $Z\subseteq \T$ open, and $U\subseteq \textup{dom}(\underline{\alpha}_m)$ open as in \autoref{sn action} is
    \[\bigcup_{n\in \ph^{-1}(m)}\{[\zeta n,x]\colon \zeta\in Z,\,x\in \textup{dom}(\underline{\alpha}_n)\}\times_{\mathsf{d},\,\rho} U, \] and is thus open in $(\mathcal{SN}_\lambda^p(\mathcal{G},\Sigma)\ltimes_{\underline{\alpha}} X)\times_{\mathsf{d},\,\rho} Y$. As a result, $\beta_{(\underline{\ph},\,\rho)}$ is an actor between the transformation groupoids with respect to the $\T$-sensitive topologies. By \autoref{Normalizer twist rec}, we can identify these Weyl twists with $(\mathcal{G},\Sigma)$ and $(\mathcal{H},\Omega)$ under their respective Weyl twist isomorphisms in \eqref{weyl rec eq}. This finishes the construction of the actor $\beta_{(\underline{\ph},\,\rho)}\colon\Sigma \curvearrowright \Omega$. 
    
    Since the base groupoids are effective, spatial normalizers are supported on open bisections by \autoref{saturated}. By \autoref{eff realizable have bis support}, we further obtain a unital semigroup homomorphism $\psi\colon \supp(\mathcal{SN}_{\lambda,\approx}^p(\mathcal{G},\Sigma)) \rightarrow \supp(\mathcal{SN}_{\lambda,\approx}^q(\mathcal{H},\Omega))$ with $\alpha\circ \psi \circ \supp= \underline{\alpha} \circ \underline{\ph}$. Since $\alpha$ is faithful, this identity shows that $\psi$ defines the base actor $h_{(\psi,\rho)}\colon \mathcal{G} \curvearrowright \mathcal{H}$ of $\beta$ as in \autoref{base actor}.\par We equip the groupoid $\mathcal{SN}_\lambda^p(\mathcal{G},\Sigma)\ltimes_{\underline{\alpha} \circ \underline{\ph}} Y$ as in \autoref{intermediate diag def} with the $\T$-sensitive topology analogously to \autoref{Weyl twist lp}. Taking strict supports is a unital semigroup homomorphism on $\mathcal{SN}_{\lambda,\approx}^p(\mathcal{G},\Sigma)$ and induces a surjective groupoid homomorphism that we still denote by $\supp$. Note that the homomorphism $\supp$ is continuous in the $\T$-sensitive topology since the strict support of a function is invariant under multiplication by constant $\T$-valued functions. Following the proof of \autoref{Weyl twist lp prop}, this homomorphism $\supp$ defines a Weyl twist. The actor diagram for $\beta$ is as in \autoref{intermediate diag def}. Note that $\kappa_\rho$ is continuous with respect to the Weyl twist topologies because it does not depend on the semigroup component. Similarly, for $\Omega$, basic open sets are of the form $[Z\cdot (m)_\approx, V]$ for a representative $m\in \mathcal{SN}_\lambda^q(\mathcal{H},\Omega)$, $Z\subseteq \T$ open, and $V\subseteq \textup{dom}(\underline{\alpha}_m)$ open. The preimage under $\kappa_{\underline{\ph}}$ of such a set is open, since it equals the union over all $[Z\cdot (n)_\approx,V]$ for representatives $n\in \mathcal{SN}_\lambda^p(\mathcal{G},\Sigma)$ with $\ph(n) =m$. Hence $\kappa_{\underline{\ph}}$ is continuous with respect to the Weyl twist topologies as well. Note that both $\kappa_\rho$ and $\kappa_{\underline{\ph}}$ are $\T$-equivariant since $\T\times X\subseteq\Sigma$ is central and $\ph$ is linear, respectively. This finishes the proof.
\end{proof}
In the special case of $p=q$, we aim to reconstruct $\ph$ out of the diagram in \autoref{ph induces actor}. On the level of the section algebras of the involved twists, any actor canonically induces a $*$-homomorphism, as we show next.

\begin{thm}\label{Fell actor hom}
    Let $(\mathcal{G},\Sigma)$ and $(\mathcal{H},\Omega)$ be twists over \'etale groupoids with compact unit spaces. Let $\beta \colon \Sigma \curvearrowright \Omega$ be an actor. Then $\beta$ induces a unital $*$-homomorphism $\ph_\beta\colon \mathcal{C}_c(\mathcal{G},\Sigma) \rightarrow \mathcal{C}_c(\mathcal{H},\Omega)$ via 
    \begin{align*}
        \ph_\beta(f)(\omega) := \sum_{\sigma\in \Sigma\colon\,\,\sigma\cdot_\beta  \,\mathsf{d}(\omega) = \omega} f(\sigma)
    \end{align*}
    for $f\in \mathcal{C}_c(\mathcal{G},\Sigma)$ and $\omega \in \Omega$. In particular, if $\supp(f)=:S$ is an open bisection in $\mathcal{G}$, then $\supp(\ph_\beta(f)) = S\cdot_h \mathcal{H}^{(0)}$ and we have $\ph_\beta(f)(\sigma\cdot_\beta y) = f(\sigma)$ for all $(\sigma,y)\in \pi_\mathcal{G}^{-1}(S)\times_{\mathsf{d},\,\rho} \mathcal{H}^{(0)}$. Finally, the assignment $\beta \mapsto \ph_\beta$ is functorial, that is, it preserves the composition of actors.
\end{thm}

\begin{proof}
    The statement is a special case of the Fell actor theory developed in \cite[Section 4]{Tay_Fellactor_2023} by \cite[Proposition 4.25]{Tay_Fellactor_2023}. Since we assumed the unit spaces to be compact, $\beta$ is automatically proper, and we may combine \cite[Lemma 4.9]{Tay_Fellactor_2023} and \cite[Proposition 4.11]{Tay_Fellactor_2023}. Functoriality is a special case of \cite[Proposition 4.20]{Tay_Fellactor_2023}.
\end{proof}

\begin{rem}\label{proper rem}
    Alternatively, there is a more elementary way to see that $\ph_\beta$ in \autoref{Fell actor hom} is a homomorphism by exploiting associativity of convolution products. Note that $\beta$ induces an action $*_\beta$ of $f\in \mathcal{C}_c(\mathcal{G},\Sigma)$ on $\mathcal{C}_c(\mathcal{H},\Omega)$ that turns out to be left convolution with $\ph_\beta(f)$ in the algebra $\mathcal{C}_c(\mathcal{H},\Omega)$. That is, for $\xi\in \mathcal{C}_c(\mathcal{H},\Omega)$ and $\omega \in \Omega$, we have
    \begin{align*}
        [f*_\beta \xi](\omega) 
        &= \sum_{\gamma \in \rho(\mathsf{r}(\omega))\mathcal{G}} f(\widehat{\gamma})\xi(\widehat{\gamma}^{-1}\cdot_\beta \omega) \\
        &= \sum_{\gamma \in \rho(\mathsf{r}(\omega))\mathcal{G}} f(\widehat{\gamma})\xi(\underbrace{\widehat{\gamma}^{-1}\cdot_\beta \mathsf{r}(\omega)}_{\widehat{\tau}^{-1}} \omega)\\
        &= \sum_{\tau\in \mathsf{r}(\omega)\mathcal{H}}\sum_{\substack{ \gamma\in \mathcal{G}\colon \\ \widehat{\gamma}^{-1}\cdot_\beta \,\mathsf{r}(\omega) =\, \widehat{\tau}^{-1}}}f(\widehat{\gamma})\xi(\widehat{\tau}^{-1} \omega) \\
        &= \sum_{\tau\in \mathsf{r}(\omega)\mathcal{H}} \bigg[ \sum_{\gamma \in \mathcal{G}\colon \,\widehat{\gamma}\cdot_\beta \,\mathsf{d}(\tau) =\, \widehat{\tau}}f(\widehat{\gamma}) \bigg] \xi(\widehat{\tau}^{-1} \omega) \\
        &= \sum_{\tau\in \mathsf{r}(\omega)\mathcal{H}} \bigg[ \sum_{\sigma \in \Sigma\colon \sigma\cdot_\beta \mathsf{d}(\tau) = \widehat{\tau}}f(\sigma) \bigg] \xi(\widehat{\tau}^{-1} \omega) \\
        &= [\ph_\beta(f)*\xi](\omega)
    \end{align*}
    independently from the choices for $\widehat{\gamma}$ or $\widehat{\tau}$ along the way.
\end{rem}
We continue to look for freeness criteria for actors between twists. In the Hausdorff case, we show that an actor is free if and only if the induced map intertwines the conditional expectations onto the unit space algebras.

\begin{prop}\label{freeness equiv}
    Let $(\mathcal{G},\Sigma)$ and $(\mathcal{H},\Omega)$ be twists over Hausdorff \'etale groupoids with compact unit spaces $X$ and $Y$, respectively. Let $\beta \colon \Sigma \curvearrowright \Omega$ be an actor. Let $\mathcal{S}_\Sigma$ and $\mathcal{S}_\Omega$ be the unital wide inverse subsemigroups of open bisections as in \autoref{supp inverse sg}, and identify the base groupoids with $\mathcal{S}_\Sigma\ltimes_\alpha X$ and $\mathcal{S}_\Omega\ltimes_\alpha Y$, respectively. Describe the base actor $h_{(\psi,\rho)}$ of $\beta \colon \Sigma \curvearrowright \Omega$ in terms of the unital semigroup homomorphism $\psi\colon\mathcal{S}_\Sigma \to \mathcal{S}_\Omega$ given by $\psi(S):= S\cdot_hY$ for all $S\in \mathcal{S}_\Sigma$ and its anchor $\rho\colon Y\to X$. Then the following are equivalent:
    \begin{enumerate}
        \item $\beta$ is free.
        \item The base actor $h_{(\psi,\rho)}$ is free.
        \item For any $S\in \mathcal{S}_\Sigma$ with $S\cap X = \emptyset$, we have $\psi(S)\cap Y = \emptyset$.
        \item We have $\ph_\beta\circ E_{\mathcal{G}}= E_{\mathcal{H}}\circ \ph_\beta$.
        \item There is a well-defined map $\pi_\beta \colon \Sigma \cdot_\beta Y \rightarrow \Sigma$ given by $\pi_\beta(\sigma\cdot_\beta y)= \sigma$ for all $\sigma\cdot_\beta y\in \Sigma\cdot_\beta Y$.
        \item The map $\psi$ of the base actor $h_{(\psi,\rho)}$ satisfies $\psi(S_1 \cap S_2) = \psi(S_1)\cap \psi(S_2)$ for all $S_1,S_2\in \mathcal{S}_\Sigma$.
        \item For any $S\in \mathcal{S}_\Sigma$, we have $\rho(\psi(S)\cap Y)\subseteq S\cap X$.
    \end{enumerate}
    Therefore, if $(\mathcal{G},\Sigma)$ and $(\mathcal{H},\Omega)$ are Weyl twists, then freeness of $\beta$ is equivalent to $\rho$ preserving the interior of the fixed point sets, that is, that for all $S\in \mathcal{S}_\Sigma$ we have \[\rho(\textup{Fix}(\alpha_{\psi(S)})^\mathrm{o}) \subseteq \textup{Fix}(\alpha_{S})^\mathrm{o}.\] 
\end{prop}

\begin{proof}
    That $(1),(2)$ and $(3)$ are equivalent follows from \autoref{actor intuition} and \autoref{trafo actors}. We continue to show that $(3)$ and $(4)$ are equivalent. Note that, by local compactness, in $(3)$ it suffices to consider those $S\in \mathcal{S}_\Sigma$ with compact closure. Likewise, by linearity, $(4)$ is equivalent to $E_{\mathcal{H}}(\ph_\beta(f)) = \ph_\beta(E_{\mathcal{G}}(f))$ for all $f\in \mathcal{C}_c(\mathcal{G},\Sigma)$ with $\supp(f)\in \mathcal{S}_\Sigma$. Let $f\in \mathcal{C}_c(\mathcal{G},\Sigma)$, set $S:= \supp(f)$, and assume $S\in \mathcal{S}_\Sigma$. If $S\subseteq X$, the desired equality is automatic since $E_{\mathcal{G}}$ acts as the identity on functions supported on the unit spaces. Thus, without loss of generality, we can assume that $S\cap X = \emptyset$. The claim reduces to proving $E_{\mathcal{H}}(\ph_\beta(f)) = 0$, that is, \[\psi(S) \cap Y =\supp(E_{\mathcal{H}}(\ph_\beta(f))) =  \emptyset,\] which is exactly $(3)$.
Condition (5) is equivalent to the statement that $\sigma_1 \cdot_\beta y$ and $\sigma_2 \cdot_\beta y\in \Sigma \cdot_\beta Y$ agree if and only if $\sigma_1 = \sigma_2 \in \Sigma \rho(y)$, that is, $\sigma_1^{-1}\sigma_2 = \rho(y)$. This is the definition of freeness and shows that $(5)$ is equivalent to $(1)$. This establishes the equivalence of the statements from $(1)$ to $(5)$. That $(5)$ implies $(6)$ follows from the identity $\pi_h^{-1}(S) = \psi(S)$ for all $S\in \mathcal{S}_\Sigma$, since taking preimages respects intersections. Furthermore, the identity $\psi(S\cap X)= \rho^{-1}(S\cap X)$ for $S_2:=Y$ shows that $(6)$ implies $(7)$. Finally, $(7)$ directly implies $(3)$ and proves all the equivalences.

Note that the fixed points of the partial homeomorphism $\alpha_S$ are at the units of isotropy elements $\mathsf{d}(\textup{Iso}(S))$. In the case of an effective groupoid, we have \[\textup{Fix}(\alpha_{S})^\mathrm{o} = \mathsf{d}(\textup{Iso}(S))^\mathrm{o} = \mathsf{d}(\textup{Iso}(S)^\mathrm{o}) = \mathsf{d}(S\cap X) = S\cap X.\] This shows that the final claim agrees with $(7)$.\qedhere
\end{proof}

\begin{df}\label{diagram notation}
    Let $(\mathcal{G},\Sigma)$ and $(\mathcal{H},\Omega)$ be twists over \'etale groupoids and let $\beta \colon \Sigma \curvearrowright \Omega$ be a free actor. We define the \emph{intermediate groupoid} associated to $\beta$ to be the image groupoid $\Sigma \cdot_\beta Y \subseteq \Omega$, and define the \textit{$\beta$-diagram} to be \[ \xymatrix{\Sigma & \Sigma \cdot_\beta Y \ar[l]_-{\pi_\beta} \ar[r]^-{\iota_\beta} & \Omega,}\] where $\pi_\beta$ is the fiberwise injective twist homomorphism as in (5) of \autoref{freeness equiv} and $\iota_\beta$ is the unitwise bijective inclusion with open image. The same notation $\mathcal{G} \leftarrow \mathcal{G} \cdot_h Y \rightarrow \mathcal{H}$ applies to the free base actor. We define induced maps between the section algebras \[\pi^*\colon \mathcal{C}_c(\mathcal{G},\Sigma) \rightarrow \mathcal{C}_c(\mathcal{G} \cdot_h Y,\Sigma \cdot_\beta Y) \andeqn \iota_{*}\colon \mathcal{C}_c(\mathcal{G} \cdot_h Y,\Sigma \cdot_\beta Y) \rightarrow \mathcal{C}_c(\mathcal{H},\Omega)\] given by 
\[\pi^*(f) := f\circ \pi_\beta \ \ \mbox{ and } \ \ \iota_{*}(g) := \ind_{\iota_\beta(\Sigma\cdot_\beta Y)}\cdot \big(g\circ \iota_\beta^{-1}\big) \]
 for all $f\in \mathcal{C}_c(\mathcal{G},\Sigma)$ and for all $g\in \mathcal{C}_c(\mathcal{G} \cdot_h Y,\Sigma \cdot_\beta Y)$.
\end{df}

\begin{prop}\label{phi ext to Fplam}
    Let $(\mathcal{G},\Sigma)$ and $(\mathcal{H},\Omega)$ be twists over Hausdorff \'etale groupoids with compact unit spaces $X$ and $Y$, respectively. Let $\beta \colon \Sigma \curvearrowright \Omega$ be a free actor. Then $\ph_\beta = \iota_{*} \circ \pi^*$. Moreover, for any $p\in [1,\infty)$, the homomorphism $\iota_{*}$ is $\|\cdot\|_{\lambda^p}$-isometric, while $\pi^*$ is $\|\cdot\|_{\lambda^p}$-isometric on functions supported on $\rho(Y)\mathcal{G}\rho(Y)$ and zero otherwise. The map $\ph_\beta$ therefore extends to a unital contractive homomorphism \[\ph_\beta = \iota_{*} \circ \pi^*\colon F_\lambda^p(\mathcal{G},\Sigma)\rightarrow F_\lambda^p(\mathcal{G} \cdot_h Y,\Sigma \cdot_\beta Y) \rightarrow F_\lambda^p(\mathcal{H},\Omega).\] In particular, $\ph_\beta$ is isometric if and only if $\rho$ is surjective.
\end{prop}

\begin{proof}
     Let $f\in \mathcal{C}_c(\mathcal{G},\Sigma)$, let $y\in Y$, and let $\sigma\in \Sigma \rho(y)$. Using the definitions, we get \[\ph_\beta(f)(\sigma\cdot_\beta y) = f(\sigma) = \pi^*(f)(\sigma\cdot_\beta y)\] and $\ph_\beta(f)(\eta) = 0$ for $\eta \in \Omega \setminus (\Sigma\cdot_\beta Y)$. Thus $\ph_\beta = \iota_* \circ \pi^*$.\par We have $\mathsf{d}(\sigma) = \rho(y)$ and $\mathsf{r}(\sigma) = \rho(\mathsf{r}(\sigma\cdot_\beta y))$ by equivariance, and hence $\ph_\beta(f)=0$ whenever $\supp(f) \cap \rho(Y)\mathcal{G}\rho(Y) = \emptyset$. We now aim to show that $\ph_\beta$ is $\lambda$-isometric on $f\in \mathcal{C}_c(\mathcal{G},\Sigma)$ with $\supp(f) \subseteq \rho(Y)\mathcal{G}\rho(Y)$. Thus, without loss of generality, we can assume that $\rho$ is surjective. In this case, $\pi_\beta$ is a surjective and fiberwise bijective twist homomorphism that satisfies the conditions in \cite[Lemma 3.2]{BarLi_Cartan_2020}. Precomposition by the bijection $(\pi_{\beta,\,b})\vert_{\mathcal{G}\cdot_h y}\colon \mathcal{G}\cdot_h y \rightarrow \mathcal{G}\rho(y)$ induces an isometric isomorphism $U_\pi\colon \ell^p(\mathcal{G}\rho(y))\to \ell^p(\mathcal{G}\cdot_h y)$. Using this isomorphism, the proof in \cite[Lemma 3.2]{BarLi_Cartan_2020} generalizes verbatim to general integrability parameters $p\in [1,\infty)$. Hence, $\pi^*$ is isometric in the reduced norm if $\rho$ is surjective. Finally, since the inclusion of the open subgroupoid $\Sigma\cdot_\beta Y \subseteq \Omega$ is an injective twist homomorphism, it satisfies the conditions in \cite[Lemma 3.4]{BarLi_Cartan_2020}. An inspection of the proof of \cite[Lemma 3.4]{BarLi_Cartan_2020} once more shows that its statement generalizes to general integrability parameters $p\in [1,\infty)$ with fiber $\ell^p$-spaces in place of Hilbert spaces. Hence, $\iota_*$ is isometric in the reduced norm and thus also the composition $\ph_\beta = \iota_*\circ \pi^*$ if and only if the anchor map $\rho$ is surjective.\qedhere
\end{proof}

\begin{rem}\label{rem eq of categories}
    Let $p\in [1,\infty)$. We have seen in \autoref{freeness equiv} and \autoref{phi ext to Fplam} that a free actor $\beta \colon \Sigma \curvearrowright \Omega$ induces a homomorphism $\ph_\beta \colon F_\lambda^p(\mathcal{G},\Sigma)\rightarrow F_\lambda^p(\mathcal{H},\Omega)$ that preserves bisection supports and satisfies $E_{\mathcal{H}}\circ\ph_\beta = \ph_\beta \circ E_{\mathcal{G}}$. If $\mathcal{G}$ and $\mathcal{H}$ are effective and $p=2$, we therefore obtain that $\ph_\beta$ is a Cartan map as in \cite[Proposition 5.4]{Li_classifiable_2020}. The assignment $(\mathcal{G},\Sigma) \mapsto C_\lambda^*(\mathcal{G},\Sigma)$ on Weyl twists and $\beta \mapsto \ph_\beta$ on actors is functorial. In fact, by restricting to the unital setup in \cite[Corollary 5.9]{Tay_Fellactor_2023}, actors canonically induce all Cartan maps and the functor is an equivalence between the category of Weyl twists with compact unit spaces and free actors as morphisms, and unital Cartan pairs with unital Cartan maps as morphisms.
\end{rem}

\begin{prop}\label{free actor recovers ph}
    Let $(\mathcal{G},\Sigma)$ and $(\mathcal{H},\Omega)$ be Weyl twists with compact unit spaces, and let $p\in (1,\infty)\setminus \{2\}$. Let $\ph\colon F_\lambda^p(\mathcal{G},\Sigma) \rightarrow F_\lambda^p(\mathcal{H},\Omega)$ be a unital contractive homomorphism, let $\beta \colon \Sigma \curvearrowright \Omega$ be its associated actor as in \autoref{ph induces actor}, and let $\ph_\beta$ be as in \autoref{phi ext to Fplam}. If $\beta$ is free, then $\ph= \ph_\beta = \iota_*\circ \pi^*$.
\end{prop}

\begin{proof}
    Since $\beta$ is free, the map $\ph_\beta = \iota_*\circ \pi^*$ is a unital contractive homomorphism by \autoref{phi ext to Fplam}, and it is moreover induced by the $\beta$-diagram in \autoref{diagram notation} with intermediate groupoid $\im(\kappa_{\underline{\ph}})$. Since open bisections in $\mathcal{S}_\Sigma$ cover $\mathcal{G}$ by \autoref{supp inverse sg}, for any $f\in \mathcal{C}_c(\mathcal{G},\Sigma)$, there exist $f_1,\dots,f_m\in \mathcal{C}_c(\mathcal{G},\Sigma)$ with $\supp(f_j)\in \mathcal{S}_\Sigma$ and $f= \sum_{j=1}^m f_j$. Since $\mathcal{C}_c(\mathcal{G},\Sigma)$ is dense in $F_\lambda^p(\mathcal{G},\Sigma)$, by \autoref{bis realizable} it suffices to show that $(\iota_*\circ \pi^*)(n) = \ph(n)$ for $n\in \mathcal{SN}_\lambda^p(\mathcal{G},\Sigma)$. Set $S= \supp(n)$ and observe that supp only depends on base groupoids. Moreover, 
    \begin{align*}
        \supp(\iota_*(\pi^*(n))) &= \supp(\pi^*(n))= \pi_h^{-1}(S) \\&= \{\gamma\cdot_h y\colon \pi_h(\gamma\cdot_h y) \in S\} = \{\gamma\cdot_h y\colon \gamma = S\rho(y)\} \\&= \{S\cdot_h y \colon y\in \rho^{-1}(\mathsf{d}(S))\} = \psi(S) \\ &= \supp(\ph(n)) = \{[\underline{\alpha}_{\ph(n)},y]\colon y\in \rho^{-1}(\mathsf{d}(S))\}.
    \end{align*} By \eqref{weyl rec eq} in \autoref{Normalizer twist rec}, for $y\in \rho^{-1}(\mathsf{d}(S))$ we get
    \[ n([n,\rho(y)]) = \sqrt{j_{n^*n}(\rho(y))} = \sqrt{j_{\ph(n^*n)}(y)} = \ph([\ph(n),y]).\]
Finally, for $\zeta\in \T$, we also get
    \begin{align*}
        \pi^*(n)([\zeta\ph(n),y]) &= \overline{\zeta}\cdot \pi^*(n)([\ph(n),y]) = \overline{\zeta}\cdot n(\pi_\beta([\ph(n),y])) \\&= \overline{\zeta} \cdot n([n,\rho(y)]) = \overline{\zeta} \cdot \ph(n)([\ph(n),y]) \\&= \ph(n)([\zeta\ph(n),y]).
    \end{align*} 
This shows that $\iota_*(\pi^*(n)) = \ph(n)$ for all $n\in \mathcal{SN}_\lambda^p(\mathcal{G},\Sigma)$ and proves the claim.
\end{proof}
Our next result, which is the main theorem of this work, characterizes all unital isometric homomorphisms between Weyl twist algebras which moreover intertwine the canonical conditional expectations.

\begin{thm}\label{isometric equivalences}
    Let $(\mathcal{G},\Sigma)$ and $(\mathcal{H},\Omega)$ be Weyl twists with compact unit spaces $X$ and $Y$. Let $p\in (1,\infty)\setminus \{2\}$ and let $\ph\colon F_\lambda^p(\mathcal{G},\Sigma) \rightarrow F_\lambda^p(\mathcal{H},\Omega)$ be a unital contractive homomorphism. Then the following are equivalent:
    \begin{enumerate}
        \item There are a Weyl twist $(\mathcal{G}\cdot_h Y,\Sigma \cdot_\beta Y)$ and twist homomorphisms 
        \[ \xymatrix{\Sigma & \Sigma \cdot_\beta Y \ar[l]_-{\pi_\beta} \ar[r]^-{\iota_\beta} & \Omega }
        \]
        with surjective and fiberwise bijective $\pi_\beta$ and injective and unitwise bijective $\iota_\beta$ with open image such that $\ph = \iota_*\circ \pi^*$ as in \autoref{diagram notation}.
        \item $\ph$ is isometric and satisfies $\ph\circ E_{\mathcal{G}} = E_{\mathcal{H}} \circ \ph$.
        \item $\ph\vert_{C(X)}$ is injective and we have $\ph\circ E_{\mathcal{G}} = E_{\mathcal{H}} \circ \ph$.
    \end{enumerate}
    Moreover, if $\mathcal{G}$ is principal or if the anchor map $\rho$ is open, then freeness of the induced actor is automatic and $\ph$ is an actor homomorphism as in $(1)$, except that $\pi_\beta$ may not be surjective. In this case, $\ph$ being isometric is equivalent to $\ph\vert_{C(X)}$ being injective. 
\end{thm}

\begin{proof}
    For $(1)$ implies $(2)$, the identity $\ph\circ E_{\mathcal{G}} = E_{\mathcal{H}} \circ \ph$ follows from the equivalence of $(4)$ and $(5)$ in \autoref{freeness equiv}. The map $\iota_*\circ \pi^*$ as in \autoref{phi ext to Fplam} is isometric because surjectivity of $\pi_\beta$ is equivalent to the anchor map $\rho$ being surjective. That $(2)$ implies $(3)$ is trivial. For $(3)$ implies $(1)$, we turn the groupoid diagram in \autoref{ph induces actor} into the desired form. The map $\iota_\beta$ is the inclusion of an open subgroupoid and clearly satisfies the listed properties. Similarly to what was done in the proof that $(4)$ implies $(3)$ in \autoref{freeness equiv}, we will now show that the assumption $\ph\circ E_{\mathcal{G}} = E_{\mathcal{H}} \circ \ph$ ensures freeness of the induced actor $h$. To reach a contradiction, assume that $h$ is not free at $y\in Y$. There is $\gamma\in \mathcal{G}\setminus X$ such that $\gamma\cdot_h y =y\in \psi(S)\cap Y$. By local compactness, we can choose a neighborhood $S\in \mathcal{S}_\Sigma$ of $\gamma$ with compact closure, $S\cap X = \emptyset$, and $\psi(S)\cap Y \neq \emptyset$. Further choose $f\in \mathcal{C}_c(\mathcal{G},\Sigma)$ with $\supp(f) = S$ and $\supp(\ph(f)) = \psi(S)$. We compute \[y\in \supp(\ph(f))\cap Y = \supp(E_{\mathcal{H}}(\ph(f))) = \supp(\ph(E_{\mathcal{G}}(f))) = \emptyset,\] 
    a contraction. Hence, the actor is free, and the map $\pi_\beta$ is well-defined and fiberwise injective. That $\ph\vert_{C(X)}$ is injective is equivalent to the anchor map $\rho$ being surjective and thus $\pi_\beta$ is surjective. Now apply \autoref{free actor recovers ph}. This completes the proof that the statements $(1)$ to $(3)$ are equivalent. \par Finally, let $S\in \mathcal{S}_\Sigma$ and let $y\in \psi(S)\cap Y$. Then $\rho(y) = \rho(\alpha_{\psi(S)}(y)) = \alpha_{S}(\rho(y))$ and this shows that $\rho(\textup{Fix}(\alpha_{\psi(S)})^\mathrm{o}) \subseteq \textup{Fix}(\alpha_{S}) = \mathsf{d}(\textup{Iso}(S))$. We claim that if $\mathcal{G}$ is principal or if $\rho$ is open, then this inclusion implies $(7)$ in \autoref{freeness equiv}. Indeed, if $\mathcal{G}$ is principal, then all isotropy elements have to be units and $\mathsf{d}(\textup{Iso}(S)) = S\cap X$. Likewise, if $\rho$ is open, then $\rho(\textup{Fix}(\alpha_{\psi(S)})^\mathrm{o})$ is in the interior of this set of fixed points, that is, in $S\cap X$. Hence \autoref{free actor recovers ph} applies in both cases and $\ph$ is an actor homomorphism as in $(1)$, except that $\pi_\beta$ may not be surjective. By the final part of \autoref{phi ext to Fplam}, we get that $\ph$ is isometric if and only if $\ph\vert_{C(X)}$ is injective.\qedhere
\end{proof}
    
\begin{rem}\label{rem untwisted intro eq}
    Note that by restricting attention to the untwisted case and by reversing the order of $(1)$ to $(3)$, we obtain Theorem \ref{isometric equivalences intro} in the introduction.
\end{rem}

\section{AF-embeddability}\label{AF chapter}
In this section, we present an application of \autoref{isometric equivalences} by obtaining a strong embedding rigidity result for spatial AF $L^p$-operator algebras and $p\in (1,\infty)\setminus\{2\}$; see \autoref{AF emb}. This establishes a major contrast between Schafhauser's C*-algebraic AF-embedding theorem from \cite{Sch_AFemb_2020} and the $L^p$-case in Corollary \ref{AF rigidity}.

\begin{nota}\label{index notation}
    We write $\N_0$ for the set of natural numbers including zero. Throughout this chapter, $j$ will refer to an indexing variable for direct limit constructions. 
    For $M\in \N$, we set $[M]:=\{1,\dots,M\}$.
\end{nota}

\begin{df}\label{AF def}
    Let $M\in \N$. For each $i\in [M]$, let $N_i\in \N$ and let $X_i$ be compact Hausdorff spaces, viewed as a \textit{trivial groupoids} with $X_i^{(0)}= X_i$. We equip the Cartesian product $[N_i]^2$ with the discrete topology, and refer to the groupoid structure on it given by $(a,b)(b,c) := (a,c)$ and $(a,b)^{-1} := (b,a)$ for $(a,b),(b,c)\in [N_i]^2$ as the \textit{full equivalence relation} on $[N_i]$. We freely identify $([N_i]^2)^{(0)}$ with $[N_i]$. Groupoids of the form $\bigsqcup_{i\in [M]} X_i\times [N_i]^2$ are called \textit{elementary}. A topological groupoid $\mathcal{F}$ is called \textit{approximately finite dimensional}, or \textit{AF} in short, if there is an increasing sequence $\left(\mathcal{F}_j\right)_{j\in \N_0}$of elementary groupoids, say $\mathcal{F}_j = \bigsqcup_{i\in [M^{(j)}]} X_i^{(j)}\times [N_i^{(j)}]^2$, with common unit space $\mathcal{F}^{(0)} \cong \mathcal{F}_j^{(0)} \cong \bigsqcup_{i\in [M^{(j)}]} X_i^{(j)}\times [N_i^{(j)}]$, such that $\mathcal{F}$ is isomorphic to the topological direct limit groupoid 
    \[ \varinjlim \mathcal{F}_j = \bigcup_j \bigsqcup_{i\in [M^{(j)}]} X_i^{(j)}\times [N_i^{(j)}]^2. \]
    See further \cite[Section 3]{GioPutSka_equivalence_2004} and \cite[Section 7.2]{GarLup_lprepr_2017}.
\end{df}


We continue to introduce $L^p$-generalizations of AF \cas.

\begin{nota}\label{spatial matrix}
    For $n\in \N$ and fixed $p\in [1,\infty)$, we write $M_n^p$ for the algebra of $n\times n$ matrices under the Banach algebraic identification with $\mathcal{B}(\ell^p([n]))$. We freely identify $F_\lambda^p([n]^2)$ with $M_n^p$ under the isomorphism that sends $\delta_{(a,b)}$ to the corresponding matrix unit $e_{a,b}$ for all $(a,b)\in [n]^2$.
\end{nota}

\begin{df}\label{AF alg df}
    Let $p\in [1,\infty)$. An $L^p$-operator algebra $A$ is called a \textit{spatial AF $L^p$-operator algebra} if there is a direct system $(A_j,\iota_j)_{j\in \N_0}$ of $L^p$-operator algebras with building blocks $A_j$ that are isometrically isomorphic to direct sums of $M_n^p$-algebras and isometric connecting maps $\iota_j\colon A_j \rightarrow A_{j+1}$ such that $A$ is isometrically isomorphic to the direct limit $\varinjlim (A_j,\iota_j)$.
\end{df}




The following is the main result of this section, and it shows a striking rigidity for $L^p$-operator algebras, with $p\neq 2$: Principal Weyl groupoid $L^p$-operator algebras that embed into a spatial AF-algebra are automatically spatial AF.

\begin{thm}\label{AF emb}
    Let $\mathcal{G}$ be a principal Hausdorff \'etale groupoid with compact unit space. Let $p\in (1,\infty)\setminus \{2\}$ and let $A$ be a unital spatial AF $L^p$-operator algebra. If there exists a unital isometric embedding $F_\lambda^p(\mathcal{G})\hookrightarrow A$, then $\mathcal{G}$ is AF.
\end{thm}

\begin{proof}
By \cite[Section 7.2.2]{GarLup_lprepr_2017}, we can without loss of generality assume that $A= F_\lambda^p(\mathcal{F})$ for some AF-groupoid $\mathcal{F}$. Let $\varphi\colon F^p_\lambda(\mathcal{G})\to A$ be a unital, isometric embedding. 
It is well-known that AF-groupoids are principal. By \autoref{isometric equivalences}, $\ph$ is induced by an intermediate groupoid diagram $\xymatrix{\mathcal{G}&\mathcal{G}\cdot_h \mathcal{F}^{(0)}\ar@{->}_-{\pi}[l]\ar@{->}^-{\iota}[r]&\mathcal{F}}$ as in $(1)$ of \autoref{isometric equivalences}. We will first show that $\mathcal{G}\cdot_h \mathcal{F}^{(0)}$ is AF using the map $\iota$, and then we will use the map $\pi$ to show that $\mathcal{G}$ is AF.

It is easy to see that subgroupoids of elementary groupoids with full unit space are elementary. 
Choosing an increasing sequence of elementary groupoids with
    \[ \mathcal{F} \cong \bigcup_j \bigsqcup_{i\in [M^{(j)}]} X_i^{(j)}\times [N_i^{(j)}]^2, \]
    we deduce that there are $n_k^{(i,j)}\in [N_i^{(j)}]\cup \{0\}$ with $\sum_{k\in [N_i^{(j)}]}n_k^{(i,j)} = N_i^{(j)}$ and
    \[ \mathcal{G}\cdot_h \mathcal{F}^{(0)} \cong \bigcup_j \bigsqcup_{i\in [M^{(j)}]}\bigsqcup_{k\in [N_i^{(j)}]} X_i^{(j)}\times [n_k^{(i,j)}]^2.
    \]
    We conclude that $\mathcal{G}\cdot_h \mathcal{F}^{(0)}$ is AF. Thus, we can without loss of generality assume that $\mathcal{F}=\mathcal{G}\cdot_h \mathcal{F}^{(0)}$.\par For $x,x'\in \mathcal{F}^{(0)}$, the fibers $\pi(\mathcal{F}x)$ and $\pi(\mathcal{F}x')$ are disjoint for $\pi(x)\neq \pi(x')$ and identical otherwise. Since $\pi$ is surjective, we can choose a section $\sigma\colon \mathcal{G}^{(0)}\rightarrow \mathcal{F}^{(0)}$ for the surjective map $\pi\vert_{\mathcal{F}^{(0)}}$. For given $j$, we have \[\im(\sigma)= \bigsqcup_{i\in [M^{(j)}]} {X'}_i^j\times [{N'}_i^{(j)}] \subseteq \bigsqcup_{i\in [M^{(j)}]} X_i^{(j)}\times [N_i^{(j)}].\]
    Recall that groupoid homomorphisms intertwine $\mathsf{d}$ and $\mathsf{r}$. Since the full equivalence relation $[N]^2$ is principal, any outgoing groupoid homomorphism is determined by the image of the diagonal. Thus, quotients are necessarily isomorphic to $[N']^2$ for some $N'\leq N$. Applying this in the final step, we get
    \begin{align*}
        \mathcal{G} &=  \bigsqcup_{\pi(x)\in \pi(\mathcal{F}^{(0)})} \pi(\mathcal{F}x) = \bigcup_j\bigsqcup_{y\in \mathcal{G}^{(0)}} \pi(\mathcal{F}_j\sigma(y)) \\&= \bigcup_j \bigsqcup_{i\in [M^{(j)}]}\pi(\mathcal{F}_j\cdot ({X'}_i^{(j)}\times [{N'}_i^{(j)}])) \cong \bigcup_j \bigsqcup_{i\in [M^{(j)}]}{X'}_i^{(j)}\times [{N'}_i^{(j)}]^2.
    \end{align*} 
We conclude that $\mathcal{G}$ is AF, as desired. 
\end{proof}

For C*-algebras, on the other hand, there are many nontrivial examples of embeddings into an AF-algebra, as Schafhauser's AF-embedding theorem illustrates. A concrete example of a family of AF-embeddable \cas\ whose $L^p$-analogs are not AF-embeddable is given by the irrational $L^p$-noncommutative tori.

\begin{cor}\label{AF rigidity text} 
    Let $p\in (1,\infty)\setminus\{2\}$, let $\theta\in \R \setminus \Q$, and let $r_\theta\in \textup{Homeo}(\T)$ be the homeomorphism given by rotation by $2\pi\theta$. Let $A_\theta^p := C(\T)\rtimes_{r_\theta}^p \Z$ be the associated $L^p$-noncommutative torus. Let $B$ be a unital spatial AF $L^p$-operator algebra. Then there is no unital contractive homomorphism from $A_\theta^p$ into $B$.
\end{cor}

\begin{proof}
The algebra $A_\theta^p$ admits a groupoid model given by the transforma\-tion group\-oid $\Z \ltimes_{r_\theta}\T$; see \cite{ChoGarThi_rigidity_2021}. This groupoid is \'etale with compact unit space $\T$. It is principal since the action is free because of $\theta$ being irrational. However, $\Z \ltimes_{r_\theta}\T$ is not AF since the unit space $\T$ is not totally disconnected. If there were a unital contractive homomorphism $\ph\colon A_\theta^p \rightarrow B$, then it would be injective since $A_\theta^p$ is simple. The final part of \autoref{isometric equivalences} would imply that $\ph$ is isometric and therefore, by \autoref{AF emb}, the groupoid $\Z \ltimes_{r_\theta}\T$ would be AF, which is known not to be the case. This contradiction shows that no such $\varphi$ exists.
\end{proof}

Being a principal groupoid is a crucial assumption in \autoref{AF emb}, as there are $L^p$-operator algebras of non-principal, non-AF groupoids that are AF-embeddable. 

\begin{eg}\label{counterex AF emb}
Let $G$ be a non-trivial finite abelian group, regarded as an \'etale groupoid with trivial unit space. For $p\in [1,\infty)$, the algebra $F_\lambda^p(G)$ is AF-embed\-dable via its Fourier transform $C(\widehat{G})\cong \C^{|G|}$, although $G$ is not an AF-groupoid.
\end{eg}

\section{Induced embeddings of topological full groups}\label{tfg chapter}
In this section, we specialize the discussion of spatial normalizers for a unital $L^p$-operator algebra from \autoref{core pres chapter} to invertible isometries. In general, there may be very few of them. However, the advantage of considering invertible isometries is that we avoid MP-partial isometries in the double dual and thus our rigidity results apply to $L^1$-operator algebras as well. If $\mathcal{G}$ is a Weyl groupoid and if $p\in [1,\infty)\setminus\{2\}$, then $\supp(\mathcal{U}_{\mathcal{G}^{(0)}}(F_\lambda^p(\mathcal{G}))) \cong [\![\mathcal{G}]\!]$ recovers the topological full group from \autoref{BG alpha and tfg nota}. In comparison with \autoref{ind psi rho}, we show in \autoref{tfg emb} that, for $p,q\in [1,\infty)\setminus\{2\}$ and Weyl groupoids $\mathcal{G}$ and $\mathcal{H}$, any unital contractive homomorphism $\ph\colon F_\lambda^p(\mathcal{G})\rightarrow F_\lambda^q(\mathcal{H})$ induces a group homomorphism $\psi_0\colon [\![\mathcal{G}]\!]\rightarrow [\![\mathcal{H}]\!]$ such that the anchor map $\rho\colon \mathcal{H}^{(0)}\rightarrow \mathcal{G}^{(0)}$ is equivariant in a suitable sense.




\begin{df}\label{ub def}
    Let $p\in [1,\infty)$. Let $A$ be a unital $L^p$-operator algebra and let $C(X)\subseteq A$ be a unital abelian C*-subalgebra. We define the group of \textit{invertible isometries} as \[\mathcal{U}(A):= \{u\in A^{-1}\colon \|u\|,\,\|u^{-1}\| \leq 1\}\] and its normal subgroup of \textit{$C(X)$-normalizing invertible isometries} as 
    \[ \mathcal{U}_X(A):= \{u\in \mathcal{U}(A) \colon uC(X)u^{-1} = C(X)\}.\] 
\end{df}

\begin{prop}\label{unitaries as sn}
    Let $p\in [1,\infty)$. Let $A$ be a unital $L^p$-operator algebra and let $C(X)\subseteq A$ be a unital abelian C*-subalgebra. Then $C(X)$-normalizing invertible isometries are spatial normalizers as in \autoref{normalizer def}. On $\mathcal{U}_X(A)$, the $\approx$-equivalence relation as in \autoref{approx def} is trivial. As a result, the spatial normalizer action in \autoref{sn action} restricts to a group homomorphism onto its image \[\underline{\alpha}\vert_{\mathcal{U}_{X}}\colon \mathcal{U}_{X}(A) \rightarrow \underline{\alpha}(\mathcal{U}_{X}(A)).\]
\end{prop}

\begin{proof}
    Let $u\in\mathcal{U}_X(A)$. Then $u$ is an MP-partial isometry with $u^\dagger = u^{-1}$. By assumption, the operator $\textup{Ad}_{u^{-1}}\colon C(X)\rightarrow C(X)$ given by $\textup{Ad}_{u^{-1}}(f):= u^{-1}fu$ for all $f\in C(X)$, is a well-defined isometric isomorphism. By Gelfand duality, there is a homeomorphism $\beta_u\colon X\rightarrow X$ such that $u^{-1}fu = f\circ \beta_u$ for all $f\in C(X)$. Thus $u$ is in $\textup{PI}_X(A^{**})$ with $U=V=X$. Hence, $u= u\ind_X$ is a spatial normalizer as in \autoref{normalizer def} with $\underline{\alpha}_u = \beta_u$. 
    
    Finally, let $u\approx u'$ in $\mathcal{U}_X(A)$. Choose strictly positive functions $f,f'\in C(X)$ with $uf=u'f'$. Since $f'$ is invertible, we have $f'f^{-1}\in C(X)_+ \cap \, \mathcal{U}_X(A)$. Note that $\mathcal{U}(C(X)) = C(X,\T)$ and thus 
    $C(X)_+\cap \ \mathcal{U}_X(A)= \{\ind_X\}$. This shows that $u=u'f'f^{-1} =u'$. The final claim follows by restricting $\underline{\alpha}_\approx$ to $\mathcal{U}_{X}(A)_{\approx} = \mathcal{U}_{X}(A)$.
\end{proof}
Generalizing \cite[Proposition 5.6]{Mat_homology_2012}, we show that the topological full group of a Weyl groupoid can canonically be computed using normalizing invertible isometries.

\begin{prop}\label{tfg}
    Let $\mathcal{G}$ be a Weyl groupoid with compact unit space $X$ and let $p\in [1,\infty)$. Then $\underline{\alpha}\vert_{\mathcal{U}_{X}}\colon \mathcal{U}_{X}(F_\lambda^p(\mathcal{G})) \rightarrow [\![\mathcal{G}]\!]$ is a surjective group homomorphism with kernel $\mathcal{U}(C(X))$. In particular, we have canonical group isomorphisms \[
        \faktor{\mathcal{U}_{X}(F_\lambda^p(\mathcal{G}))}{\mathcal{U}(C(X))} \cong [\![\mathcal{G}]\!] \cong \{S\in \mathcal{B}(\mathcal{G})\colon \mathsf{d}(S)=\mathsf{r}(S) = X\} .\]
\end{prop}

\begin{proof}
    By \autoref{unitaries as sn}, we have $\mathcal{U}_X(F_\lambda^p(\mathcal{G}))\subseteq \mathcal{SN}_\lambda^p(\mathcal{G})$. We are in a special case of \autoref{sn inv sem prop} and thus $\underline{\alpha}$ restricts to a group homomorphism onto $\underline{\alpha}(\mathcal{U}_X(F_\lambda^p(\mathcal{G})))\subseteq \mathcal{R}_X(F_\lambda^p(\mathcal{G}))$. Since $\mathcal{G}$ is effective, the bisection action $\alpha$ is faithful by \autoref{alpha and eff}, showing the second isomorphism. Furthermore, $\underline{\alpha}= \alpha \circ \supp$ by \autoref{eff realizable have bis support}, and thus $\underline{\alpha}\vert_{\mathcal{U}_{X}}$ maps into the topological full group. Surjectivity follows from \autoref{bis realizable}, since for a compact open bisection $S$, the invertible isometry $\ind_S\in \mathcal{U}_X(F_\lambda^p(\mathcal{G}))$ realizes $\alpha_S$. Finally, we have $\ker(\underline{\alpha}\vert_{\mathcal{U}_{X}})=\mathcal{U}(C(X))$ by \autoref{saturated}, establishing the first isomorphism.
\end{proof}
For $p\neq 2$, we observe the following rigidity result.

\begin{prop}\label{ub=u}
    Let $p\in [1,\infty)\setminus\{2\}$ and let $A$ be a unital $L^p$-operator algebra with core $C(X)$. Then the invertible isometries automatically normalize the core, that is, $\mathcal{U}_X(A) = \mathcal{U}(A)$.
\end{prop}

\begin{proof}
    Let $u\in \mathcal{U}(A)$, let $t\in \R$, and let $h\in A_{\mathrm{h}}$. One readily checks that $ue^{ith}u^{-1} = e^{ituhu^{-1}}$. Hence,
    \begin{align*}
        1&= \|u^{-1}ue^{ith}u^{-1}u\| \leq \|ue^{ith}u^{-1}\| 
         = \|e^{ituhu^{-1}}\| \leq \|u\|\|e^{ith}\|\|u^{-1}\| = 1
    \end{align*}
    and thus $uhu^{-1}\in A_{\mathrm{h}}$. Now $C(X)= A_{\mathrm{h}} + iA_{\mathrm{h}}$ implies that conjugation by $u$ preserves the core and because of invertibility we have $C(X)= uC(X)u^{-1}$.
\end{proof}

If $\mathcal{G}$ is a Weyl groupoid with compact unit space $X$ and $p\neq 2$, combining \autoref{tfg} with \autoref{core = CX for p not 2} and \autoref{ub=u}, we get canonical isomorphisms
\[ [\![\mathcal{G}]\!] \cong \faktor{\mathcal{U}_{X}(F_\lambda^p(\mathcal{G}))}{\mathcal{U}(C(X))} \cong \faktor{\mathcal{U}(F_\lambda^p(\mathcal{G}))}{\mathcal{U}(\textup{core}(F_\lambda^p(\mathcal{G})))}.\] The latter allows for a version of \autoref{ind psi rho} including $q=1$ in the codomain.

\begin{thm}\label{tfg emb}
    Let $p,q\in [1,\infty)\setminus \{2\}$ and let $\mathcal{G}$ and $\mathcal{H}$ be Weyl groupoids with compact unit spaces $X$ and $Y$, respectively. Let $\ph\colon F_\lambda^p(\mathcal{G})\rightarrow F_\lambda^q(\mathcal{H})$ be a unital contractive homomorphism and let $\rho\colon Y\rightarrow X$ be the continuous map satisfying $\ph\vert_{C(X)}=\rho^*$. Then $\ph$ induces a group homomorphism $\psi_0\colon [\![\mathcal{G}]\!] \rightarrow [\![\mathcal{H}]\!]$ such that 
    \[\rho\circ\psi_0(\alpha_S) = \alpha_S\circ \rho\] for all $\alpha_S\in [\![\mathcal{G}]\!]$. Moreover, $\psi_0$ is injective if and only if $\ph\vert_{\mathcal{U}_{X}(F_\lambda^p(\mathcal{G}))}^{-1}(C(Y,\T)) = C(X,\T)$. In particular, if $\ph$ is isometric, then $\psi_0$ is injective.
\end{thm}

\begin{proof}
    If $q\neq 1$, the existence of $\psi_0$ follows from \autoref{ind psi rho} as it is the restriction $\psi\vert_{\{S\in \mathcal{S}\colon \mathsf{d}(S)=\mathsf{r}(S)=X\}}$ under the second isomorphism in \autoref{tfg}. In general, if we allow for $q=1$, we first note that $\ph$ preserves invertible isometries. Since $p\neq 2$ and $q\neq 2$, \autoref{ub=u} applies and $\ph$ restricts to a group homomorphism \[\ph\vert_{\mathcal{U}_{X}(F_\lambda^p(\mathcal{G}))}\colon \mathcal{U}_{X}(F_\lambda^p(\mathcal{G})) \rightarrow \mathcal{U}_{Y}(F_\lambda^q(\mathcal{H}))\] such that $\ph(C(X,\T))\subseteq C(Y,\T)$. We obtain an induced homomorphism \[\ph_0\colon \faktor{\mathcal{U}_{X}(F_\lambda^p(\mathcal{G}))}{\mathcal{U}(C(X))}\rightarrow \faktor{\mathcal{U}_{Y}(F_\lambda^q(\mathcal{H}))}{\mathcal{U}(C(Y))}.\] Using the isomorphisms in \autoref{tfg}, for all $S\in \mathcal{B}(\mathcal{G})$ with $\mathsf{d}(S)=\mathsf{r}(S)=X$, the homomorphism $\ph_0$ corresponds to the homomorphism $\psi_0\colon [\![\mathcal{G}]\!] \rightarrow [\![\mathcal{H}]\!]$ given by $\psi_0(\alpha_S):= \alpha_{\supp(\ph(\ind_S))}$. The identity $\rho\circ\psi_0(\alpha_S) = \alpha_S\circ \rho$ follows from the final part of \autoref{ind psi rho}. Finally, $\psi_0$ is injective if and only if $\ph_0$ is. By using the definition of quotient groups, this is equivalent to \[\ph\vert_{\mathcal{U}_{X}(F_\lambda^p(\mathcal{G}))}^{-1}(C(Y,\T)) = C(X,\T),\]
    as claimed. If $\ph$ is isometric, we may apply \autoref{hermitian preservation} to $F_\lambda^p(\mathcal{G})_{\mathrm{h}} = C(X,\R)$ and obtain $\ph^{-1}(C(Y,\R))= C(X,\R)$. This implies that $\ph^{-1}(C(Y))= C(X)$ and thus any element in $\ph\vert_{\mathcal{U}_{X}(F_\lambda^p(\mathcal{G}))}^{-1}(C(Y,\T))$ is a unitary in $C(X)$. Hence $\psi_0$ is injective.
\end{proof}

For non-effective groupoids, \autoref{tfg} no longer applies and \autoref{tfg emb} breaks down for multiple reasons, as we show next.

\begin{eg}\label{Z counterex}
Let $G$ be a discrete abelian group, let $p\neq 2$, and set $A = F_\lambda^p(G)$. Then the Gelfand transform is a unital isometric homomorphism $\ph\colon A\rightarrow C(\widehat{G})$, and we have $\textup{core}(A)=\C\cdot \delta_e$. The invertible isometries of $A$ are given by the $\mathbb{T}$-multiples of the canonical generators $\{\delta_g \colon g\in G\}$. Thus,
\[\faktor{\mathcal{U}_{\{e\}}(A)}{\mathcal{U}(\textup{core}(A))} \cong \faktor{G\times \T}{\T} \cong G \andeqn \faktor{\mathcal{U}_{\widehat{G}}(C(\widehat{G}))}{\mathcal{U}(\textup{core}(C(\widehat{G})))} \cong \{e\}.\]
Note that $\ph$ is isometric, while the induced homomorphism $\ph_0\colon G\rightarrow \{e\}$ is not injective if $G$ is non-trivial. 
\end{eg}

We proceed to introduce a class of ample Weyl groupoids for which the data in \autoref{tfg emb} suffice to apply the actor theory in \autoref{actor chapter}.

\begin{df}\label{tfg gpd df}
An \textit{ample groupoid} is an \'etale groupoid whose topology has a basis of compact open sets. 
For an \'etale groupoid $\mathcal{G}$, we set 
\[ [\![\mathcal{G}]\!]_{\textup{loc}} := \left\{S\in \mathcal{B}(\mathcal{G})\colon\,\,\,
\begin{aligned}
	&\textup{there exist}\,\, T\in \alpha^{-1}([\![\mathcal{G}]\!])\,\,\textup{and}\\ &U\subseteq X\,\,\textup{clopen such that}\,\, S=TU
\end{aligned}\right\}.\]

We say that an ample groupoid $\mathcal{G}$ with compact unit space $X$ satisfies \emph{Condition (W)} if for every $x\in X$, we have $\mathcal{G}x\setminus x\mathcal{G}x \neq \emptyset$. 
\end{df}

Note that an ample Weyl groupoid which is minimal automatically satisfies condition (W). The definition of condition (W) is tailored to the proof of \autoref{tfg wide} and refers to ``wide''; see \autoref{wide}. We also mention, without proof, that the transformation groupoid of a topologically free group action by a discrete group on the Cantor space satisfies condition (W) if and only if there is no global fixed point. 

\begin{rem}\label{aw example}
It is easy to see that products of groupoids with condition (W) again satisfy condition (W).
\end{rem}

\begin{prop}\label{tfg wide}
    Let $\mathcal{G}$ be an ample, \'etale groupoid with condition (W). Then $[\![\mathcal{G}]\!]_{\textup{loc}} \subseteq \mathcal{B}(\mathcal{G})$ is wide. In particular, we have $\mathcal{G}\cong [\![\mathcal{G}]\!]_{\textup{loc}}\ltimes_\alpha \mathcal{G}^{(0)}$.
\end{prop}

\begin{proof}
The set $[\![\mathcal{G}]\!]_{\textup{loc}}$ contains the unit space and is clearly closed under products and inversion. We proceed to show that the bisections of the topological full group $\alpha^{-1}([\![\mathcal{G}]\!])$ cover $\mathcal{G}$. Let $\gamma\in \mathcal{G}$. If $\mathsf{d}(\gamma)\neq \mathsf{r}(\gamma)$, we can separate them by disjoint compact open neighborhoods $U$ and $V$ since $X$ is Hausdorff and totally disconnected. Choose a compact open bisection $S'$ containing $\gamma$. Then $S:= VS'U$ is a compact open bisection containing $\gamma$ with $\mathsf{d}(S)\cap \mathsf{r}(S) = \emptyset$. Set $T_\gamma:= S \cup S^{-1}\cup  X\setminus(\mathsf{d}(S)\cup \mathsf{r}(S))$, which is a bisection neighborhood with full domain and range in $\alpha^{-1}([\![\mathcal{G}]\!])$. If $\mathsf{d}(\gamma)= \mathsf{r}(\gamma)$, by assumption there exists $x\in \mathsf{r}(\mathcal{G}\mathsf{d}(\gamma))\setminus \{\mathsf{d}(\gamma)\}$. Choose $\gamma_2\in x\mathcal{G}\mathsf{d}(\gamma)$ and set $\gamma_1:= \gamma \gamma_2^{-1}.$ Then $\gamma = \gamma_1\gamma_2$ is the product of two non-isotropy elements and $S_{\gamma_1}S_{\gamma_2}\in \alpha^{-1}([\![\mathcal{G}]\!])$ is the desired neighborhood.     
Finally, in order to show that $[\![\mathcal{G}]\!]_{\textup{loc}}$ is wide, it suffices to show that it is closed under intersections. For $T_1U_1, T_2U_2\in [\![\mathcal{G}]\!]_{\textup{loc}}$, we have $T_1U_1\cap T_2U_2 = T_1(\mathsf{d}(T_1\cap T_2)\cap U_1\cap U_2)$. The last assertion follows from \autoref{etale as trafo grpd}.
\end{proof}

For groupoids with property (W), \autoref{tfg emb} results in an induced actor analogous to \autoref{ph induces actor}. In this case, by \autoref{automatic psi inj}, we do not even need to assume that $\ph$ is isometric for $\psi_0$ to be injective.

\begin{thm}\label{tfg actor and inj}
    Let $p,q\in [1,\infty)\setminus \{2\}$ and let $\mathcal{G}$ and $\mathcal{H}$ be ample Weyl groupoids satisfying condition (W) with compact unit spaces $X$ and $Y$, respectively. Let $\ph\colon F_\lambda^p(\mathcal{G})\rightarrow F_\lambda^q(\mathcal{H})$ be a unital contractive homomorphism, let $\rho\colon Y\rightarrow X$ be the continuous map satisfying $\ph\vert_{C(X)}=\rho^*$, and let $\psi_0\colon [\![\mathcal{G}]\!]\rightarrow [\![\mathcal{H}]\!]$ be as in \autoref{tfg emb}. Then, for $TU\in [\![\mathcal{G}]\!]_{\textup{loc}}$, the assignment \[\psi(TU):= (\alpha^{-1}\circ\psi_0\circ\alpha)(T)\cdot\rho^{-1}(U)\] is a well-defined unital semigroup homomorphism $\psi\colon [\![\mathcal{G}]\!]_{\textup{loc}}\rightarrow [\![\mathcal{H}]\!]_{\textup{loc}}$. Moreover, the pair $(\psi,\rho)$ is an actor $h\colon \mathcal{G}\curvearrowright \mathcal{H}$. In particular, $\psi_0$ is injective if $\rho$ is surjective. 
\end{thm}

\begin{proof}
    By the second isomorphism in \autoref{tfg} and by \autoref{tfg emb}, the homomorphism $\ph$ induces a group homomorphism $\widetilde{\psi}_0 := \alpha^{-1}\circ\psi_0\circ\alpha$. We claim that our definition of $\psi$ extends this group homomorphism to a unital semigroup homomorphism by restricting the domain. To see that $\psi$ is well-defined, let $T_1U= T_2U\in [\![\mathcal{G}]\!]_{\textup{loc}}$. Then $\ind_{T_i} = \ind_{T_1U} + \ind_{T_i (X\setminus U)}$ and since taking preimages under $\rho$ preserves disjoint domain sets, we obtain a disjoint union
    \[ \widetilde{\psi}_0(T_i) = \supp(\ph(\ind_{T_1U}) + \ph(\ind_{T_i(X\setminus U)})) = \supp(\ph(\ind_{T_1U})) \cup \supp(\ph(\ind_{T_i(X\setminus U)})).
    \]
    In particular, $\widetilde{\psi}_0(T_1)\rho^{-1}(U) = \widetilde{\psi}_0(T_2)\rho^{-1}(U)$. This shows that $\psi$ is well-defined. To see multiplicativity, for $T_1U_1,\,T_2U_2\in [\![\mathcal{G}]\!]_{\textup{loc}}$, we first observe that $\rho^{-1}(\alpha_{T_2}^{-1}(U_1)) = \alpha_{\psi_0(T_2)}^{-1}(\rho^{-1}(U_1))$, and using this at the third step, we get
    \begin{align*}
        \psi(T_1U_1T_2U_2) &= \widetilde{\psi}_0(T_1T_2)\rho^{-1}(\alpha_{T_2^{-1}}(U_1)\cap U_2)\\
        &= \widetilde{\psi}_0(T_1)\widetilde{\psi}_0(T_2)\rho^{-1}(\alpha_{T_2^{-1}}(U_1))\rho^{-1}(U_2)\\
        &= \widetilde{\psi}_0(T_1)\alpha_{\widetilde{\psi}_0(T_2)}(\rho^{-1}(\alpha_{T_2^{-1}}(U_1)))\widetilde{\psi}_0(T_2)\rho^{-1}(U_2)\\
        &= \widetilde{\psi}_0(T_1)\rho^{-1}(U_1)\widetilde{\psi}_0(T_2)\rho^{-1}(U_2)\\
        &= \psi(T_1U_1)\psi(T_2U_2).
    \end{align*}
    By \autoref{tfg wide} and \autoref{psi rho}, the pair $(\psi,\rho)$ defines an actor $h\colon \mathcal{G}\curvearrowright \mathcal{H}$. Since $\mathcal{G}$ is effective, \autoref{automatic psi inj} implies that $\psi$, and thereby also $\psi_0$, is injective if $\rho$ is surjective.
\end{proof}
The actor in \autoref{tfg actor and inj} allows us to prove a characterization analogous to \autoref{isometric equivalences}.

\begin{thm}
    Let $p\in [1,\infty)\setminus \{2\}$ and let $\mathcal{G}$ and $\mathcal{H}$ be ample Weyl groupoids satisfying condition (W) with unit spaces $X$ and $Y$. Let $\ph\colon F_\lambda^p(\mathcal{G})\rightarrow F_\lambda^p(\mathcal{H})$ be a unital contractive homomorphism. The following are equivalent:
    \begin{enumerate}
        \item There are an intermediate Weyl groupoid $\mathcal{K}$ and twist homomorphisms 
        \[ \xymatrix{\mathcal{G} & \mathcal{K}\ar[l]_-{\pi} \ar[r]^-{\iota} & \mathcal{H}}
        \]
        with surjective and fiberwise bijective $\pi$ and injective and unitwise bijective $\iota$ with open image such that $\ph= \iota_*\circ \pi^*$ as in \autoref{diagram notation}.
        \item $\ph$ is isometric and satisfies $\ph\circ E_{\mathcal{G}} = E_{\mathcal{H}} \circ \ph$.
        \item $\ph\vert_{C(X)}$ is injective and we have $\ph\circ E_{\mathcal{G}} = E_{\mathcal{H}} \circ \ph$.
    \end{enumerate}
    Moreover, if $\mathcal{G}$ is principal or if the anchor map $\rho$ is open, then freeness of the induced actor in \autoref{tfg actor and inj} is automatic and $\ph$ is an actor homomorphism as in $(1)$, except that $\pi$ may not be surjective. In this case, being isometric is equivalent to $\ph\vert_{C(X)}$ being injective. 
\end{thm}

\begin{proof}
    We can rerun the proof of \autoref{isometric equivalences}. For $(1)$ implies $(2)$, the identity $\ph\circ E_{\mathcal{G}} = E_{\mathcal{H}} \circ \ph$ follows from the equivalence of $(4)$ and $(5)$ in \autoref{freeness equiv}. The map $\iota_*\circ \pi^*$ as in \autoref{phi ext to Fplam} is isometric because surjectivity of $\pi_h$ is equivalent to the anchor map $\rho$ being surjective. That $(2)$ implies $(3)$ is trivial. For $(3)$ implies $(1)$, the assumption that $\ph\circ E_{\mathcal{G}} = E_{\mathcal{H}} \circ \ph$ ensures freeness of the induced actor $h$ in \autoref{tfg actor and inj}. By construction, the actor $h$ in \autoref{tfg actor and inj} also satisfies $\ph_h(\ind_S) = \ph(\ind_S)$ for all $S\in \alpha^{-1}([\![\mathcal{G}]\!])$. With this version of \autoref{free actor recovers ph}, we can apply 
    \autoref{phi ext to Fplam} to the associated $h$-diagram.
    This completes the proof that the statements $(1)$ to $(3)$ are equivalent. Finally, the conditions for automatic freeness are exactly the same as in \autoref{isometric equivalences}.\qedhere
\end{proof}

\section{Tensor products of $L^p$-Cuntz algebras}\label{Cuntz chapter}

In this section, we recall the construction of the $L^p$-Cuntz algebras and their tensor products. Moreover, we observe that the embedding theory from \autoref{tfg chapter} applies to their groupoid models, namely products of groupoids associated to SFTs. Finally, we identify the associated topological full groups as generalized Brin-Thompson groups and make use of \autoref{tfg actor and inj} to establish \autoref{cuntz rigidity text} as the main result of this section.

\begin{df}\label{on def}
    Let $n\in \N_{\geq 2}$ and let $p\in [1,\infty)$. Let $S_0,T_0,\dots,S_{n-1},T_{n-1}\in \mathcal{B}(\ell^p(\N_0))$ be the shift operators defined on the canonical basis $(\delta_k)_{k\in \N_0}$ by
    \[
        S_j(\delta_k):= \delta_{nk+j} \andeqn T_j(\delta_k):= \begin{cases} \delta_{l},\,\,\, &\textup{if\quad} k= nl+j, \\ 0,\,\,\,&\textup{otherwise}.\end{cases}
    \]
    The Banach algebra $\mathcal{O}_n^p\subseteq \mathcal{B}(\ell^p(\N_0))$ generated by $S_0,T_0,\dots,S_{n-1},T_{n-1}$ is the \textit{$L^p$-Cuntz algebra} on $n$ generators.
\end{df}

It is easy to see that when $p=2$, then we have $T_j=S_j^*$ for $j=0,\ldots,n-1$, and hence $\mathcal{O}_n^2 = \mathcal{O}_n$ is the usual Cuntz C*-algebra on $n$ generators.


In general, Banach algebras admit several notions of tensor products, which arise as norm completions of the algebraic tensor product in a suitable Banach norm. For $L^p$-operator algebras, we will work with the spatial tensor product $\otimes_p$ as in \cite[Definition 7.1]{DefFlo_norms_1993} and \cite[Theorem 2.16]{Phi_Cuntz_2012}, and we recall it next.

\begin{df}\label{lp tensor}
    Let $p\in [1,\infty)$ and let $(\Omega,\mathcal{A},\mu)$ and $(\Omega',\mathcal{A}',\mu')$ be measure spaces. An element of the algebraic tensor product $L^p(\mu)\odot L^p(\mu')$ can be viewed as a function $\xi\in L^p(\mu,\,L^p(\mu'))$, and the latter becomes a normed space with
    \[ \|\xi\|_p := \left(\int_\Omega \|f(\omega)\|_{p,\mu'}^p\,\textup{d}\mu(\omega)\right)^{\frac{1}{p}}.\]
    The completion in this norm is called the $L^p$-\textit{tensor product} $\otimes_{p}$ of $L^p(\mu)$ and $L^p(\mu')$. 
\end{df}

In the context of \autoref{lp tensor}, there is a canonical isometric isomorphism $L^p(\mu)\otimes_{p} L^p(\mu') \cong L^p(\mu\times\mu')$ by \cite[Proposition 7.2]{DefFlo_norms_1993}. For $a\in \mathcal{B}(L^p(\mu))$ and $b\in \mathcal{B}(L^p(\mu'))$, there is a unique product operator $a\otimes b \in \mathcal{B}(L^p(\mu\times\mu'))$ that acts on elementary tensors by $(a\otimes b)(\xi\otimes \xi') = a(\xi)\otimes b(\xi')$ for all $\xi\otimes \xi'\in L^p(\mu)\otimes_{p} L^p(\mu') \cong L^p(\mu\times\mu')$. For a more detailed discussion of these product operators, we refer the reader to \cite[Section 7]{Phi_Cuntz_2012}. 
    

The $L^p$-tensor product of $L^p$-spaces can be used to define the \textit{spatial tensor product} for $L^p$-operator algebras; see \cite[Definition 7.2]{ChoGarThi_rigidity_2021}.

\begin{df}\label{sp tensor}
    Let $p\in [1,\infty)$, and let $A$ and $B$ be $L^p$-operator algebras with isometric and nondegenerate representations $\iota_A\colon A\rightarrow \mathcal{B}(L^p(\mu))$ and $\iota_B\colon B\rightarrow \mathcal{B}(L^p(\mu'))$. Then, for all $a\in A$ and $b\in B$, we have $\iota_A(a)\otimes \iota_B(b) \in \mathcal{B}(L^p(\mu\times\mu'))$, and we define the \textit{spatial norm} on $A\odot B$ by $\|a\otimes b\| := \|\iota_A(a)\otimes \iota_B(b)\|$. The completion $A\otimes_{p} B$ of $A\odot B$ in this norm is called the \textit{($L^p$-)spatial tensor product}.
\end{df}
In order to study embeddings into $L^p$-Cuntz algebras, we realize $\mathcal{O}_n^p$ and its spatial tensor pro\-ducts as reduced groupoid algebras.

\begin{df}\label{full shift}
    Let $n\in \N_{\geq 2}$ and consider the Cantor set $X_n := \{0,\dots,n-1\}^{\N}$. To lighten the notation, we write elements in $X_n$ as infinite words without separators, and the set of finite words is denoted by $\{0,\dots,n-1\}^*$. We define the \textit{full shift in the alphabet} $\{0,\dots,n-1\}$, also known as the \textit{shift of finite type} (SFT), as the dynamical system given by the shift map $\sigma\colon X_n\rightarrow X_n$ by $\sigma(x_1x_2\cdots):= (x_2x_3\cdots)$ for all $(x_1x_2x_3\cdots)\in X_n$. The groupoid associated to an SFT is given by
    \begin{align*}
        \mathcal{G}_n := \bigg\{(x,k,y)\in X_n\times \Z \times X_n\colon \,\,\begin{aligned} &\textup{there are}\,\, k_+,k_-\in \N \,\,\textup{such that}\\
    &k=k_+-k_- \,\,\textup{and}\,\, \sigma^{k_+}(x)=\sigma^{k_-}(y) \end{aligned}\bigg\}.
    \end{align*} 
    The range and domain maps are the projections onto the first and third coordinate, respectively, and multiplication and inversion for $(x,k,y),(y,l,z)\in \mathcal{G}_n$ are defined as $(x,k,y)(y,l,z) := (x,k+l,z)$ and $(x,k,y)^{-1} := (y,-k,x)$. We equip $\mathcal{G}_n$ with the topology generated by the clopen sets \[Z(\gamma,\delta) := \{(\gamma x, |\gamma|-|\delta|,\delta x)\colon x\in X_n\}\] for given finite words $\gamma,\,\delta \in \{0,\dots,n-1\}^*$.
\end{df}

\begin{prop}\label{full shift adjectives}
    The full shift $\mathcal{G}_n$ in the alphabet $\{0,\dots,n-1\}$ is a minimal amenable Weyl groupoid with compact unit space $X_n \cong \mathcal{G}_n^{(0)}$. As a consequence, products of groupoids associated to SFTs satisfy condition (W) from \autoref{tfg gpd df}.
\end{prop}

\begin{proof}
    It is clear that $X_n \cong \mathcal{G}_n^{(0)}$ is compact and that $\mathcal{G}_n$ is an ample groupoid. Since isotropy elements have periodic domain and since the subset of periodic points in $X_n$ has empty interior, we deduce that $\mathcal{G}_n$ is a Weyl groupoid. Amenability follows from \cite[Example 4.1.12]{Sims_groupoids_2011}. Minimality is a special case of \cite[Lemma 6.1]{Mat_tfg_2015} and therefore \autoref{aw example} applies.
\end{proof}



Matui showed in \cite[Theorem 3.10]{Mat_tfg_2015} that full shifts are characterized by their topological full groups. We describe these topological full groups as generalized versions of Brin-Thompson groups, which are closely related to Thompson's well-studied group $V= V_{2}$. See also \cite[Section 9]{Bri_thompson_2004}.

\begin{df}\label{Brin-Thompson}
    Let $m\in \N$ and let $k_1,\dots k_m\in \N_{\geq 2}$ be alphabet sizes. Using the notation of \autoref{full shift}, a \textit{table} for the Cantor set $X=\prod_{i=1}^m X_{k_i}$ of size $l\in \N$ is a matrix of the form \[\begin{pmatrix} v\\ u\end{pmatrix}= \begin{pmatrix} v^{(1)}& v^{(2)}& \cdots & v^{(l)} \\ u^{(1)}& u^{(2)}& \cdots & u^{(l)}\end{pmatrix}\]
    with entries $v^{(j)},\,u^{(j)}\in \prod_{i=1}^m\{0,\dots,k_i-1\}^*$ for all $j=1,\dots,l$ such that both rows describe a partition of \[X = \bigsqcup_{j=1}^l\prod_{i=1}^m u_i^{(j)}X_{k_i} = \bigsqcup_{j=1}^l\prod_{i=1}^m v_i^{(j)}X_{k_i}.\] That is, for every $y\in X$, there is a unique index $j$ and tail $x\in X$ such that $y=v^{(j)}x$. Every table induces a homeomorphism by replacing the $v$-initial word tuple with the one of $u$, that is,
    \[ \overline{\begin{pmatrix} v\\ u\end{pmatrix}}\left(v^{(j)}x\right) := u^{(j)}x.
    \]
    The composition of two such homeomorphisms is induced by a refined table of possibly greater size and we call the thereby formed group of induced homeomorphisms the \textit{generalized Brin-Thompson group} $V_{k_1,\dots,k_m}$:
    \[ V_{k_1,\dots,k_m} := \Big\{\overline{\begin{pmatrix} v\\ u\end{pmatrix}}\in \textup{Homeo}(X)\colon \begin{pmatrix} v\\ u\end{pmatrix} \textup{\, is a table}\Big\}.
    \]
\end{df}
The name is motivated by the definition of the \textit{Brin-Thompson group} $mV_k$ in the case $k_1=\dots=k_m =k$. Next, we unify \cite[Theorem 5.4]{DicPer_BrinThom_2014} for Leavitt algebras and \cite[Proposition 9.6]{Nek_pimsner_2004} for graph \cas\ by proving that generalized Brin-Thompson groups are the topological full groups of products of groupoids associated to SFTs.

\begin{prop}\label{tfg of shifts}
    Let $k_1,\dots,k_m\in \N_{\geq 2}$. Then $[\![\mathcal{G}_{k_1}\times\dots\times \mathcal{G}_{k_m}]\!] \cong V_{k_1,\dots,k_m}$. In particular, $[\![\mathcal{G}_{k}^m]\!]$ is isomorphic to the Brin-Thompson group $mV_{k}$. 
\end{prop}

\begin{proof}
    Since $\mathcal{G}_{k_1}\times\dots\times \mathcal{G}_{k_m}$ is effective, we may regard elements in its topological full group as compact open bisections with full domain and range. Let $S$ be such an open bisection. By definition of the product topology, we can cover $S$ with products of basic open neighborhoods $\prod_{i=1}^m Z(u_i,v_i)$ for tuples of finite words $u,v\in \prod_{i=1}^m\{0,\dots,k_i-1\}^*$. Since $S$ is compact, there are finitely many $u^{(j)},v^{(j)}\in \prod_{i=1}^m\{0,\dots,k_i-1\}^*$ for $j=1,\dots, l$ such that $S = \bigcup_{j=1}^l\prod_{i=1}^m Z(u_i^{(j)},v_i^{(j)})$. Without loss of generality, this union can be arranged to be disjoint. By assumption, we now have
    \[\prod_{i=1}^m X_{k_i} = \mathsf{d}(S) = \bigsqcup_{j=1}^l\prod_{i=1}^m v_i^{(j)}X_{k_i} = \mathsf{r}(S) = \bigsqcup_{j=1}^l\prod_{i=1}^m u_i^{(j)}X_{k_i}.\]
    That is, the tuples $u^{(j)}$ and $v^{(j)}$ form a table for $\prod_{i=1}^m X_{k_i}$ of size $l$ and the corresponding replacement homeomorphism is $\alpha_S \in V_{k_1,\dots,k_m}$. No refinement of the table changes $S$ and all tables induce an admissible bisection of the form $\bigsqcup_{j=1}^l\prod_{i=1}^m  Z(u_i^{(j)},v_i^{(j)})$. Hence, we have $[\![\mathcal{G}_{k_1}\times\dots\times \mathcal{G}_{k_m}]\!] \cong V_{k_1,\dots,k_m}$.
\end{proof}

Using \autoref{tfg of shifts}, embeddability between tensor products of $L^p$-Cuntz algebras is directly related to embeddability of the groups $[\![\mathcal{G}_{k_1}\times\dots\times \mathcal{G}_{k_m}]\!]$, for which some rigidity results have been established in \cite{Mat_shifts_2016} and \cite{MatBon_rigidity_2018}. In combination with \autoref{tfg actor and inj}, this leads to the main rigidity result of this section, which greatly generalizes \cite[Theorem~C]{ChoGarThi_rigidity_2021}.

\begin{thm}\label{cuntz rigidity text}
    Let $p\in [1,\infty)\setminus \{2\}$, let $m,n\in \N$, and let $k_1,\ldots,k_m,\ell_1,\ldots,
\ell_n\in \N_{\geq 2}$. If there exists a unital contractive homomorphism
    \[ \underbrace{\mathcal{O}_{k_1}^p \otimes_{p} \dots \otimes_{p} \mathcal{O}_{k_m}^p}_m \rightarrow \underbrace{\mathcal{O}_{\ell_1}^p \otimes_{p} \dots \otimes_{p} \mathcal{O}_{\ell_n}^p}_n,\]
then $m\leq n$.
\end{thm}

\begin{proof}
Assume that a unital contractive homomorphism as in the statement exists. Since the algebra $\mathcal{O}_{k_1}^p \otimes_{p} \dots \otimes_{p} \mathcal{O}_{k_m}^p$ is simple by \cite[Theorem 5.14]{Phi_Cuntzsimp_2013}, such a homomorphism is necessarily injective. By \cite[Theorem~7.6]{ChoGarThi_rigidity_2021} and \cite[Theorem~7.7]{GarLup_lprepr_2017}, the involved $L^p$-operator algebras have groupoid models that are concretely given by products of groupoids associated to SFTs. These models meet the conditions for \autoref{tfg actor and inj} because of \autoref{full shift adjectives}. Therefore, by injectivity on the core $C(\prod_{i=1}^m X_{k_i})$ and the final claim of \autoref{tfg actor and inj}, there is a group embedding \[\psi_0\colon [\![\mathcal{G}_{k_1}\times\dots\times \mathcal{G}_{k_m}]\!] \hookrightarrow [\![\mathcal{G}_{\ell_1}\times\dots\times \mathcal{G}_{\ell_n}]\!].\] Assume by contradiction that $m>n$. By \cite[Corollary 11.19]{MatBon_rigidity_2018}, any homomorphism $[\![\mathcal{G}_{k_1}\times\dots\times \mathcal{G}_{k_m}]\!] \rightarrow [\![\mathcal{G}_{\ell_1}\times\dots\times \mathcal{G}_{\ell_n}]\!]$ as above has abelian image. By \autoref{tfg of shifts}, the image $\psi_0([\![\mathcal{G}_{k_1}\times\dots\times \mathcal{G}_{k_m}]\!])$ is isomorphic to the group $V_{k_1,\dots,k_m}$ from \autoref{Brin-Thompson}. This group, however, is not abelian: For example, for any alphabet, there are two noncom\-muting transpositions that solely swap the initial words $00 \leftrightarrow 01$, and $01 \leftrightarrow 10$. This is a contradiction, which proves that $m\leq n$.
\end{proof}

The following rules out any reasonable analog of Kirchberg's celebrated $\mathcal{O}_2$-embedding theorem, which in particular implies that $\mathcal{O}_2\otimes\mathcal{O}_2$ embeds into $\mathcal{O}_2$.

\begin{cor}\label{non emb text}
    Let $p\in [1,\infty)\setminus \{2\}$. Then there is no unital contractive homomorphism from $\mathcal{O}_{2}^p \otimes_{p} \mathcal{O}_{2}^p$ into $\mathcal{O}_{2}^p$.
\end{cor}


\end{document}